\documentclass[a4paper,10pt,leqno]{article}
\usepackage[margin=2.5cm]{geometry}
\usepackage[all,2cell]{xy}
\UseAllTwocells
\usepackage{amsmath}
\usepackage{amssymb}
\usepackage{amsxtra}
\usepackage{amsthm}
\usepackage[pagebackref,breaklinks]{hyperref}
\usepackage[T1]{fontenc}
\usepackage{lmodern}
\usepackage[shortlabels]{enumitem}

\theoremstyle{definition}
\newtheorem{Definition}[subsection]{Definition}
\newtheorem{Notation}[subsection]{Notation}
\newtheorem{Convention}[subsection]{Convention}
\newtheorem{Construction}[subsection]{Construction}

\newtheorem{Remark}[subsection]{Remark}
\theoremstyle{plain}
\newtheorem{Lemma}[subsection]{Lemma}
\newtheorem{Proposition}[subsection]{Proposition}
\newtheorem{Theorem}[subsection]{Theorem}
\newtheorem{Corollary}[subsection]{Corollary}
\newtheorem*{sTheorem}{Theorem}

\DeclareMathOperator{\Spec}{Spec}
\DeclareMathOperator{\Et}{Et}
\DeclareMathOperator{\Aut}{Aut}
\DeclareMathOperator{\Tr}{Tr}
\DeclareMathOperator{\tot}{tot}
\DeclareMathOperator{\Ind}{Ind}
\DeclareMathOperator{\Coker}{Coker}
\DeclareMathOperator{\Ker}{Ker}
\DeclareMathOperator{\Img}{Im}

\newcommand{\BC}{\mathrm{BC}}
\newcommand{\GBC}{\mathrm{GBC}}
\newcommand{\PF}{\mathrm{PF}}
\newcommand{\KF}{\mathrm{KF}}
\newcommand{\ev}{\mathrm{ev}}
\newcommand{\Dbc}{D^b_c}
\newcommand{\ft}{\mathrm{ft}}
\newcommand{\Dcft}{D_{c\ft}}

\newcommand{\adj}{\mathrm{adj}}
\newcommand{\red}{\mathrm{red}}
\newcommand{\norm}{\mathrm{norm}}
\newcommand{\et}{{\mathrm{et}}}

\newcommand{\beq}{\begin{equation}}
\newcommand{\eeq}{\end{equation}}
\newcommand{\simto}[1][]{\xrightarrow[#1]{\sim}}
\newcommand{\xto}[1]{\xrightarrow{#1}}
\newcommand{\cHom}{\mathcal{H}om}
\newcommand{\cTor}{\mathcal{T}or}
\newcommand{\cExt}{\mathcal{E}xt}
\newcommand{\GD}{\mathrm{GD}}
\newcommand{\PsFun}{\mathrm{PsFun}}

\newcommand{\Ob}{\mathrm{Ob}}
\newcommand{\trdeg}{\mathrm{tr.deg}}

\newcommand{\Etqc}{\mathrm{Etqc}}
\newcommand{\Rep}{\mathrm{Rep}}

\newcommand{\Ext}{\mathrm{Ext}}
\newcommand{\Hom}{\mathrm{Hom}}
\newcommand{\Mod}{\mathrm{Mod}}
\newcommand{\cosk}{\mathrm{cosk}}
\newcommand{\op}{\mathrm{op}}

\newcommand{\lisse}{\mathrm{lisse}}
\newcommand{\cA}{\mathcal{A}}
\newcommand{\cB}{\mathcal{B}}
\newcommand{\cC}{\mathcal{C}}
\newcommand{\cD}{\mathcal{D}}

\newcommand{\cF}{\mathcal{F}}
\newcommand{\cG}{\mathcal{G}}
\newcommand{\cN}{\mathcal{N}}
\newcommand{\cO}{\mathcal{O}}
\newcommand{\cP}{\mathcal{P}}
\newcommand{\cS}{\mathcal{S}}
\newcommand{\cT}{\mathcal{T}}
\newcommand{\cU}{\mathcal{U}}
\newcommand{\cV}{\mathcal{V}}

\newcommand{\cX}{\mathcal{X}}
\newcommand{\cY}{\mathcal{Y}}
\newcommand{\cZ}{\mathcal{Z}}
\newcommand{\fH}{\mathcal{H}}
\newcommand{\fm}{\mathfrak{m}}
\newcommand{\one}{\mathrm{id}}
\newcommand{\F}{F}
\newcommand{\Z}{\mathbb{Z}}

\newcommand{\N}{\mathbb{N}}
\newcommand{\Kt}{K\sptilde{}}
\newcommand{\Cart}{\mathrm{Cart}}
\newcommand{\res}{\mathbin{|}}
\newcommand{\id}{\mathrm{id}}

\setlist[enumerate,1]{(1)}
\setlist{nosep}

\begin{document}
\numberwithin{equation}{subsection}
\title{Six operations and Lefschetz-Verdier formula for Deligne-Mumford stacks}
\author{Weizhe Zheng\thanks{Morningside Center of Mathematics, Academy of Mathematics and Systems Science, Chinese Academy of
Sciences, Beijing 100190, China; email: \texttt{wzheng@math.ac.cn}}}
\date{}
\maketitle

\begin{abstract}
Laszlo and Olsson constructed Grothendieck's six operations for
 constructible complexes on Artin stacks in \'etale cohomology under an
 assumption of finite cohomological dimension, with base change
 established on the level of sheaves. In this article we give a more
 direct construction of the six operations for complexes on
 Deligne-Mumford stacks without the finiteness assumption and establish
 base change theorems in derived categories. One key tool in our
 construction is the theory of gluing finitely many pseudofunctors
 developed in \cite{glue}. As an application, we prove a Lefschetz-Verdier
 formula for Deligne-Mumford stacks. We include both torsion and
 $\ell$-adic coefficients.
\end{abstract}

\section*{Introduction}
Grothendieck's six operations in \'etale cohomology of schemes were
constructed in \cite{SGA4}. In \cite{LO1}, \cite{LO2}, Laszlo and Olsson
extended much of this formalism to Artin stacks. One key ingredient in their
approach is biduality, which allows them to define the extraordinary direct
image functor $Rf_!$ as $D (Rf_*) D$. Base change for the functor $Rf_!$
defined in this way was only proved on the level of cohomology sheaves, due
to the difficulty of gluing objects in derived categories. The theory has
some other restrictions: it only works for constructible complexes of
modules over Gorenstein rings, and the base scheme $S$ is assumed to be
excellent admitting a dimension function with the additional assumption that
every finite-type $S$-scheme has finite $\ell$-cohomological dimension. Note
that the spectrum of a field does not satisfy this assumption in general.

In this article, we propose a more direct construction of Grothendieck's six
operations for Deligne-Mumford stacks, which allows to lift most of the
above-mentioned restrictions. Our construction of the functor $Rf_!$ is
closer to the original construction for schemes in SGA 4 \cite[XVII]{SGA4}.
Recall that every separated finite-type morphism $f\colon X\to Y$ of
quasi-compact quasi-separated schemes is the composition of an open
immersion $j$ followed by a proper morphism $p$ by Nagata compactification.
The functors $Rf_!$ for such morphisms is constructed by gluing \emph{two}
pseudofunctors: direct image $p\mapsto Rp_*$ and extension by zero $j\mapsto
j_!$. Although at present we do not have Nagata compactification for general
Deligne-Mumford stacks\footnote{See \cite{Rydhcomp} for partial results
toward this direction.}, Nagata compactification for algebraic spaces was
established by Conrad, Lieblich, and Olsson \cite{CLO}. Applying this to the
coarse spaces provides a decomposition of a morphism of Deligne-Mumford
stacks into \emph{three} morphisms, the first and third being proper and the
second being an open immersion. We then apply the theory of gluing finitely
many pseudofunctors developed in \cite{glue}, which extends the results on
gluing two cofibered categories in \cite[XVII]{SGA4}. Our construction is
compatible with the construction via biduality, whenever the two are both
defined.

Let us summarize the six operations and state the base change theorem for
torsion coefficients. By the functoriality of \'etale topoi of
Deligne-Mumford stacks, we have, for every morphism $f\colon \cX\to \cY$ of
Deligne-Mumford stacks, and every commutative ring $\Lambda$, functors
\begin{gather*}
  -\otimes^L-\colon D(\cX,\Lambda)\times D(\cX,\Lambda)\to D(\cX,\Lambda),\quad  R\cHom\colon D(\cX,\Lambda)^{\op}\times D(\cX,\Lambda)\to D(\cX,\Lambda),\\
f^*\colon D(\cY,\Lambda)\to D(\cX,\Lambda), \quad Rf_*\colon D(\cX,\Lambda)\to D(\cY,\Lambda).
\end{gather*}
Now let $S$ be a Noetherian scheme, and let $\Lambda$ be a commutative ring
annihilated by an integer $m$. For every separated morphism $f\colon \cX\to
\cY$ of $m$-prime inertia (Definition \ref{d.inert}) between finite-type and
finite-inertia Deligne-Mumford $S$-stacks, we construct functors
\[
Rf_!\colon D(\cX,\Lambda)\to D(\cY,\Lambda), \quad Rf^!\colon D(\cY,\Lambda)\to D(\cX,\Lambda).
\]
Here $D(-,\Lambda)$ denotes the unbounded derived category, with no
assumption on constructibility. Moreover, we construct pseudonatural
transformations encoding base change isomorphisms in derived categories. In
particular, we have the following (special case of Theorem \ref{p.bcc}
\ref{p.bcc1}).

\begin{sTheorem}
Let
  \[\xymatrix{\cX'\ar[r]^{h}\ar[d]_{f'} & \cX\ar[d]^{f}\\
  \cY'\ar[r]^{g} & \cY}\]
be a 2-Cartesian square of Deligne-Mumford stacks of finite type and finite
inertia over~$S$ with $f$ separated of $m$-prime inertia. The base change
map
  \[g^* Rf_! M \to Rf'_! h^* M\]
  is an isomorphism for all $M\in D(\cX,\Lambda)$.
\end{sTheorem}

Another approach for six operations on algebraic stacks has been developed
in a series of joint work with Yifeng Liu (\cite{LZ2}, \cite{LZ3}). One key
ingredient in that approach is homological descent, which reduces the
construction of $Rf_!$ to the case of schemes. Although that approach gives
the most general results in many cases, the techniques required, including a
theory on gluing in the $\infty$-categorical setting \cite{LZ1}, are quite
involved.

As an application of the constructions in the present article, we prove a
Lefschetz-Verdier formula for Deligne-Mumford stacks. Let us state the
formula for torsion coefficients (special case of Corollary \ref{c.lv}).
Note that even in the case of schemes, our result is slightly more general
than Lefschetz-Verdier formulas previously proven by Illusie \cite[III
Corollaire 4.5]{SGA5} and Varshavsky \cite[Proposition 1.2.5]{Varshavsky} as
we make fewer properness assumptions.

\begin{sTheorem}
Let $S$ be the spectrum of a field and let $\Lambda$ be a Noetherian
commutative ring annihilated by an integer $m$ invertible on $S$. Let
\[\xymatrix{X_1\ar[d]_{f_1} & B'\ar[l]_{b'_1}\ar[d]^{g'}\ar[r]^{b'_2} & X_2\ar[d]^{f_2} & B''\ar[l]_{b''_2}\ar[d]^{g''}\ar[r]^{b''_1} &
X_1\ar[d]^{f_1}\\
Y_1 & C'\ar[l]\ar[r] & Y_2 & C''\ar[l]\ar[r] & Y_1}
\]
be a 2-commutative diagram of separated Deligne-Mumford stacks of $m$-prime
inertia and of finite type over $S$ such that the morphisms $g'$ and $B''\to
C''\times_{Y_2} X_2$ are proper. For $i=1,2$, let $L_i\in
\Dcft(X_i,\Lambda)$,\footnote{For a Deligne-Mumford stack $T$,
$\Dcft(T,\Lambda)$ denotes the full subcategory of $D(T,\Lambda)$ spanned by
complexes of constructible cohomology sheaves and finite tor-dimension.}
$i=1,2$, $u\colon b'^*_1 L_1\to b'^!_2 L_2$, $v\colon b''^*_2L_2\to b''^!_1
L_1$. Then the morphism
\[g\colon B=B' \times_{X_1\times X_2} B'' \to C'\times_{Y_1\times Y_2} C''=C\]
is proper and we have an identity
\[g_!\langle u,v\rangle = \langle g'^\sharp_! u, g''^\sharp_! v\rangle\]
in $H^0(C,K_C)$. Here $K$ denotes the dualizing complex, $g_!\colon
H^0(B,K_B)\to H^0(C,K_C)$ is the map induced by the adjunction map $g_!
K_B\simeq g_! g^! K_C\to K_C$, $g'^\sharp_!$, $g''^\sharp_!$ are pushforward
maps defined in Construction \ref{c.shriek}, and $\langle-,-\rangle$ is the
Lefschetz-Verdier pairing defined in Construction \ref{c.pair}.
\end{sTheorem}

This article is organized as follows. In Section~\ref{s.1}, we apply the
theory of gluing pseudofunctors \cite{glue} and obtain a general gluing
result for Deligne-Mumford stacks. In Section~\ref{s.2}, we apply this
gluing result to construct the six operations for torsion coefficients and
relations between them, including the base change isomorphism (Theorem
\ref{p.bcc}). In Section~\ref{s.3}, we specialize the six operations to
constructible complexes and compare with operations constructed by duality.
In Section~\ref{s.topos}, we develop an $\ell$-adic formalism for a general
topos. In Section~\ref{s.5}, we apply this formalism to Deligne-Mumford
stacks and construct the six operations for $D^+_c(-,\cO)$, where $\cO$ is a
complete discrete valuation ring of residue characteristic~$\ell$ invertible
on $S$. Again, our gluing technique allows us to construct the base change
isomorphism in derived categories (Theorem \ref{p.dagO}). In
Section~\ref{s.6}, we show that if the fraction field $E$ of $\cO$ has
characteristic $0$, the restriction on the inertia of $f$ in the
construction of $Rf_!$ disappears when we pass from $\cO$ to $E$. In
Section~\ref{s.cc}, we define pushforward of cohomological correspondences,
which is used in the statement of Lefschetz-Verdier formula. In
Section~\ref{s.LV}, we state and prove the Lefschetz-Verdier formula for
Deligne-Mumford stacks.

The six operations constructed in this article are used in joint work with
Illusie \cite[Section~3]{Illusie-Zheng} to give a generalization of Laumon's
theorem comparing direct image and extraordinary direct image. In an
appendix to the present article (Section~\ref{s.7}), we relate this to
modular representation theory and collect a couple of consequences.

\subsection*{Acknowledgments}
The author warmly thanks Luc Illusie for the numerous comments he made on
drafts of this paper. The author also thanks Johan de Jong, Ofer Gabber,
Yifeng Liu, Martin Olsson, Jo\"el Riou, David Rydh, and Yichao Tian for
useful discussions. The author is grateful to the referee for a very careful
reading and many helpful suggestions. This work was partially supported by
China's Recruitment Program of Global Experts; National Natural Science
Foundation of China Grant 11321101; Hua Loo-Keng Key Laboratory of
Mathematics, Chinese Academy of Sciences; National Center for Mathematics
and Interdisciplinary Sciences, Chinese Academy of Sciences.

\section{Gluing pseudofunctors on Deligne-Mumford stacks}\label{s.1}
In this section, we prove a general result for the construction of
pseudofunctors from the 2-category of Deligne-Mumford stacks. Although the
statement of the result is about gluing two pseudofunctors, the key tool in
the proof is the general theory of gluing finitely many pseudofunctors
developed in \cite{glue}.

\begin{Convention}\label{s.stack}
In this article, all algebraic spaces and Deligne-Mumford stacks are assumed
to be quasi-separated. By a \emph{Deligne-Mumford stack}, we thus mean a
stack $\cX$ such that the diagonal $\Delta\colon \cX\to \cX\times_S\cX$ is
representable, quasi-compact and quasi-separated, and such that there exists
an algebraic space $X$ and an \'etale surjective morphism $X\to \cX$. As in
\cite{SP} and contrary to \cite[D\'efinition 4.1]{LMB}, we do not assume
that the diagonal of a Deligne-Mumford stack is separated. Note however that
a Deligne-Mumford stack of finite inertia has separated diagonal. Unless
otherwise stated, all rings are assumed to be commutative. Quasi-excellent
schemes are assumed to be Noetherian.
\end{Convention}

\begin{Proposition}\label{p.comp}
Let $S$ be a Noetherian scheme and let $f\colon \cX\to \cY$ be a separated
morphism between finite-type finite-inertia Deligne-Mumford $S$-stacks. Then
$f$ is isomorphic to a composition
  \[\cX\xrightarrow{\pi}\cX'\xrightarrow{j} \cZ \xrightarrow{p} \cY,\]
where $\pi$ is a proper homeomorphism, $j$ is an open immersion, and $p$ is
proper and representable. Moreover, if $f$ is quasi-finite, we can take $p$
to be finite.
\end{Proposition}

\begin{proof}
Let $g\colon X\to Y$ be the morphism of coarse spaces associated to~$f$,
which exists by the Keel-Mori theorem \cite{KM} (see also \cite[Theorem
1.1]{Conrad} and \cite[Theorem 6.12]{Rydhquot}). Then $g$ is a separated
morphism between finite-type algebraic $S$-spaces. Applying Nagata
compactification \cite[Theorem 1.2.1]{CLO} (or \cite[Theorem~B]{Rydhcomp})
to $g$, we get $g=qk$, where $q\colon Z\to Y$ is a proper morphism of
algebraic spaces and $k\colon X\to Z$ is an open immersion. It then suffices
to take $p$ to be the base change of $q$ and $j$ to be the base change
of~$k$. If $f$ is quasi-finite, so is $g$, and it then suffices to apply
Zariski's Main Theorem \cite[Th\'eor\`eme 16.5]{LMB} to~$g$.
\end{proof}

Let $\cC$ be a 2-category admitting 2-fiber products. The \emph{inertia} of
a morphism $f\colon X\to Y$ is defined to be
\[I_f=X\times_{\Delta_f,X\times_Y X,\Delta_f} X.\]
The following is an immediate consequence of \cite[Lemma
3.10]{Illusie-Zheng}.

\begin{Lemma}\label{p.I}
  Let
  \[\xymatrix{U'\ar[rr]\ar[dd]\ar[rd] && X'\ar[dd]|\hole \ar[rd]\\
  & U\ar[rr]\ar[dd] && X\ar[dd]\\
  V' \ar[rr]|\hole\ar[rd] && Y' \ar[rd]\\
  & V\ar[rr] && Y }\]
  be a 2-commutative cube in $\cC$ with 2-Cartesian bottom and top squares. Then the square
  \[\xymatrix{I_{U'/V'}\ar[r]\ar[d] & I_{X'/Y'}\ar[d]\\
  I_{U/V}\ar[r] & I_{X/Y}}\]
  is 2-Cartesian.
\end{Lemma}

It follows that Deligne-Mumford stacks of finite inertia are stable under
2-fiber products.

\begin{Definition}\label{d.inert}
Let $m$ be an integer. Following \cite[Definition 3.8]{Illusie-Zheng}, we
say that a morphism $f\colon \cX\to \cY$ of Deligne-Mumford stacks \emph{has
$m$-prime inertia} if, for every algebraically closed field~$\Omega$ and
every point $x\in \cX(\Omega)$, the order of the group
\[\Aut_{\cX_y}(x)\simeq \Ker(\Aut_{\cX}(x)\to \Aut_{\cY}(y))\]
is prime to $m$, where $y\in \cY(\Omega)$ is the image of $x$ under~$f$.
Note that morphisms of $m$-prime inertia are closed under composition and
base change.
\end{Definition}

Let us recall some definitions from \cite{glue}. Let $\cC$ be a
(2,1)-category, namely, a 2-category whose 2-cells are all invertible. A
2-subcategory $\cA$ of $\cC$ is called \emph{arrowy} if $\Ob(\cA)=\Ob(\cC)$
and for every pair of objects $(X, Y)$ of $\cC$, $\cA(X,Y)$ is a full
subcategory of $\cC(X,Y)$. Let $\cA$ and $\cB$ be two 2-arrowy subcategories
of $\cC$ and let $\cD$ be a 2-category. We defined in \cite[Definition
6.3]{glue} a 2-category $\GD^\Cart_{\cA,\cB}(\cC,\cD)$ whose objects are
quadruples $(F_\cA,F_\cB,(G_D),\rho)$, where $F_\cA\ \colon \cA\to \cD$ and
$F_\cB\colon \cB\to \cD$ are pseudofunctors, $(G_D)$ is a system of
compatibility data between $F_\cA$ and $F_\cB$ for Cartesian 2-squares $D$,
and $\rho\colon F_\cB\res \cA\cap \cB \to F_\cA \res \cA\cap \cB$ is a
pseudonatural equivalence.

\begin{Proposition}\label{p.gluestack}
Let $m$ be an integer, let $S$ be a Noetherian scheme, and let $\cC$ be the
2-category whose objects are Deligne-Mumford $S$-stacks of finite type and
finite inertia and whose morphisms are the separated morphisms of $m$-prime
inertia. Let $\cX$ be an object of $\cC$, let $\cA_\cX$ be the arrowy
subcategory of $\cC/\cX$ spanned by the open immersions, let $\cB_\cX$ be
the arrowy subcategory of $\cC/\cX$ spanned by the proper morphisms, and let
$\cD$ be a 2-category. Then the $\cD^{\Ob(\cC/\cX)}$-functor
  \begin{equation}\label{e.gluestack}
  \PsFun(\cC/\cX,\cD)\to \GD^\Cart_{\cA_\cX,\cB_\cX}(\cC/\cX,\cD)
  \end{equation}
is a $\cD^{\Ob(\cC/\cX)}$-equivalence \cite[Definition 1.5]{glue}.
\end{Proposition}

\begin{proof}
To simplify notation, let $\cC'=\cC/\cX$, $\cA=\cA_\cX$, $\cB=\cB_\cX$ and
let $\cA_1$ (resp.\ $\cA_2$, resp.\ $\cB_1$, resp.\ $\cB_2$) be the arrowy
subcategories of $\cC'$ whose morphisms are the quasi-finite (resp.\
representable and quasi-finite, resp.\ quasi-finite and proper, resp.\
finite) morphisms. By \cite[Theorem 6.5]{glue}, we may replace $\GD^\Cart$
by $\GD$. Then the $\cD^{\Ob(\cC')}$-functor \eqref{e.gluestack} is a
composite
  \begin{multline*}
    \PsFun(\cC',\cD)\xrightarrow{E_1} \GD_{\cA_1,\cB}(\cC',\cD) \xrightarrow{E_2} \GD_{\cB_1,\cA_2,\cB}(\cC',\cD)\\
    \xrightarrow{P_1} \GD_{\cA_2,\cB}(\cC',\cD) \xrightarrow{E_3} \GD_{\cA,\cB_2,\cB}(\cC',\cD) \xrightarrow{P_2} \GD_{\cA,\cB}(\cC',\cD),
  \end{multline*}
where $P_1$ and $P_2$ are $\cD^{\Ob(\cC')}$-equivalences by
\cite[Proposition 5.6]{glue} since $\cB_1,\cB_2\subset \cB$. Moreover, $E_1$
is a $\cD^{\Ob(\cC')}$-equivalence by \cite[Theorem 4.13]{glue}: it
satisfies assumption (1) of \cite[Theorem 4.13]{glue} by Proposition
\ref{p.comp}. Furthermore, $E_2$ and $E_3$ are
$\cD^{\Ob(\cC')}$-equivalences by \cite[Theorem 5.10]{glue}: $E_2$ and $E_3$
satisfy assumption (1) of \cite[Theorem 5.10]{glue} by Proposition
\ref{p.comp} and Zariski's Main Theorem, respectively. The other assumptions
of \cite[Theorems 4.13, 5.10]{glue} because the arrowy 2-subcategories are
stable under 2-base change and taking diagonals in $\cC'$.
\end{proof}

\section{Operations on $D(\cX,\Lambda)$}\label{s.2}
In this section, we apply the gluing result of the previous section to
construct the six operations for torsion coefficients and several relations
between them, including the base change isomorphism (Theorem \ref{p.bcc}
\ref{p.bcc1}) and the projection formula (Theorem \ref{p.bcc} \ref{p.bcc2}).
We then construct a functor $Rf_\dag$ in Construction \ref{s.dag}, which
will be used in the construction of $Rf_!$ for coefficients of
characteristic $0$ in Section~\ref{s.6}.

Let $\Lambda$ be a (commutative) ring. See Convention \ref{s.stack}.

\begin{Construction}\label{s.top}
Let $f\colon \cX\to \cY$ be a morphism of topoi. The functors
\begin{gather*}
  -\otimes_\Lambda -\colon \Mod(\cX,\Lambda)\times \Mod(\cX,\Lambda)\to \Mod(\cX,\Lambda),\quad  \cHom\colon \Mod(\cX,\Lambda)^{\op}\times \Mod(\cX,\Lambda)\to \Mod(\cX,\Lambda),\\
f^*\colon \Mod(\cY,\Lambda)\to \Mod(\cX,\Lambda), \quad f_*\colon \Mod(\cX,\Lambda)\to \Mod(\cY,\Lambda)
\end{gather*}
have derived functors \cite[Section 18.6]{KSCat}
\begin{gather*}
  -\otimes^L_\Lambda-\colon D(\cX,\Lambda)\times D(\cX,\Lambda)\to D(\cX,\Lambda),\quad  R\cHom\colon D(\cX,\Lambda)^{\op}\times D(\cX,\Lambda)\to D(\cX,\Lambda),\\
f^*\colon D(\cY,\Lambda)\to D(\cX,\Lambda), \quad Rf_*\colon D(\cX,\Lambda)\to D(\cY,\Lambda).
\end{gather*}
For $M\in D(\cX,\Lambda)$ and $N\in D(\cY,\Lambda)$, we have a canonical
isomorphism
\begin{equation}\label{e.RcHomadj}
R\cHom_\cY(N,Rf_*M)\simto Rf_*R\cHom_\cX(f^*N,M).
\end{equation}
We consider the projection formula map
\begin{equation}\label{e.pf}
  N\otimes_\Lambda^L Rf_* M \to Rf_* (f^*N \otimes_\Lambda^L M),
\end{equation}
adjoint to the composition
\[f^*(N\otimes_\Lambda^L Rf_*M)\xrightarrow{\sim} f^*N\otimes_\Lambda^L f^*Rf_* M \to f^*N \otimes_\Lambda^L M,\]
where the second map is induced by the adjunction $f^* Rf_* M\to M$.

If $\cX$ is algebraic, $\cY$ is locally coherent \cite[VI D\'efinition
2.3]{SGA4} and $f$ is coherent \cite[VI D\'efinition 3.1]{SGA4},  then
$R^qf_*$ commutes with small filtered inductive limits for all $q$ by
\cite[VI Th\'eor\`eme 5.1]{SGA4}. Thus, in this case, $f_*$-acyclic sheaves
on $\cX$ are stable under small filtered inductive limits. If, moreover,
$f_*$ is of finite cohomological dimension, then $Rf_*$ commutes with small
direct sums by \cite[Proposition 14.3.4 (ii)]{KSCat}.
\end{Construction}

\begin{Definition}
Let $\cX$ be a topos. For $M\in D(\cX,\Lambda)$ and an interval $I\subset
\Z$, we say $M$ has tor-dimension contained in $I$ if for every (constant)
$\Lambda$-module $N$, and every $i\in \Z\backslash I$, $\fH^i
(M\otimes^L_\Lambda N)=0$. We let $D_{\ft}(\cX,\Lambda)$ denote the full
subcategory of $D(\cX,\Lambda)$ spanned by objects of finite tor-dimension.
\end{Definition}

\begin{Remark}
Let $\cX$ be a Deligne-Mumford stack. The 2-category of Deligne-Mumford
$S$-stacks representable over $\cX$ is 2-equivalent to the 1-category
$\Rep(\cX)$ obtained by identifying isomorphic morphisms. The objects of
$\Rep(\cX)$ are pairs $(\cY,f)$, where $\cY$ is a Deligne-Mumford $S$-stack
and $f\colon \cY\to \cX$ is a representable morphism. A morphism of
$\Rep(\cX)$ from $(\cY,f)$ to $(\cZ,g)$ is an isomorphism class of pairs
$(h,\alpha)$, where $h\colon \cY\to \cZ$ is a morphism of Deligne-Mumford
$S$-stacks and $\alpha\colon f\Rightarrow g h$ is a 2-cell. In such a pair,
$h$ is necessarily a representable morphism. Two such pairs $(h,\alpha)$ and
$(i,\beta)$ are isomorphic if there exists a 2-cell $\gamma\colon
h\Rightarrow i$ such that $\beta=(g \gamma) \alpha$, or, in other words,
that $\beta$ is equals the composite $2$-cell
\[\xymatrix{\cY\ruppertwocell^i{^\gamma}\drlowertwocell_f{^\alpha} \ar[r]_h & \cZ\ar[d]^{g} \\
& \cX.}\]
\end{Remark}

\begin{Construction}
We define the \'etale site $\Et(\cX)$ of $\cX$ to be the full subcategory
$\Et(\cX)$ of $\Rep(\cX)$ consisting of Deligne-Mumford stacks representable
and \'etale over $\cX$, endowed with the \'etale topology. The category
$\Et(\cX)$ admits finite projective limits. The corresponding topos
$\cX_\et$ is algebraic. The full subcategory of $\Et(\cX)$ consisting of
affine schemes \'etale over~$\cX$, endowed with the \'etale topology,
defines the same topos by \cite[III Th\'eor\`eme 4.1]{SGA4}. If $\cX$ is
quasi-compact, then $\cX_\et$ is coherent (see Convention \ref{s.stack}). We
will write $\Mod(\cX,\Lambda)$ for $\Mod(\cX_\et,\Lambda)$ and
$D(\cX,\Lambda)$ for $D(\cX_\et,\Lambda)$.

Let $f\colon \cX\to \cY$ be a morphism of Deligne-Mumford stacks. The
functor
\[f^{-1}\colon \Et(\cY)\to \Et(\cX), \quad \cU\mapsto \cU\times_{\cY} \cX\]
is left exact and continuous by \cite[III Proposition 1.6]{SGA4}. Hence it
induces a morphism of topoi $(f_*, f^*)\colon \cX_\et \to \cY_\et$ by
\cite[III Proposition 1.3]{SGA4} and Construction \ref{s.top} applies. If
$f$ is quasi-compact, then this morphism of topoi is coherent. If
$g\colon\cX\to \cY$ is also a morphism and $\alpha\colon f\Rightarrow g$ is
a 2-cell, $\alpha$ induces a natural transformation $g^{-1}\Rightarrow
f^{-1}$, and hence a 2-cell $(f_*,f^*)\Rightarrow(g_*,g^*)$.

A surjective smooth morphism of Deligne-Mumford stacks is of cohomological descent, because \'etale locally it has a section.
\end{Construction}

\begin{Definition}\label{s.uh}
We say that a morphism $f$ of Deligne-Mumford stacks is a \emph{universal
homeomorphism} if it is a homeomorphism and remains so after every 2-base
change of Deligne-Mumford stacks. Note that we do \emph{not} assume $f$ to
be representable. Unlike the case of schemes, a universal homeomorphism does
\emph{not} induce an equivalence of \'etale topoi in general.
\end{Definition}

\begin{Construction}\label{d.jt}
Let $\cX$ be a topos and $U$ be an object of $\cX$. Let $\cU=\cX/U$ and
consider the morphism of topoi $j\colon \cU\to \cX$. The restriction functor
$j^*\colon \Mod(\cX,\Lambda) \to \Mod(\cU,\Lambda)$ admits a left adjoint
\cite[IV Proposition 11.3.1]{SGA4}
\[j_{!}\colon \Mod(\cU,\Lambda)\to \Mod(\cX,\Lambda).\]
We denote $j_!\Lambda_\cU$ by $\Lambda_{\cU,\cX}$. The functor $j_!$ is exact and induces a triangulated functor
\[j_!\colon D(\cU,\Lambda)\to D(\cX,\Lambda).\]

For $M\in D(\cU,\Lambda)$, $N\in D(\cX,\Lambda)$, we have canonical
isomorphisms
\begin{align}
\label{e.jshradj}
R\Hom_\cX(j_! M, N) &\to R\Hom_\cU(M,j^* N),\\
\label{e.jshradju}
R\cHom_\cX(j_! M, N) &\to Rj_* R\cHom_\cU(M,j^* N)
\end{align}
induced by
\begin{align*}
\Hom^\bullet_\cX(j_! M, N') &\simto \Hom^\bullet_\cU(M,j^* N'),\\
\cHom^\bullet_\cX(j_! M', N') &\simto j_* \cHom^\bullet_\cU(M',j^* N'),
\end{align*}
where $N'$ is homotopically injective and equipped with a quasi-isomorphism
$N\to N'$, $M'$ belongs to $\tilde\cP_\cU$ and is equipped with a
quasi-isomorphism $M'\to M$. Here $\Hom^\bullet$ and $\cHom^\bullet$ denote
the complexes associated to double complexes by taking products, and
$\tilde\cP_\cU$ is the smallest full triangulated subcategory of
$K(\Mod(\cU,\Lambda))$ closed under small direct sums and containing
$K^{-}(\cP_\cU)$, where $\cP_\cU$ is the full additive subcategory of
$\Mod(\cU,\Lambda)$ spanned by flat sheaves. Note that $j^*$ preserves
homotopically injective complexes. The map \eqref{e.jshradju} is adjoint to
the composition
\[j^*R\cHom_\cX(j_! M, N) \xrightarrow{\sim} R\cHom_\cU(j^*j_!M,j^* N) \to R\cHom_\cU(M,j^* N),\]
where the second map is deduced from the adjunction $M\to j^* j_! M$.

The projection formula map
\begin{equation}\label{e.jpf}
j_!(j^* N\otimes_\Lambda^L M)\to N\otimes_\Lambda^L j_! M,
\end{equation}
adjoint to the composition
\[j^* N\otimes^L_\Lambda M \to j^*N\otimes^L_\Lambda j^*j_! M \simeq j^*(N\otimes^L_\Lambda j_! M),\]
where the first map is deduced from the adjunction $M\to j^* j_! M$, is an
isomorphism. In fact, \eqref{e.jpf} is induced by the isomorphism of
complexes \cite[Proposition 18.2.5]{KSCat}
\[j_!\tot_\oplus(j^* N'\otimes_\Lambda M) \simto \tot_\oplus( N'\otimes_\Lambda j_! M)\]
where $N'$ belongs to $\tilde\cP_\cX$ and is equipped with a quasi-isomorphism $N'\to N$.
\end{Construction}

\begin{Construction}\label{s.jD}
Let $f\colon \cY\to \cX$ be a morphism of topoi and $U$ be an object of $\cX$. Let $V=f^{-1}(U)$,
$\cU=\cX/U$, $\cV=\cY/V$. Consider the following 2-commutative square $D$ of topoi
  \[\xymatrix{\cV\ar[r]^{j'}\ar[d]_{g} & \cY\ar[d]^f\\
\cU\ar[r]^j & \cX.}
\]
The base change map, natural transformation
of functors $D(\cY,\Lambda)\to D(\cU,\Lambda)$
\[B_D\colon j^* Rf_* \Rightarrow Rg_* j'{}^*,\]
is a natural isomorphism. Thus the base change map deduced from $B_D$ by
taking left adjoints is a natural isomorphism of functors $D(\cU,\Lambda)\to
D(\cY,\Lambda)$
\begin{equation}\label{e.shrbc}
A_D\colon j'_!g^* \Rightarrow f^* j_!.
\end{equation}
By \cite[Construction 8.6]{glue}, $B_D^{-1}$ and $A_D^{-1}$ induce by
adjunction the same natural transformation of functors $D(\cV,\Lambda)\to
D(\cX,\Lambda)$
\begin{equation}\label{e.shrstar}
G_D\colon j_! Rg_* \Rightarrow Rf_* j'_!.
\end{equation}
\end{Construction}

\begin{Construction}\label{d.j}
Let $j\colon \cU\to \cX$ be a representable \'etale morphism of
Deligne-Mumford $S$-stacks. Then Construction \ref{d.jt} applies to
$j_{\et}\colon \cU_{\et}\to \cX_{\et}$. If $j$ is an open immersion, the
adjunction map $j^*Rj_*\Rightarrow \one_{D(\cU,\Lambda)}$ is invertible and
we have thus a natural transformation of functors $D(\cU,\Lambda)\to
D(\cX,\Lambda)$
\begin{equation}\label{e.rho}
  j_!\Rightarrow Rj_*,
\end{equation}
compatible with composition of open immersions.

If
\[\xymatrix{\cV\ar[r]^{j'}\ar[d]_{g} & \cY\ar[d]^f\\
\cU\ar[r]^j & \cX}
\]
is a 2-Cartesian square of Deligne-Mumford $S$-stacks with $j$ representable
and \'etale, then Construction \ref{s.jD} applies.
\end{Construction}

\begin{Construction}\label{d.dist}
Let $f\colon \cX\to\cY$ be a finite morphism of Deligne-Mumford $S$-stacks.
Then
\[f_*\colon \Mod(\cX,\Lambda)\to \Mod(\cY,\Lambda)\]
is exact and commutes with small filtered inductive limits. Let $f^!\colon
\Mod(\cY,\Lambda)\to \Mod(\cX,\Lambda)$ be its right adjoint, which is left
exact and thus has a right derived functor $Rf^!\colon D(\cY,\Lambda)\to
D(\cX,\Lambda)$, right adjoint to $Rf_*\colon D(\cX,\Lambda)\to
D(\cY,\Lambda)$. By finite base change, the projection formula map
\eqref{e.pf} is an isomorphism. It follows that for $K\in D(\cX,\Lambda)$,
$M,N\in D(\cY,\Lambda)$, we have canonical isomorphisms
\begin{gather}\notag
  R\cHom_{\cY}(Rf_* K, N) \to Rf_* R\cHom_{\cX}(K, Rf^! N),\\
  R\cHom_\cX(f^* M, Rf^! N) \to Rf^!R\cHom_\cY(M,N).\label{e.dist4}
\end{gather}
We consider the map
\begin{equation}\label{e.pffinite}
  f^*N\otimes_\Lambda^L Rf^!M \to Rf^!(N\otimes^L_\Lambda M)
\end{equation}
adjoint to the composite
\[f_*(f^*N\otimes_\Lambda^L Rf^!M)\simto N\otimes^L_\Lambda f_*Rf^!M \to N\otimes^L_\Lambda M,\]
where the first map is the inverse of \eqref{e.pf} and the second map is
induced by the adjunction $f_*Rf^!M \to M$.
\end{Construction}

\begin{Construction}
Let $i\colon \cY\to\cX$ be a closed immersion of Deligne-Mumford stacks, and
let $j\colon \cU \to \cX$ be the complementary open immersion. For any
complex $M$ of $\Lambda$-modules on $\cX$, we have a natural short exact
sequence
\[0 \to j_!j^* M \to M \to i_* i^* M \to 0,\]
hence a distinguished triangle in $D(\cX,\Lambda)$
\begin{equation}\label{e.dist1}
j_!j^* M \to M \to i_* i^* M \to.
\end{equation}

For any complex $N$ of injective $\Lambda$-modules on $\cX$, we have a natural short exact sequence
\[0\to i_* i^! N \to N \to j_*j^* N \to 0.\]
It follows that, for any $N\in D(\cX,\Lambda)$, we have a distinguished triangle
\begin{equation}\label{e.dist2}
i_*Ri^!N \to N \to Rj_* j^* N \to.
\end{equation}
In fact, it suffices to take a quasi-isomorphism $N'\to N$, where $N'$ is
homotopically injective with injective components \cite[Proposition 14.1.6,
Theorem 14.1.7]{KSCat}.
\end{Construction}

\begin{Notation}
For a 2-Cartesian square
\[\xymatrix{\cX\ar[r]^{p_2}\ar[d]_{p_1} &\cX_2\ar[d]\\
\cX_1\ar[r] & \cS }
\]
of Deligne-Mumford stacks and $M_i\in D(\cX_i,\Lambda)$, $i=1,2$, we write
\[M_1\boxtimes^L_{\cS,\Lambda} M_2=p_1^*M_1\otimes^L_\Lambda p_2^* M_2.\]
We write $\boxtimes^L_\cS$ for $\boxtimes^L_{\cS,\Lambda}$ when no confusion
arises.
\end{Notation}

\begin{Proposition}[Smooth base change]\label{p.sbc}
Let
\[\xymatrix{\cX'\ar[r]^{h}\ar[d]_{f'} & \cX\ar[d]^{f}\\
\cY'\ar[r]^{g} & \cY}
\]
be a 2-Cartesian square of Deligne-Mumford stacks, and let $M\in
D(\cX,\Lambda)$. Assume either
 \begin{itemize}
 \item[(a)] $g$ is \'etale, or
 \item[(b)] $g$ is smooth, $\Lambda$ is annihilated by an integer $m$
     invertible on~$\cY$, and $M\in D^{+}(\cX,\Lambda)$.
 \end{itemize}
 Then the base change map
\[g^* Rf_* M \to Rf'_* h^* M\]
is an isomorphism.
\end{Proposition}

\begin{proof}
This is standard. We give a proof for the sake of completeness.

Consider the diagram with 2-Cartesian square
\[\xymatrix{Y''\ar[r]^{\beta}& Y'\ar[r]^{g'}\ar[d]_{\alpha'} & Y\ar[d]^{\alpha}\\
&\cY'\ar[r]^{g} & \cY}\] where $\alpha$ and $\beta$ are \'etale
presentations with $Y$ and $Y''$ disjoint unions of quasi-compact schemes.
Since base change by $\alpha$ and $\alpha'\beta$ holds trivially, up to
replacing $g$ by $g' \beta$, we may assume that $\cY$ and $\cY'$ are
quasi-compact schemes. Case (a) is then trivial. In case (b), take an
\'etale presentation $\gamma\colon X\to \cX$ with $X$ an algebraic
$S$-space. Let $X_\bullet=\cosk_0\gamma$ and let $\gamma_\bullet\colon
X_\bullet\to \cX$ be the projection. Then $X_n=(X/\cX)^{n+1}$ is an
algebraic $S$-space (even if we had taken $X$ to be an $S$-scheme). Take
$X'_\bullet=\cX'\times_\cX X_\bullet$ (2-fiber product) and consider the
square
\[\xymatrix{X'_\bullet\ar[r]^{h_\bullet}\ar[d]_{\gamma'_\bullet} & X_\bullet\ar[d]^{\gamma_\bullet}\\
\cX'\ar[r]^{h} & \cX.}
\]
By cohomological descent, the adjunction $M\to
R\gamma_{\bullet*}\gamma_\bullet^*M$ is an isomorphism. It follows that it
suffices to show that the base change maps
\[g^* R(f\gamma_\bullet)_* M_\bullet \to R(f'\gamma'_\bullet)_*h^* M_\bullet, \quad h^* R\gamma_{\bullet*} M_\bullet\to R\gamma'_{\bullet*} h_\bullet^* M_\bullet
\]
are isomorphisms, where $M_\bullet = \gamma_\bullet^* M$. For the second
map, repeating the first reduction step of this proof, we may assume that
$\cX$ and $\cX'$ are quasi-compact schemes. Therefore, we are reduced to
proving the theorem under the additional hypotheses that $\cY$, $\cY'$ are
quasi-compact schemes and $\cX$ is an algebraic space. In this case, take an
\'etale presentation $X\to \cX$ with $X$ a scheme and repeat the preceding
reduction. We may then assume that $\cX$ is also an $S$-scheme. In this
case, the result is classical (\cite[XVI Corollaire 1.2]{SGA4} for $f$
quasi-compact, and a consequence of base change for $Rg_!$ \cite[XVII
Th\'eor\`eme 5.2.6]{SGA4} and Poincar\'e duality \cite[XVIII Th\'eor\`eme
3.2.5]{SGA4} in the general case).
\end{proof}

\begin{Proposition}\label{p.Homsm}
Let $f\colon \cX\to \cY$ be a smooth morphism of Deligne-Mumford stacks.
Assume that $\Lambda$ is annihilated by an integer invertible on $\cY$. Let
$M,L\in D(\cY,\Lambda)$. Then the map
  \begin{equation}\label{e.Homsm}
  f^* R\cHom_\cY(M,L)\to R\cHom(f^* M,f^* L)
  \end{equation}
  is an isomorphism.
\end{Proposition}

\begin{proof}
The problem is local for the \'etale topology on $\cX$ and on $\cY$. We may
assume that $\cX$ and $\cY$ are quasi-compact schemes and $f$ is separated.
Then $f$ is compactifiable and \eqref{e.Homsm} becomes the inverse of the
trivial duality isomorphism \cite[XVIII Corollaire 3.1.12.2]{SGA4} via the
natural isomorphism $f^*(d)[2d]\simto Rf^!$ \cite[XVIII Th\'eo\`eme
3.2.5]{SGA4}, where $d$ is the relative dimension of $f$. Note that trivial
duality holds in fact for unbounded complexes: the proof of
\eqref{e.shradj2} applies.
\end{proof}

Assume in the rest of this section that $S$ is a Noetherian
scheme\footnote{In fact $S$ will denote a Noetherian scheme in the rest of
the article. Since we will at times make other assumptions on $S$, we prefer
to list the assumptions on $S$ in each section.}, and $\Lambda$ is
annihilated by an integer $m\ne 0$. Unless otherwise stated, we do
\emph{not} assume that $m$ is invertible on $S$. We say that a morphism
$f\colon \cX\to \cY$ of Deligne-Mumford $S$-stacks is \emph{$S$-proper}, if
$f$ is the 2-base change of some proper morphism $f_0\colon \cX_0\to \cY_0$
of Deligne-Mumford $S$-stacks with $\cY_0$ locally of finite type over $S$.
Recall that the existence of finite covers by schemes \cite[Th\'eor\`eme
16.6]{LMB} implies base change in $D^+$ for $S$-proper morphisms. For
morphisms of $m$-prime inertia, this can be generalized to unbounded
complexes as follows.

\begin{Lemma}\label{p.cd}
Let $f\colon \cX\to \cY$ be an $S$-proper morphism of Deligne-Mumford
$S$-stacks. Assume that $f$ has $m$-prime inertia and the fibers have
dimension $\le d$. Then $f_*$ has cohomological dimension $\le 2d$.
\end{Lemma}

\begin{proof}
Up to replacing $\cY$ by an \'etale presentation and $\cX$ by the
corresponding 2-pullback, we may assume that $\cY$ is a separated scheme.
Then $f$ factorizes through the coarse space of $\cX$, and we are reduced to
two cases: (a) $f$ is a universal homeomorphism; (b) $f$ is representable.
By proper base change for $D^+$, we may assume that $\cY$ is the spectrum of
an algebraically closed field. In case (a), $\cX^\red=BG$, $G$ of order
prime to $m$. In this case, $f_*$ can be identified with the functor of
taking $G$-invariants, which is exact. Case (b) follows from \cite[Lemma
7.4]{Illusie-Zheng}. We give a proof that does not make use of cohomology
with proper support. We proceed by induction on $d$. By Chow's lemma
\cite[Theorem IV.3.1]{Knutson}, there exists $\pi\colon X\to \cX$ proper and
birational such that $X$ is a scheme. For any $\cF\in \Mod(\cX,\Lambda)$,
complete the adjunction $\cF\to R\pi_* \pi^*\cF$ into a distinguished
triangle
  \[\cF\to R\pi_* \pi^*\cF\to M \to.\]
By the case of schemes of this lemma \cite[X Corollaires 4.3, 5.2]{SGA4},
$Rf_* R\pi_* \pi^*\cF \simeq R(f\pi)_* \pi^* \cF\in D^{\le 2d}$ and, for
all~$i$, the support of $\fH^i M$ has dimension $\le d-1-i/2$. By induction
hypothesis, $Rf_* M \in D^{\le 2d-2}$. Thus $Rf_* \cF\in D^{\le 2d}$.
\end{proof}

\begin{Proposition}\label{p.pbc}\leavevmode
\begin{enumerate}
\item \label{p.pbc1} (Proper base change) Let
\[\xymatrix{\cX'\ar[r]^{h}\ar[d]_{f'} & \cX\ar[d]^{f}\\
\cY'\ar[r]^{g} & \cY}
\]
be a 2-Cartesian square of Deligne-Mumford
$S$-stacks where $f$ is $S$-proper and has $m$-prime inertia, $M\in
D(\cX,\Lambda)$. Then the base change map \[g^* Rf_* M \to Rf'_* h^* M\] is
an isomorphism.

\item \label{p.pbc2} (Projection formula) Let $f\colon \cX\to \cY$ be an
    $S$-proper morphism of Deligne-Mumford $S$-stacks, and let $M\in
    D(\cX,\Lambda)$, $N\in D(\cY,\Lambda)$. Assume that $f$ has $m$-prime
    inertia. Then the map \eqref{e.pf}
\[N\otimes^L_\Lambda Rf_* M\to Rf_*(f^* N \otimes^L_\Lambda M)\]
is an isomorphism.
\end{enumerate}
\end{Proposition}

\begin{proof}
\begin{itemize}
  \item[\ref{p.pbc1}] We may assume $\cY$ quasi-compact. By Proposition
      \ref{p.cd}, we are then reduced to the known case when $M\in D^+$.

  \item[\ref{p.pbc2}] By \ref{p.pbc1}, we may assume that $\cY$ is the
      spectrum of a separably closed field. Since $Rf_*$ commutes with
      small direct sums (Construction \ref{s.top}), we may assume $M\in
      D^-(\cX,\Lambda)$, and $N$ is represented by a complex in
      $C^-(\cY,\Lambda)$ with flat components. Then we may assume $M\in
      \Mod(\cX,\Lambda)$, and $N$ is a flat $\Lambda$-module. Then $N$ is
      a filtered inductive limit of finite free $\Lambda$-modules. Since
      $R^q f_*$ commutes with such limits (Construction \ref{s.top}), we
      may assume that $N$ is a finite free $\Lambda$-module. In this case
      \eqref{e.pf} is obviously an isomorphism.
\end{itemize}
\end{proof}

\begin{Proposition}\label{p.uhmp}
  Let
  \[\xymatrix{\cX\ar[r]^f\ar[d] & \cY\ar[d]\\
  \cX_0 \ar[r]^{f_0} & \cY_0}\]
be a 2-Cartesian square of Deligne-Mumford $S$-stacks where $f_0$ is a
proper universal homeomorphism (Definition \ref{s.uh}) of $m$-prime inertia,
and $\cY_0$ is locally of finite type over $S$. Let $M\in D(\cY,\Lambda)$.
Then the adjunction map $M\to Rf_* f^* M$ is an isomorphism.
\end{Proposition}

\begin{proof}
By proper base change (Proposition \ref{p.pbc}), we may assume that $\cY$ is
the spectrum of an algebraically closed field. Then $\cX^\red=BG$, $G$ of
order prime to $m$. In this case the assertion is trivial.
\end{proof}

\begin{Construction}\label{s.fshr}
We use the gluing result of Proposition \ref{p.gluestack} to give a
construction of $Rf_!$. Let $\cC$, $\cA=\cA_S$, $\cB=\cB_S$ be as in
Proposition \ref{p.gluestack}, and let $\cD$ be the 2-category of
triangulated categories. Define an object $(F_\cA,F_\cB,G,\rho)$ of
$\GD^\Cart_{\cA,\cB}(\cC,\cD)$ as follows. Let $F_\cA\colon \cA\to \cD$ be
the pseudofunctor given by Construction \ref{d.j}:
  \[\cX\mapsto D(\cX,\Lambda),\quad j\mapsto j_!, \quad \alpha\mapsto \alpha_!,\]
and let $F_\cB\colon \cB\to \cD$ be the pseudofunctor given by
  \[\cX\mapsto D(\cX,\Lambda),\quad p\mapsto Rp_*, \quad \alpha\mapsto R\alpha_*.\]
For every proper open immersion $f$, the 2-cell \eqref{e.rho} $f_!\Rightarrow
Rf_*$ is invertible and let $\rho(f)$ be its inverse. For every
2-Cartesian square $D$ of the form
  \begin{equation}\label{e.2}
  \xymatrix{X\ar[r]^j\ar[d]_q & Y \ar[d]^p\\
  Z\ar[r]^i & W,}
  \end{equation}
let $G_D\colon i_! Rq_* \Rightarrow Rp_* j_!$ be the 2-cell as in
\eqref{e.shrstar}. Then $i^* G_D$ can be identified with $\one_{Rq_*}$,
hence is an isomorphism. Consider the complementary square
  \[\xymatrix{Y-X\ar[r]^{j'}\ar[d]_{q'} & Y\ar[d]^p\\
  W-Z\ar[r]^{i'} & W.}\]
The proper base change map
  \begin{equation}\label{e.fshr}
  i'^*Rp_* \Rightarrow Rq'_*j'{}^*
  \end{equation}
is an isomorphism by Proposition \ref{p.pbc} \ref{p.pbc1}, hence $i'{}^*
G_D$ is $0\Rightarrow 0$. It follows that $G_D$ is an isomorphism. Axioms
(b), (b$'$), (c), (c$'$) in the definition of $\GD^\Cart$ follow from
\cite[Proposition 8.8]{glue}. By Proposition \ref{p.gluestack}, this defines
a pseudofunctor $F\colon \cC\to\cD$. For any morphism of~$\cC$, namely, any
separated morphism $f\colon \cX\to \cY$ of $m$-prime inertia between
finite-type and finite-inertia Deligne-Mumford $S$-stacks, we define
  \[Rf_!\colon D(\cX,\Lambda)\to D(\cY,\Lambda)\]
to be $F(f)$.

The gluing formalism also enables us to construct the following natural
transformations.
\begin{enumerate}
\item \emph{Support-forgetting map.} Let $\tilde \cD$ be the 2-category
    whose objects are triangulated categories and morphisms $\cT\to \cT'$
    are triples $(E,E',\epsilon)$, where $E,E'\colon \cT\to \cT'$ are
    triangulated functors and $\epsilon\colon E\Rightarrow E'$ is a
    pseudonatural transformation. Let $\tilde F_\cA\colon \cA\to \tilde
    \cD$ be the pseudofunctor
  \[\cX\mapsto D(\cX,\Lambda), \quad j\mapsto (j_!\Rightarrow Rj_*).\]
given by \eqref{e.rho} and let $\tilde F_\cB\colon \cB\to \tilde \cD$ be
the pseudofunctor given by
\[\cX\mapsto D(\cX,\Lambda), \quad p\mapsto (Rp_* = Rp_*).\]
Let $F'\colon \cC\to \cD$ be the pseudofunctor given by
\[\cX\mapsto D(\cX,\Lambda), \quad f\mapsto Rf_*.\]
By \cite[Proposition 8.8 (4)]{glue}, the gluing data for $F$ and $F'$
determine an object $(\tilde F_\cA,\tilde F_\cB, \tilde G,\tilde \rho)$ of
$\GD^\Cart_{\cA,\cB}(\cC,\tilde \cD)$, thus defines a pseudofunctor
$\tilde F\colon \cC\to \tilde \cD$ by Proposition \ref{p.gluestack}. For
any morphism $f\colon \cX\to \cY$ of $\cC$, this defines a natural
transformation of functors $D(\cX,\Lambda)\to D(\cY,\Lambda)$
  \[f_!\Rightarrow Rf_*,\]
which is a natural isomorphism if $f$ is proper.

\item \label{s.fshr2} \emph{Base change isomorphism.} Let $g\colon \cX'\to
    \cX$ be a morphism of Deligne-Mumford stacks, with $\cX$ of finite
    type and finite inertia over $S$, $\cX'$ of finite type and finite
    inertia over some Noetherian scheme $S'$. For every object $\cY$ of
    $\cC/\cX$, fix a 2-base change $\cY'\to \cY$ of $g$. Then $\cY'$ is of
    finite inertia over $S'$ by Lemma \ref{p.I}. For every morphism
    $\cZ\to \cY$ of $\cC/\cX$, fix a 2-Cartesian square of $S$-stacks
    obtained by 2-base change by $g$
\[\xymatrix{\cZ'\ar[r]^{g''}\ar[d]_{f'} & \cZ\ar[d]^f\\
  \cY'\ar[r]^{g'} & \cY}
\]
In this way, we have defined a pseudofunctor $\cC/\cX\to \cC_{S'}$, where
 $\cC_{S'}$ is defined similarly to $\cC$ but with $S$ replaced by $S'$.
 Let $F_1$ be the composition $\cC/\cX\to \cC\xrightarrow{F} \cD$, let
 $F_2$ be the composition $\cC/\cX\to \cC_{S'}\xrightarrow{F_{S'}} \cD$,
 and let $\epsilon_0\colon \lvert F_1 \rvert \to \lvert F_2 \rvert$ be
 given by $\epsilon_0(\cY)=g'{}^*$ \cite[Notation 1.9]{glue}. Let
 $\epsilon_\cA\colon F_1\res\cA_\cX \Rightarrow F_2\res\cA_\cX$ be the
 pseudonatural transformation with $\lvert \epsilon_\cA \rvert=\epsilon_0$
 given by the inverse of the base change map \eqref{e.shrbc}, and let
 $\epsilon_\cB\colon F_1 \res\cB_\cX \Rightarrow F_2\res\cB_\cX$ be
 pseudonatural transformation with $\lvert \epsilon_\cB \rvert=\epsilon_0$
 given by proper base change. It follows from \cite[Propositions 8.10,
 8.11]{glue} that $(\epsilon_\cA,\epsilon_\cB)$ is a morphism of
 $\GD^\Cart_{\cA_\cX,\cB_\cX}(\cC/\cX,\cD)$. Thus it induces a
 pseudonatural transformation $\epsilon \colon F_1\Rightarrow F_2$ with
 $\lvert \epsilon \rvert = \epsilon_0$. For any morphism $f\colon \cZ\to
 \cY$ of $\cC/\cX$, $\epsilon(f)$ is a natural isomorphism
  \[g'{}^* Rf_!\Rightarrow Rf'_! g''{}^*.\]

\item \emph{Projection formula isomorphism.} Let $\cX$ be a
    Deligne-Mumford $S$-stack of finite type and finite inertia, and let
    $N\in D(\cX,\Lambda)$. Let $F_1$ be the composition $\cC/\cX\to \cC
    \xrightarrow{F} \cD$ as in \ref{s.fshr2} and let $\epsilon_0\colon
    \lvert F_1\rvert \to \lvert F_1 \rvert$ be given by
    $\epsilon_0(\cY)=N_\cY\otimes^L_\Lambda -$, where $N_\cY$ is the pull
    back of $N$ to $\cY$. Let $\epsilon_\cA\colon F_1\res\cA_\cX
    \Rightarrow F_1\res\cA_\cX$ be the pseudonatural transformation with
    $\lvert \epsilon_\cA\lvert = \epsilon_0$ given by the inverse of  the
    projection formula \eqref{e.jpf}, and let $\epsilon_\cB\colon
    F_1\res\cB_\cX \Rightarrow F_1\res\cB_\cX$ be the pseudonatural
    transformation with $\lvert \epsilon_\cB\lvert = \epsilon_0$ given by
    the projection formula (Proposition \ref{p.pbc} \ref{p.pbc2}).  Let
    $[1]$ be the category whose objects are $0$ and $1$ and whose
    morphisms are $\one_0$, $\one_1$ and $s\colon 0\to 1$. Consider the
    pseudofunctor $\cC/\cX\times [1]\to \cD$ given by
  \[(\cY,i)\mapsto D(\cY,\Lambda), \quad (\one_\cY,s)\mapsto N_\cY\otimes^L_\Lambda-, \quad (f,\one_i)\mapsto f^*,\]
  where $i=0,1$. Then \eqref{e.jpf} and the map in Proposition \ref{p.pbc}
  \ref{p.pbc2} are respectively the base change maps \cite[Constructions
  8.5 and 8.3]{glue}. Hence \cite[Propositions 8.10, 8.11]{glue} implies
  that $(\epsilon_\cA,\epsilon_\cB)$ is a morphism of
  $\GD^\Cart_{\cA_\cX,\cB_\cX}(\cC/\cX,\cD)$.  Thus it induces a
  pseudonatural transformation $\epsilon \colon F_1\Rightarrow F_1$ with
  $\lvert \epsilon \rvert= \epsilon_0$. For any morphism $f\colon \cZ\to
  \cY$ of $\cC/\cX$, $\epsilon(f)$ is a natural isomorphism
  \[N_\cY\otimes^L_\Lambda Rf_! - \Rightarrow Rf_!(N_\cZ\otimes^L_\Lambda -).\]
\end{enumerate}
  We have obtained the following.
\end{Construction}

\begin{Theorem}\label{p.bcc}\leavevmode
\begin{enumerate}
\item \label{p.bcc1} (Base change) For every 2-Cartesian square of Deligne-Mumford stacks
  \[\xymatrix{\cX'\ar[r]^{h}\ar[d]_{f'} & \cX\ar[d]^{f}\\
  \cY'\ar[r]^{g} & \cY}\] with $\cX$ and $\cY$ of finite type and finite
inertia over $S$, $\cY'$ of finite type and finite inertia over some
Noetherian scheme $S'$, $f$ separated of $m$-prime inertia, the base
change map
  \[g^* Rf_! M \to Rf'_! h^* M\]
  is an isomorphism for all $M\in D(\cX,\Lambda)$.

\item \label{p.bcc2} (Projection formula) Let $f\colon \cX\to \cY$ be a
    separated morphism between finite-type finite-inertia Deligne-Mumford
    $S$-stacks, and let $M\in D(\cX,\Lambda)$, $N\in D(\cY,\Lambda)$.
    Assume that $f$ has $m$-prime inertia. Then the projection formula map
    \begin{equation}\label{e.pfshr0}
  N\otimes^L_\Lambda Rf_! M \to Rf_!(f^* N \otimes^L_\Lambda M)
  \end{equation}
  is an isomorphism.
\end{enumerate}
\end{Theorem}

\begin{Corollary}
Let $f\colon \cX\to \cY$ be a separated morphism between finite-type and
finite-inertia Deligne-Mumford $S$-stacks. Assume that $f$ has $m$-prime
inertia and the fibers of $f$ have dimension $\le d$. Then the cohomological
amplitude of $Rf_!\colon D(\cX,\Lambda)\to D(\cY,\Lambda)$ is contained in
$[0,2d]$. Moreover, for $M\in D(\cX,\Lambda)$ of tor-dimension $\ge 0$,
$Rf_! M$ has tor-dimension $\ge 0$.
\end{Corollary}

\begin{proof}
For the first assertion, we may assume that $\cY$ is the spectrum of a field
by Theorem \ref{p.bcc} \ref{p.bcc1}. We decompose $f\simeq pj\pi$ as in
Proposition \ref{p.comp} with $j$ dominant. Then $j_!$ is exact, $\pi_*$ is
exact, and $p_*$ has cohomological dimension $\le 2d$ by Lemma \ref{p.cd},
because the source of $p$ has dimension $\le d$. The second assertion
follows from Theorem \ref{p.bcc} \ref{p.bcc2}.
\end{proof}

The following is immediate.

\begin{Corollary}[K\"unneth formula]
For every 2-commutative diagram
\[\xymatrix{\cX_1\ar[rd]\ar[d]_{f_1} && \cX_2\ar[d]^{f_2}\ar[ld]\\
\cY_1\ar[r] &\cS& \cY_2\ar[l]}
\]
between finite-type and finite-inertia Deligne-Mumford $S$-stacks such that
$f_1$ and $f_2$ are separated of $m$-prime inertia, and for $M_i\in
D(\cX_i,\Lambda)$, $i=1,2$, we have a canonical isomorphism
\begin{equation}\label{e.Kunnethlshr}
\KF_!\colon f_{1!}M_1\boxtimes^L_{\cS} f_{2!}M_2 \simto f_! (M_1\boxtimes^L_{\cS} M_2),
\end{equation}
where $f\colon \cX_1\times_{\cS}\cX_2 \to \cY_1\times_{\cS}\cY_2$.
\end{Corollary}

Moreover, \eqref{e.Kunnethlshr} is compatible with composition of morphisms.

\begin{Proposition}\label{p.trace}
There exists a unique way to define, for every separated flat morphism
$f\colon \cX\to \cY$ of $m$-prime inertia and fibers of dimension $\le d$
between Deligne-Mumford stacks of finite type and finite inertia over some
Noetherian scheme $T$ and every sheaf of $\Lambda$-modules $\cF$ on $\cY$, a
trace map
  \[\Tr_f(\cF)\colon R^{2d}f_!f^* \cF (d)\to \cF\]
  satisfying the following conditions:
  \begin{itemize}
    \item[(a)] (Compatibility with base change) For every 2-Cartesian square of Deligne-Mumford stacks
    \[\xymatrix{\cX'\ar[r]^h\ar[d]_{f'}&\cX\ar[d]^f\\
    \cY'\ar[r]^g & \cY}\] with $\cX$ and $\cY$ of finite type and finite
  inertia over some Noetherian scheme $T$, $\cX'$ and $\cY'$ of finite
  type and finite inertia over some Noetherian scheme $T'$, $f$ as above,
  and for every $\cF\in \Mod(\cY,\Lambda)$, the following diagram commutes
    \[\xymatrix{g^*R^{2d} f_! f^* \cF (d)\ar[r]^-{g^*\Tr_f(\cF)}\ar[d]_c & g^*\cF\\
    R^{2d}f'_! h^* f^* \cF (d)\ar[r]^\simeq & R^{2d}f'_! f'{}^* g^* \cF (d).\ar[u]_{\Tr_{f'}(g^*\cF)}}\]
    Here $c$ is the base change isomorphism.
    \item[(b)] (Compatibility with composition) For every composable pair
        of separated flat morphisms
        \[\cX\xrightarrow{f}  \cY \xrightarrow{g}  \cZ \]
    of $m$-prime inertia of Deligne-Mumford stacks of finite type and
    finite inertia over some Noetherian scheme $T$ and such that the
    fibers of $f$ and $g$ have dimension $\le d$ and $\le e$,
    respectively, and for every $\cF\in \Mod(\cZ,\Lambda)$, the following
    diagram commutes
        \[\xymatrix{R^{2(d+e)}(gf)_!(gf)^* \cF(d+e) \ar[rr]^-{\Tr_{gf}(\cF)}\ar[d]_{\simeq} && \cF\\
        R^{2e}g_!R^{2d}f_!f^*g^*\cF (d+e)\ar[rr]_-{R^{2e}g_!\Tr_f(g^*\cF)(e)} && R^{2d}f_!f^* \cF(d).\ar[u]_{\Tr_f(\cF)}}\]
    \item[(c)] (Normalization) If $f$ is finite flat of constant degree
        $n$, the composition
    \[\xymatrix{\cF\ar[r]^-a & f_* f^* \cF \ar[r]^{\simeq} & Rf_! f^* \cF \ar[r]^-{\Tr_f(\cF)} & \cF,} \]
    where $a$ is the adjunction, is multiplication by $n$.
  \end{itemize}
  Moreover, $\Tr_f(\cF)$ is functorial in $\cF$, and, if $f$ is \'etale and surjective, then
  the sequence
  \[g_!g^*\cF\xrightarrow{\Tr_{p_1}(g^*\cF)-\Tr_{p_2}(g^*\cF)} f_!f^* \cF \xrightarrow{\Tr_f(f^*\cF)} \cF \to 0\]
  is exact. Here $p_1$, $p_2$ are the projections $\cX\times_\cY \cX\to \cX$ and $g\colon \cX\times_\cY \cX\to \cY$.
\end{Proposition}

\begin{proof}
The construction of $\Tr_f$ being local for the \'etale topology on $\cY$
and on $\cX$, the proposition follows from the case of schemes \cite[XVIII
Th\'eor\`eme 2.9]{SGA4}.
\end{proof}

\begin{Remark}\label{s.Rshr=shr}
Let $j\colon \cX\to \cY$ be a separated \'etale representable morphism
between finite-type and finite-inertia Deligne-Mumford $S$-stacks. We have
two functors $\Mod(\cX,\Lambda)\to \Mod(\cY,\Lambda)$: $j_!$ from
Construction \ref{d.j} and the restriction $R^0j_!$ of $Rj_!$ from
Construction \ref{s.fshr}. The composite natural transformation
  \[\xymatrix{R^0j_!\ar@{=>}[r]^-{R^0j_! a} & Rj_!j^*j_! \ar@{=>}[r]^-{\Tr_j j_!} & j_!,}\]
where $a\colon \one\Rightarrow j^* j_!$ is the adjunction, is a natural
equivalence. In fact, by base change, we are reduced to the trivial case
where $\cY$ is the spectrum of a separably closed field. More generally, we
have a natural isomorphism $Rj_!\simeq j_!$ of functors $D(\cX,\Lambda)\to
D(\cY,\Lambda)$.
\end{Remark}

\begin{Proposition}\label{p.shradj}
Let $f\colon \cX\to \cY$ be a separated morphism between finite-type and
finite-inertia Deligne-Mumford $S$-stacks. Assume that $f$ has $m$-prime
inertia. Then $Rf_!$ admits a right adjoint $Rf^!\colon D(\cY,\Lambda)\to
D(\cX,\Lambda)$. In particular, for $K\in D(\cX,\Lambda)$ and $L\in
D(\cY,\Lambda)$, we have a canonical isomorphism
  \begin{equation}\label{e.shradj}
  \Hom_\cY(Rf_!K,L)\simto \Hom_\cX(K,Rf^! L),
  \end{equation}
  functorial in $K$ and $L$.
\end{Proposition}

If $f$ is finite, $Rf^!$ is isomorphic to the functor in Construction
\ref{d.dist}.

\begin{proof}
This is a formal consequence of Brown Representability Theorem \cite[Theorem
14.3.1 (ix)]{KSCat} and the fact that $Rf_!$ commutes with small direct
sums.  We may also repeat the construction of $Rf^!$ in \cite[XVIII
Th\'eor\`eme 3.1.4]{SGA4} as follows. Choose a decomposition $f\simeq pj\pi$
as in Proposition \ref{p.comp} and an integer $d$, upper bound of the
dimensions of the fibers of $f$. For $\cF\in \Mod(\cX,\Lambda)$, define
  \[f_!^\bullet\cF=p_* \tau_{\le 2d}\cC_\ell^\bullet(j_!\pi_*\cF),\]
  where $\cC_\ell^\bullet$ is the modified Godement resolution (Construction \ref{s.God}).
Then $f_!^\bullet\cF$ computes $Rf_! \cF$. For every $q$, the functor
$f_!^q\colon \Mod(\cX,\Lambda)\to \Mod(\cY,\Lambda)$ is exact and commutes
with small inductive limits, hence admits by Adjoint Functor Theorem a right
adjoint \cite[XVIII Lemme 3.1.3]{SGA4}
  \[f^!_q\colon \Mod(\cY,\Lambda)\to \Mod(\cX,\Lambda)\]
  that preserves injectives. The functor
  \[\tot f^!_\bullet\colon C(\Mod(\cY,\Lambda))\to C(\Mod(\cX,\Lambda))\]
  gives rise to a functor between homotopy categories, which has a right localization \cite[Theorem 14.3.1 (vi)]{KSCat}
  \[Rf^!\colon D(\cY,\Lambda)\to D(\cX,\Lambda).\]
  For $K\in C(\Mod(\cX,\Lambda))$ and $M\in C(\Mod(\cY,\Lambda))$, adjunction induces an
  isomorphism of complexes
  \begin{equation}\label{e.shradj0}
  \Hom^\bullet_\cY(\tot f_!^\bullet K, M)\to \Hom^\bullet_\cX(K,\tot f^!_\bullet M).
  \end{equation}
  It follows that $\tot f^!_\bullet$ preserves homotopically injective complexes.
  The isomorphism \eqref{e.shradj} is induced by \eqref{e.shradj0},
  where $M$ is a homotopically injective resolution of $L$.
\end{proof}

\begin{Proposition}
Let $f\colon \cX\to \cY$ be as in Proposition \ref{p.shradj}, let $K\in
D(\cX,\Lambda)$, and let $L,M\in D(\cY,\Lambda)$. We have canonical
isomorphisms
\begin{gather}\label{e.shradj1}
  R\cHom_{\cY}(Rf_! K, L) \simto Rf_* R\cHom_{\cX}(K, Rf^! L),\\
  R\cHom_\cX(f^* M, Rf^! L) \simto Rf^!R\cHom_\cY(M,L),\label{e.shradj2}
\end{gather}
functorial in $K$, $L$ and $M$.
\end{Proposition}

\begin{proof}
  These are induced by the following isomorphisms
  \begin{multline*}
    \Hom_\cY(M,R\cHom_\cY(Rf_!K,L)) \simeq \Hom_\cY(M\otimes^L_\Lambda Rf_!K,L)\xrightarrow{A^{-1}} \Hom_\cY(Rf_!(f^*M\otimes^L_\Lambda K),L)\\
    \shoveright{\simeq \Hom_\cX(f^*M\otimes^L_\Lambda K,Rf^!L)
    \simeq \Hom_\cX(f^*M,R\cHom_\cX(K,Rf^!L))\simeq \Hom_\cY(M,Rf_*R\cHom_\cX(K,Rf^!L)),}\\
    \shoveleft{\Hom_\cX(K,R\cHom_\cX(f^*M,Rf^!L))\simeq \Hom_\cX(f^*M\otimes^L_\Lambda K,Rf^!L) \simeq\Hom_\cY( Rf_!(f^*M\otimes^L_\Lambda K),L)}\\
    \xrightarrow{A} \Hom_\cY(M\otimes^L_\Lambda Rf_!K,L)
    \simeq\Hom_\cY(Rf_!K,R\cHom_\cY(M,L)) \simeq \Hom_\cX(K,Rf^!R\cHom_\cY(M,L)),
  \end{multline*}
where $A$ is induced by projection formula $M\otimes^L_\Lambda Rf_! K \simto
Rf_!(f^* \otimes^L_\Lambda K)$ (Theorem \ref{p.bcc} \ref{p.bcc2}).
\end{proof}

The base change isomorphism in Theorem \ref{p.bcc} \ref{p.bcc1} induces the
following by adjunction.

\begin{Proposition}\label{p.usbc}
Let
    \[\xymatrix{\cX'\ar[r]^h \ar[d]_{f'} & \cX\ar[d]^f\\
  \cY'\ar[r]^g & \cY}\]
be a 2-Cartesian square of Deligne-Mumford $S$-stacks of finite type and
finite inertia, with $g$ separated of $m$-prime inertia. Let $M\in
D(\cX,\Lambda)$. Then the map
  \[Rf'_*Rh^! M \to Rg^! Rf_* M\]
  is an isomorphism.
\end{Proposition}

\begin{Construction}
Let $f\colon \cX\to \cY$ be a separated morphism of $m$-prime inertia
between finite-type finite-inertia Deligne-Mumford $S$-stacks, and let
$M,N\in D(\cY,\Lambda)$. We consider the map
\begin{equation}\label{e.pfshr}
f^*N\otimes^L_\Lambda Rf^! M \to Rf^!(N\otimes^L_\Lambda M)
\end{equation}
adjoint to the composite
\[Rf_!(f^*N\otimes^L_\Lambda Rf^! M)\simto N\otimes^L_\Lambda Rf_!Rf^! M \to N\otimes^L_\Lambda M,\]
where the first map is the inverse of \eqref{e.pfshr0} and the second map is
induced by the adjunction $Rf_!Rf^!M\to M$.

If the fibers of $f$ have dimension $\le d$, then we obtain a
map
\begin{equation}\label{e.Poincare}
f^*N(d)[2d]\to Rf^! N,
\end{equation}
composite of \eqref{e.pfshr} applied to $M=\Lambda$ and the map $f^*\Lambda
(d)[2d]\to Rf^!\Lambda$ adjoint to $\Tr_f(\Lambda)$. The map
\eqref{e.Poincare} is an isomorphism for $f$ \'etale and $d=0$. To see this,
we reduce to the case $f$ is \'etale and representable, where the assertion
follows from Remark \ref{s.Rshr=shr}.
\end{Construction}

\begin{Construction}\label{s.dag}
The construction of $Rf_!$  in Construction \ref{s.fshr} works without
assumption on the inertia of $f$ if we restrict to $D^+$. More precisely,
let $\cC_1$ be the 2-category whose objects are Deligne-Mumford $S$-stacks
of finite type and finite inertia and whose morphisms are the separated
morphisms, let $\cA$ be the arrowy 2-subcategory of $\cC_1$ whose morphisms
are the open immersions, let $\cB_1$ be the arrowy 2-subcategory of $\cC_1$
whose morphisms are the proper morphisms, and let $\cD$ be the 2-category of
triangulated categories. Let $F_\cA\colon \cA\to \cD$ be the pseudofunctor
given by Construction \ref{d.j}:
  \[\cX\mapsto D^+(\cX,\Lambda), \quad j\mapsto j_!, \quad \alpha\mapsto \alpha_!,\]
and let $F_{\cB_1}\colon \cB_1\to \cD$ be the pseudofunctor given by
  \[\cX\mapsto D^+(\cX,\Lambda), \quad p\mapsto Rf_*, \quad \alpha\mapsto R\alpha_*.\]
We define $G$ and $\rho$ as before. Note that the proper base change map
\eqref{e.fshr} is still an isomorphism in this case because we work in $D^+$. This defines an
object of $\GD^\Cart_{\cA,\cB_1}(\cC_1,\cD)$. Let $F\colon \cC_1\to \cD$
be a corresponding pseudofunctor. For any morphism $f\colon \cX\to \cY$ of
$\cC_1$, we define
  \[Rf_\dag \colon D^+(\cX,\Lambda)\to D^+(\cY,\Lambda)\]
to be $F(f)$. If $f$ has $m$-prime inertia, $Rf_\dag$ is isomorphic to the
restriction of $Rf_!$ to $D^+$. In general, $Rf_\dag$ is not the correct
definition of $Rf_!$, but in Section~\ref{s.6} we will use it to give the
correct definition of $Rf_!$ for coefficients of characteristic $0$.

For any morphism $f\colon \cX\to \cY$ of $\cC_1$, we have the
support-forgetting natural isomorphism of functors $D^+(\cX,\Lambda)\to
D^+(\cY,\Lambda)$
  \[Rf_\dag\Rightarrow Rf_*,\]
which is a natural isomorphism if $f$ is proper. Base change isomorphisms
are also constructed as before.
\end{Construction}

We have obtained the following proposition.

\begin{Proposition}\label{p.dagbc}
For any 2-Cartesian square of Deligne-Mumford stacks
  \[\xymatrix{\cX'\ar[r]^{h}\ar[d]_{f'} & \cX\ar[d]^{f}\\
  \cY'\ar[r]^{g} & \cY}\]
with $\cX$ and $\cY$ of finite type and finite inertia over $S$, $\cY'$ of
finite type and finite inertia over some Noetherian scheme $S'$, $f$
separated, the base change map
  \[g^* Rf_\dag M \to Rf'_\dag h^* M\]
is an isomorphism for all $M\in D^+(\cX,\Lambda)$.
\end{Proposition}

\begin{Proposition}\label{p.cdim}
Assume that $S$ is finite-dimensional and $\Lambda$ is annihilated by an
integer $m$ invertible on $S$. Let $f\colon \cX\to \cY$ be a morphism of
$m$-prime inertia between finite-type Deligne-Mumford $S$-stacks.
\begin{enumerate}
\item \label{p.cdim1} $Rf_*\colon D(\cX,\Lambda)\to D(\cY,\Lambda)$ has
    finite cohomological amplitude. Moreover, for $M\in D(\cX,\Lambda)$
    and $N\in D(\Lambda)$, the projection formula map \eqref{e.pf}
    \[N\otimes_\Lambda^L Rf_* M\to Rf_*(N\otimes^L_\Lambda M)\]
is an isomorphism. In particular, $Rf_*$ preserves complexes of
tor-dimension $\ge 0$.

\item \label{p.cdim2} Assume that $f$ is a closed immersion (resp.\ $f$ is
    separated and $\cX$, $\cY$ are of finite inertia). Then $Rf^!\colon
    D(\cY,\Lambda)\to D(\cX,\Lambda)$ has finite cohomological amplitude.
    Moreover, for $M\in D(\cY,\Lambda)$ and $N\in D(\Lambda)$ the map
    \eqref{e.pffinite} (resp.\ \eqref{e.pfshr})
    \[  N\otimes_\Lambda^L Rf^!M \to Rf^!(N\otimes^L_\Lambda M)
\]
is an isomorphism. In particular, $Rf^!$ carries $D_{\ft}$ to $D_{\ft}$.
\end{enumerate}
\end{Proposition}

Note that \eqref{e.pf} is not an isomorphism for $N\in D(\cY,\Lambda)$ in
general. Indeed, if $f$ an open immersion and $N$ is supported on the
complement of $\cX$, then $Rf_*(f^*N\otimes^L_\Lambda M)=0$ but
$N\otimes^L_\Lambda Rf_* M$ is not necessarily zero.

\begin{proof}
\begin{itemize}
  \item[\ref{p.cdim1}] The second assertion follows from the first one by
      an argument similar to the proof of Proposition \ref{p.pbc}
      \ref{p.pbc2}. To show the first assertion, up to replacing $\cY$ by
      an \'etale presentation and $\cX$ by the corresponding 2-pullback,
      we may assume that $\cY$ is a scheme. In particular, if $f$ is an
      open immersion, then the result follows from the case of schemes,
      which is a result of Gabber \cite[Corollary 1.4]{ILOXVIIIA}. In the
      general case, we proceed by induction. Let $j\colon\cU\to \cX$ be a
      dominant open immersion with $\cU$ separated over $S$ and let $i$ be
      a complementary closed immersion. Applying $Rf_*$ to \eqref{e.dist2}
      and induction hypothesis to $fi$, we are reduced to prove that
      $R(fj)_*$ has finite cohomological dimension. In other words, we may
      assume $\cX$ separated over $S$. By the factorization in Proposition
      \ref{p.comp}, we are then reduced to two cases: (a) $f$ is an open
      immersion; (b) $f$ is proper. Case (a) has been proven above and
      case (b) follows from Lemma \ref{p.cd}.

  \item[\ref{p.cdim2}] We easily reduce to two cases: (a) $f$ is a smooth
      morphism of schemes; (b) $f$ is a closed immersion. In case (a),
      $Rf^!\simeq f^*(d)[2d]$, where $d$ is the relative dimension, and
      the assertions are obvious. In case (b), the first assertion follows
      from \ref{p.cdim1} and the distinguished triangle \eqref{e.dist2},
      and the second assertion follows from the first one and an argument
      similar to the proof of Proposition \ref{p.pbc} \ref{p.pbc2}.
\end{itemize}
\end{proof}

\begin{Proposition}[Poincar\'e duality]\label{p.Poincare}
Assume that $\Lambda$ is annihilated by an integer $m$ invertible on $S$.
Let $f\colon \cX\to \cY$ be a separated morphism, smooth of relative
dimension $d$, of $m$-prime inertia between finite-type finite-inertia
Deligne-Mumford $S$-stacks. For every $N\in D(\cY,\Lambda)$, the map $f^*
N(d)[2d]\to Rf^! N$ \eqref{e.Poincare} is an isomorphism.
\end{Proposition}

\begin{proof}
We easily reduce to the case of schemes \cite[XVIII
      Th\'eor\`eme 3.2.5]{SGA4}.
\end{proof}

\section{Operations on $D_c(\cX,\Lambda)$}\label{s.3}

In this section, after reviewing the notion of constructible sheaves on
Deligne-Mumford stacks, we investigate the effects of the six operations
constructed in the previous section on constructible complexes. We then
compare them with operations constructed by duality.

In Definition \ref{s.constr} through Construction \ref{s.God}, let $\Lambda$
be a ring.

\begin{Definition}\label{s.constr}
We say a sheaf of $\Lambda$-modules $\cF$ on a Deligne-Mumford stack $\cX$
is \emph{constructible} if $\alpha^* \cF\in \Mod(X,\Lambda)$ is
constructible\footnote{For $\Lambda$ not necessarily Noetherian, we use the
definition of constructible sheaves of $\Lambda$-modules on a scheme $X$
given in \cite[IX D\'efinition 2.3]{SGA4}, not the one suggested in the
footnote at that place. In other words, we say that $\cG\in \Mod(X,\Lambda)$
is constructible if for every affine open subscheme $U\subset X$, there
exists a finite partition of $U$ into constructible locally closed
subschemes $U=\bigcup U_i$ such that $\cG\res U_i$ is locally constant of
finite presentation for every $i$.} for some (or, equivalently, for every)
\'etale presentation $\alpha\colon X\to\cX$ where $X$ is a scheme.
\end{Definition}

The full subcategory $\Mod_c(\cX,\Lambda)$ of $\Mod(\cX,\Lambda)$ consisting
of constructible sheaves $\Lambda$-modules is stable under extensions and
cokernels.

\begin{Proposition}\label{p.ctriv} Let $f\colon \cX\to \cY$ be a morphism of Deligne-Mumford stacks.
\begin{enumerate}
\item $f^*\colon\Mod(\cY,\Lambda)\to \Mod(\cX,\Lambda)$ preserves
    constructible sheaves.

\item \label{p.cjshr} If $f$ is representable and \'etale, then $f_!\colon
    \Mod(\cX,\Lambda)\to \Mod(\cY,\Lambda)$ preserves constructible
    sheaves.
\end{enumerate}
\end{Proposition}

\begin{proof}
The first assertion is obvious. To prove the second assertion, replacing
$\cY$ by an \'etale presentation, we may assume $\cY$ is a scheme. Let
$g\colon X\to \cX$ be an \'etale presentation. For all $\cF\in
\Mod(\cX,\Lambda)$, we have an exact sequence
  \[h_! h^* \cF \to g_!g^* \cF\to \cF\to 0,\]
where $h$ is the morphism $X\times_\cX X\to \cX$. Thus we may assume $\cX$
is a scheme. The assertion is well-known in this case.
\end{proof}

\begin{Proposition}\label{p.c}
Let $\cX$ be a quasi-compact Deligne-Mumford stack. We let $\Etqc(\cX)$
denote the full subcategory of $\Et(\cX)$ consisting of Deligne-Mumford
stacks representable, \'etale and quasi-compact over $\cX$. Let $\cF$ be a
sheaf of $\Lambda$-modules $\cF$ on $\cX$. Then the following conditions are
equivalent:
  \begin{itemize}
   \item[(a)] $\cF$ is constructible;
   \item[(b)] For any epimorphism $\alpha\colon \cF_I=\bigoplus_{i\in
       I}\Lambda_{\cU_i,\cX}\to \cG$, where $\cU_i$ is an object of
       $\Etqc(\cX)$ for every $i\in I$, and $\cG$ is either $\cF$ or the
       kernel of an epimorphism $\Lambda_{\cV,\cX}\to \cF$, where $\cV$ is
       an object of $\Etqc(\cX)$, there exists a finite subset $J\subset
       I$ such that $\cF_J\hookrightarrow \cF_I \xrightarrow{\alpha} \cG$
       is an epimorphism, where $\cF_J=\bigoplus_{j\in
       J}\Lambda_{\cU_j,\cX}$;
   \item[(c)] $\cF$ is isomorphic to the cokernel of a map $\Lambda_{\cU,\cX} \to \Lambda_{\cV,\cX}$, where $\cU$ and $\cV$ are  objects of $\Etqc(\cX)$;
   \end{itemize}
\end{Proposition}

Here $\Lambda_{\cU,\cX}=j_!\Lambda_\cU$ is the sheaf defined in Construction
\ref{d.jt}, where $j\colon \cU\to\cX$.

\begin{proof}
Let $f\colon X\to \cX$ be a quasi-compact \'etale presentation where $X$ is
a scheme.

(c)$\implies$(a). This follows from Proposition \ref{p.ctriv} \ref{p.cjshr}.

(a)$\implies$(b). Note that $f^*\cG$ would satisfy the definition of
constructible sheaf if we replace ``of finite presentation'' in the
definition by ``of finite type''. Since $f^*\alpha \colon f^*\cF_I\to f^*
\cG$ is an epimorphism, there exists a finite subset $J\subset I$ such that
$f^*\cF_J\to f^*\cG$ is an epimorphism by the proof of \cite[IX Proposition
2.7]{SGA4}. Then $\cF_J\to \cG$ is an epimorphism.

(b)$\implies$(c). Note that $(\Lambda_{\cU,\cX})_{\cU\in \Etqc(\cX)}$ is a system of generators of $\Mod(\cX,\Lambda)$. 
  Thus there exists an epimorphism $\cF_I\to \cF$.  By (b), we may assume that $I$ is finite. Then $\cF_I\simeq\Lambda_{\cV,\cX}$, where $\cV=\coprod_{i\in J} \cU_i \in \Etqc(\cX)$. Applying the above argument to $\cG=\Ker(\Lambda_{\cV,\cX}\to \cF)$, we get an epimorphism $\Lambda_{\cU,\cX}\to \cG$, $\cU\in \Etqc(\cX)$.
\end{proof}

\begin{Corollary}\label{p.c1}
Let $\cX$ be a quasi-compact Deligne-Mumford stack, and let $\cF$ be a
constructible $\Lambda$-module on~$\cX$. The functor $\Hom_\cX(\cF,-)\colon
\Mod(\cX,\Lambda)\to \Lambda$ commutes with small inductive limits.
\end{Corollary}

This follows from Proposition \ref{p.c} and \cite[VI Th\'eor\`eme
1.23]{SGA4}.

\begin{Corollary}\label{p.c2}
Let $\cX$ be a quasi-compact Deligne-Mumford stack. Any $\Lambda$-module
$\cF$ on $\cX$ is a filtered inductive limit of constructible
$\Lambda$-modules on $\cX$.
\end{Corollary}

\begin{proof}
Fix an epimorphism $\bigoplus_{i\in I}\cF_i\to \cF$, where $\cF_i\in
\Mod_c(\cX,\Lambda)$, $i\in I$. For every finite subset $A\subset I$, fix a
surjection $\bigoplus_{j\in J_A}\cG_j\to \Ker(\cF_A\to \cF)$, where
$\cG_j\in \Mod_c(\cX,\Lambda)$, $\cF_A=\bigoplus_{i\in A}\cF_i$. Then $\cF$
is the inductive limit of $\Coker(\cG_B\to \cF_A)$, where $(A,B)$ runs over
pairs of finite subsets $A\subset I$, $B\subset J_A$, $\cG_B=\bigoplus_{j\in
B}\cG_j$. The pairs are ordered by inclusion.
\end{proof}

\begin{Corollary}\label{p.c3}
  Let $\cX$ be a quasi-compact Deligne-Mumford stack. The functor
  \[\varinjlim\colon \Ind\Mod_c(\cX,\Lambda)\to \Mod(\cX,\Lambda)\]
  is an equivalence of categories.
\end{Corollary}

This follows from Corollaries \ref{p.c1} and \ref{p.c2}.

\begin{Construction}\label{s.God}
As in \cite[XVIII 3.1.2]{SGA4}, Corollary \ref{p.c3} allows us to construct
the modified Godement resolution. Let $\cX$ be a quasi-compact
Deligne-Mumford stack. For any sheaf $\cF$ of $\Lambda$-modules on $\cX$, we
let $\cC^\bullet(\cF)$ denote the Godement resolution of $\cF$ and we define
the modified Godement resolution to be
\[\cC_\ell^\bullet(\cF)=\varinjlim_{i\in I}\cC^\bullet(\cF_i),\]
where $(\cF_i)_{i\in I}$ is a system of constructible $\Lambda$-modules such
that $\cF\simeq \varinjlim_{i\in I}\cF_i$. This is a flabby resolution of
$\cF$, independent of the choice of $(\cF_i)_{i\in I}$ up to isomorphism.
For all $q$, the functor $\cC_\ell^q(\cF)$ is exact and commutes with small
inductive limits and \'etale localization.
\end{Construction}

\begin{Definition}\label{d.thick}
Following \cite[Definition 8.3.21 (iv)]{KSCat}, we say that an additive full
subcategory of an Abelian category $\cA$ is \emph{thick} if it is closed
under kernels, cokernels, and extensions. A \emph{Serre subcategory} of
$\cA$ is a thick subcategory of $\cA$ closed under sub-objects and
quotients. We say that a triangulated subcategory of a triangulated category
is \emph{thick} if it is stable under direct summand.
\end{Definition}

Thick subcategories of Abelian categories are called ``complete
subcategories'' and Serre subcategories are called ``thick subcategories''
in \cite[1.11]{Tohoku}. By \cite[Proposition 1.3]{Rickard}, the above
definition of thick subcategories of triangulated subcategories coincides
with \cite[C.D., D\'efinition I.2.1.1]{SGA4d}. If $\cB$ is a thick
subcategory of an Abelian category $\cA$, then the full subcategory
$D_\cB(\cA)$ of $D(\cA)$ spanned by complexes with cohomology in $\cB$ is a
thick subcategory.

In Propositions \ref{p.cN} through \ref{p.RHomft}, let $\Lambda$ be a
Noetherian ring. In this case, $\Mod_c(\cX,\Lambda)$ is a thick subcategory
of $\Mod(\cX,\Lambda) $. Let $D_c(\cX,\Lambda)$ be the triangulated
subcategory of $D(\cX,\Lambda)$ consisting of complexes with constructible
cohomology sheaves, which is a thick subcategory. We put
$\Dcft(\cX,\Lambda)=D_c(\cX,\Lambda)\cap D_{\ft}(\cX,\Lambda)$.

\begin{Proposition}\label{p.cN}
Let $\cX$ be a Noetherian Deligne-Mumford stack. A sheaf of
$\Lambda$-modules $\cF$ on $\cX$ is constructible if and only if it is
Noetherian.
\end{Proposition}

It follows that $\Mod_c(\cX,\Lambda)$ is a Serre subcategory of
$\Mod(\cX,\Lambda)$ in this case.

\begin{proof}
Let $f\colon X\to \cX$ be a quasi-compact \'etale presentation where $X$ is
a scheme. If $\cF$ is constructible, then $f^*\cF$ is constructible, hence
Noetherian \cite[IX Proposition 2.9 (ii)]{SGA4}. It follows that $\cF$ is
Noetherian. If $\cF$ is Noetherian, then the sheaf $\cG$ in Proposition
\ref{p.c} (b) is Noetherian by assumption (if $\cG=\cF$) or by the above (if
$\cG=\Ker(\Lambda_{\cV,\cX}\to \cF)$), thus $\cF$ is constructible by
Proposition \ref{p.c} (b)$\implies$(a).
\end{proof}

The following is obvious.

\begin{Proposition}\label{p.tensor}
Let $\cX$ be a Deligne-Mumford stack. The functor
  \[-\otimes^L_\Lambda
      -\colon D(\cX,\Lambda)\times D(\cX,\Lambda) \to D(\cX,\Lambda)
      \]
  sends $D^-_c\times D^-_c$ to $D^-_c$.
\end{Proposition}

\begin{Proposition}\label{p.cproper}
Let $f\colon \cX\to \cY$ be a proper finitely-presented morphism of
Deligne-Mumford stacks. Assume $\Lambda$ is annihilated by an integer $m$.
then the functor
      $Rf_*\colon D(\cX,\Lambda)\to D(\cY,\Lambda)$ carries
      $D^+_c(\cX,\Lambda)$ to $D^+_c(\cY,\Lambda)$. If, moreover, $f$ has
      $m$-prime inertia, then $Rf_*$ carries $D_c(\cX,\Lambda)$ to
      $D_c(\cY,\Lambda)$.
\end{Proposition}

\begin{proof}
For the first assertion, we reduce by cohomological descent to the
    case of schemes, which is \cite[XIV Th\'eor\`eme 1.1]{SGA4}. The second assertion
    follows from the first one.
\end{proof}

Deligne's theorem on generic constructibility and generic base change
extends to Deligne-Mumford stacks by cohomological descent as follows. See
\cite[Proposition 2.12]{Quillen} for the case of Artin stacks.

\begin{Proposition}\label{p.gbc}
Let $\cS$ be a Noetherian Deligne-Mumford stack. Assume that $\Lambda$ is
annihilated by an integer $m$ invertible on $\cS$. Let $f\colon \cX\to \cY$
be a morphism between finite-type Deligne-Mumford $\cS$-stacks and let $M\in
D^+_c(\cX,\Lambda)$. For every integer $i$, there exists a dense open
substack $\cS^\circ$ of $\cS$ satisfying the following conditions
\begin{enumerate}
\item The restriction of $R^if_*M$ to $\cY\times_\cS \cS^{\circ}$ is
    constructible.

\item $R^if_* M $ is compatible with arbitrary base change $\cS'\to
    \cS^\circ\subset \cS$ of Deligne-Mumford stacks.
\end{enumerate}
\end{Proposition}

\begin{Proposition}
Let $S$ be a Noetherian scheme. Assume that $\Lambda$ is annihilated by an
integer $m$ invertible in $S$. let $f\colon \cX\to \cY$ be a separated
morphism of $m$-prime inertia between finite-type and finite-inertia
Deligne-Mumford $S$-stacks. Then $Rf_\dag$ carries $D^+_c(\cX,\Lambda)$ to
$D^+_c(\cY,\Lambda)$. If, moreover, $f$ has $m$-prime inertia, then $Rf_!$
carries $D_c(\cX,\Lambda)$ to $D_c(\cY,\Lambda)$.
\end{Proposition}

\begin{proof}
By construction we are reduced to two cases: (a) $f$ is an open immersion;
(b) $f$ is proper. Case (a) is clear while (b) is a special case of
Proposition \ref{p.cproper}.
\end{proof}

\begin{Proposition}\label{p.RHomft}
Let $\cX$ be a finite-dimensional Noetherian Deligne-Mumford stack. Assume
that $\Lambda$ is annihilated by an integer $m$ invertible on $\cX$. Then
  \[R\cHom_{\cX}\colon D(\cX,\Lambda)^\op\times D(\cX,\Lambda) \to D(\cX,\Lambda)\]
  carries $(\Dcft)^\op \times D_{\ft}$ to $D_{\ft}$.
\end{Proposition}

\begin{proof}
As in \cite[Th.\ finitude, Remarque 1.7]{SGA4d}, this follows from
Proposition \ref{p.cdim} \ref{p.cdim1}.
\end{proof}

In Propositions \ref{p.cX} and \ref{p.cf}, let $S$ be either a regular
scheme of dimension $\le 1$ or a quasi-excellent scheme, and let $\Lambda$
be a Noetherian ring annihilated by an integer $m$ invertible on $S$. Under
these assumptions, we can apply finiteness results of Deligne \cite[Th.\
finitude]{SGA4d} or Gabber \cite{ILOXIII}. In the second case, one reduces
to the case $\Lambda=\Z/m\Z$ using arguments similar to \cite[Th.\ finitude,
2.2]{SGA4d}.

\begin{Proposition}\label{p.cX}
  Let $\cX$ be a Deligne-Mumford $S$-stack of finite type. Then
  \[R\cHom_{\cX}\colon D(\cX,\Lambda)^\op\times D(\cX,\Lambda) \to D(\cX,\Lambda)\]
  carries $(D^-_c)^\op \times D^+_c$ to $D^+_c$.
\end{Proposition}

\begin{proof}
Take an \'etale presentation $X\to \cX$ with $X$ a separated finite type
$S$-scheme. It suffices to verify the assertions for $R\cHom_X$.  This is
\cite[Th.\ finitude, Corollaire 1.6]{SGA4d} for $S$ regular of dimension
$\le 1$. For $S$ quasi-excellent, the same proof allows us to deduce this
from \cite[1.1]{ILOXIII}.
\end{proof}

\begin{Proposition}\label{p.cf}
Let $f\colon\cX\to \cY$ be a morphism between finite-type Deligne-Mumford $S$-stacks.
\begin{enumerate}
\item \label{p.cf1} $Rf_*$ sends $D^+_c(\cX,\Lambda)$ to
    $D^+_c(\cY,\Lambda)$. If $f$ has $m$-prime inertia, then $Rf_*$ sends
    $D^b_c(\cX,\Lambda)$ to $\Dbc(\cY,\Lambda)$. If, moreover, $S$ is
    finite-dimensional, then $Rf_*$ sends $D_c(\cX,\Lambda)$ to
    $D_c(\cY,\Lambda)$.

\item \label{p.cf3} Assume that $f$ is a closed immersion (resp.\ $f$ is
    separated, and $\cX$ and $\cY$ are of finite inertia). Then
    $Rf^!\colon D(\cY,\Lambda)\to D(\cX,\Lambda)$ sends
    $D^+_c(\cY,\Lambda)$ to $D^+_c(\cX,\Lambda)$ and $\Dbc(\cY,\Lambda)$
    to $\Dbc(\cX,\Lambda)$. If, moreover, $S$ is finite-dimensional, then
    $Rf^!$ sends $D_c(\cY,\Lambda)$ to $D_c(\cX,\Lambda)$.
\end{enumerate}
\end{Proposition}

\begin{proof}
We will freely use Proposition \ref{p.cdim} to extend results to $D_c$.
\begin{itemize}
\item[\ref{p.cf1}] The case of schemes was proved by Deligne \cite[Th.\
    finitude, Remarque 1.3]{SGA4d} and Gabber \cite[1.1]{ILOXIII}. Up to
    replacing $\cY$ by an \'etale presentation and $\cX$ by the
    corresponding pullback, we may assume that $\cY$ is a scheme. For the
    first assertion, we then reduce by cohomological descent to the known
    case where $\cX$ is also a scheme. For the second assertion, as in the
    proof of Proposition \ref{p.cdim} \ref{p.cdim1}, we reduce first to
    the case where $\cX$ is separated, and then to the cases: (a) is an
    open immersion; (b) $f$ proper. Case (a) follows from the case of
    schemes and case (b) follows from Lemma \ref{p.cd}.

\item[\ref{p.cf3}] We easily reduce to two cases: (a) $f$ is a smooth
    morphism of schemes; (b) $f$ is a closed immersion. In case (a),
    $Rf^!\simeq f^*(d)[2d]$, where $d$ is the relative dimension of $f$,
    and the assertions are trivial. In case (b), the assertions follow
    from \ref{p.cf1} and the distinguished triangle \eqref{e.dist2}.
\end{itemize}
\end{proof}

In the rest of this section let $S$ be either (a) a regular scheme of
dimension $\le 1$ or (b) a finite-dimensional excellent scheme, endowed with
a dimension function $\delta_S$. Recall that every integral universally
catenary scheme admits $x\mapsto -\dim(\cO_{X,x})$ as a dimension function,
and the difference between two dimension functions on a scheme is a locally
constant function. To state results on cohomological dimension more
conveniently, we assume $\delta_S(s)\ge 0$ for every $s\in S$. Let $\Lambda$
be a Noetherian ring annihilated by an integer $m$ invertible on $S$.

\begin{Construction}\label{c.D}
In case (a), let $\Omega_S\in D^b_c(S,\Lambda)$ be the object such that
$\Omega_S\res T=\Lambda(d_T)[2d_T]$ for every connected component $T$ of
$S$, where $d_T=\max (\delta_S\res T)$. In case (b), let $\Omega_S\in
\Dcft(S,\Lambda)$ be the potential dualizing complex associated to
$(S,\delta_S)$ in the sense of \cite[D\'efinition 2.1.2]{ILOXVII}, which
exists and is unique up to isomorphism by \cite[Th\'eor\`emes 5.1.1,
6.1.1]{ILOXVII}\footnote{The definition \cite[D\'efinition 2.1.2]{ILOXVII},
existence and uniqueness \cite[Th\'eor\`eme 5.1.1]{ILOXVII} are stated for
$\Z/m\Z$, but can be extended to $\Lambda$, as observed by Jo\"el Riou
(private communication). If $K$ is a potential dualizing complex with
coefficients in $\Z/m\Z$, then $K\otimes^L_{\Z/m\Z} \Lambda$ is a potential
dualizing complex with coefficients in $\Lambda$ by Proposition \ref{p.cdim}
\ref{p.cdim2}.}. For any finite-type and separated morphism $a\colon X\to S$
of schemes, let $\Omega_X=Ra^! \Omega_S \in \Dcft(X,\Lambda)$. We define a
triangulated functor
\[D_X\colon D(X,\Lambda)^\op \to D(X,\Lambda)\]
by $D_X M = R\cHom_\Lambda(M,\Omega_X)$. We endow $X$ with the dimension
function $\delta_X(x)=\delta_S(f(x))+\trdeg(x/f(x))$ and we take $d_X=\max
\delta_X =\max_{s\in S} (\dim
 X_s +\delta_S(s))$.
\end{Construction}

\begin{Lemma}\label{p.Dscheme}\leavevmode
\begin{enumerate}
\item \label{p.Dscheme1}  The functor $D_X$ restricts to functors
    $(D^-_c)^\op \to D^+_c$ and $\Dcft^\op \to \Dcft$. Moreover, for $M\in
    \Dcft(X,\Lambda)$ of tor-amplitude contained in $[a,b]$, $D_X M$ has
    tor-amplitude contained in $[-b-2d_X,-a]$.

\item \label{p.Dscheme2} For $M\in \Dcft(X,\Lambda)$, the natural map
    $M\to D_XD_XM$ is an isomorphism.
\end{enumerate}
\end{Lemma}

In particular, $\Omega_X$ has tor-amplitude contained in $[-2d_X,0]$ and
$R\cHom(\Omega_X, \Omega_X) = D_X \Omega_X \simeq \Lambda$. The complex
$\Omega_X$ is a dualizing complex in the sense of \cite[D\'efinition
7.1.1]{ILOXVII} and, for $\Lambda=\Z/m\Z$, also in the sense of \cite[I
D\'efinition 1.7]{SGA5}.

\begin{proof}
\begin{itemize}
\item[\ref{p.Dscheme1}] It follows from Proposition \ref{p.cX} that $D_X$
    sends $(D^-_c)^\op$ to $D^+_c$. The bound can be obtained similarly to
    the proof of \cite[Th.\ finitude, Corollaire 1.6]{SGA4d}. In fact, we
    may assume $M=j_! L$, where $j\colon Y\to X$ is an immersion, $Y$ is
    regular and connected and $L\in \Dcft(Y,\Lambda)$ has tor-amplitude
    contained in $[a,b]$ and locally constant cohomological sheaves. Then
    $D_X j_! L\simeq Rj_* D_Y L$ by \eqref{e.jshradju} and
  \[(D_Y L)_{\bar y} = R\cHom_\Lambda(L_{\bar y}, \Lambda(d_Y)[2d_Y])\]
  for every geometric point $\bar{y}\to Y$. We then apply the fact that
  $Rj_*$ has cohomological dimension $\le \max\{2d_Y-1,0\}$ \cite[Theorem
  1.1]{ILOXVIIIA}.

\item[\ref{p.Dscheme2}] In case (a), this is \cite[Th.\ finitude,
    Th\'eor\`eme 4.3]{SGA4d}. In case (b), \cite[Proposition
    4.1.2]{ILOXVII} (either extended to $\Lambda$ or combined with
    Proposition \ref{p.cdim} \ref{p.cdim2}) shows that $\Omega_X$ is a
    potential dualizing complex for $(X,\delta_X)$. It then suffices to
    apply \cite[Th\'eor\`emes 6.1.1, 7.1.3]{ILOXVII}.
\end{itemize}
\end{proof}

\begin{Construction}\label{c.DX}
Let $\cX$ be a finite-type Deligne-Mumford $S$-stack. We apply the
Be{\u\i}linson-Bernstein-Deligne gluing theorem \cite[Th\'eor\`em
3.2.4]{BBD} to define a dualizing complex $\Omega_\cX$. For every \'etale
morphism $\alpha\colon X\to \cX$ with $X$ a separated finite-type
$S$-scheme, we associate $\Omega_X$. For any morphism $f\colon X\to Y$
between such morphisms, there is a canonical isomorphism $f^* \Omega_Y
\simto \Omega_X$. Since, for all~$X$, $\Omega_X$ belongs to
$D^{[-2d_\cX,0]}(X,\Lambda)$, where $d_\cX=\max_{s\in S}(\dim
\cX_s+\delta_S(s))$, there exists an object $\Omega_\cX\in
D^{[-2d_\cX,0]}(\cX,\Lambda)$, unique up to isomorphism, such that, for
every \'etale morphism $\alpha\colon X\to \cX$ with $X$ a separated
finite-type $S$-scheme, we have $\alpha^*\Omega_\cX \simeq \Omega_X$. It
follows that $\Omega_\cX$ belongs to $D_c$ and has tor-amplitude contained
in $[-2d_\cX,0]$. We define a triangulated functor
\[D_\cX\colon D(\cX,\Lambda)^\op \to D(\cX,\Lambda)\]
by $D_\cX M = R\cHom_\Lambda(M,\Omega_\cX)$.
\end{Construction}

\begin{Proposition}\label{p.D}\leavevmode
\begin{enumerate}
\item \label{p.D1} The functor $D_\cX$ restricts to functors $(D_c^-)^\op
    \to D_c^+$ and $\Dcft^\op\to \Dcft$. Moreover, for $M\in
    \Dcft(\cX,\Lambda)$ of tor-amplitude contained in $[a,b]$, $D_{\cX} M$
    has tor-amplitude contained in $[-b-2d_{\cX},-a]$.

\item \label{p.D2} For $M\in \Dcft(\cX, \Lambda)$, the natural map $M\to
    D_\cX D_\cX M$ is an isomorphism.
\end{enumerate}
\end{Proposition}

\begin{proof}
Take an \'etale presentation $X\to \cX$ with $X$ a separated finite
 type $S$-scheme. It suffices to verify the proposition for $X$, which
 is Lemma \ref{p.Dscheme}.
\end{proof}

\begin{Proposition}\label{p.GD}
Assume that $\Lambda$ is Gorenstein of dimension 0.
\begin{enumerate}
\item \label{p.GD1} The functor $D_\cX\colon (D_c^-)^\op \to D_c^+$ has
    cohomological amplitude contained in $[-2d_\cX, 0]$. In particular
    $D_\cX$ sends $(D^b_c)^\op$ to $D^b_c$.

\item \label{p.GD2} For $M\in \Dbc(\cX, \Lambda)$, the natural map $M\to
    D_\cX D_\cX M$ is an isomorphism.
\end{enumerate}
\end{Proposition}

The additional assumption on $\Lambda$ in Proposition \ref{p.GD} is
satisfied notably if $\Lambda=\cO/\fm^{n+1}$, where $\cO$ a discrete
valuation ring of and $\fm$ is the maximal ideal of~$\cO$.

\begin{proof}
It suffices to combine the proofs of Lemma \ref{p.Dscheme} and Proposition
\ref{p.D} with \cite[Th.\ finitude, 4.7]{SGA4d} and \cite[Th\'eor\`eme
7.1.2]{ILOXVII}.
\end{proof}

\begin{Remark}\label{s.fushr}
Let $f\colon \cX\to \cY$ be a separated morphism of $m$-prime inertia
between finite-type Deligne-Mumford $S$-stacks. Assume either $f$ is a
closed immersion, or $\cX$ and $\cY$ are of finite inertia. The functor
$Rf^!\colon D(\cY,\Lambda)\to D(\cX,\Lambda)$ as defined in Construction
\ref{d.dist} and Proposition \ref{p.shradj} preserves $\Dcft$. The
isomorphisms \eqref{e.dist4} and \eqref{e.shradj2} give an isomorphism
$D_\cX f^* \simto R f^! D_\cY $ of functors $D(\cY,\Lambda)\to
D(\cX,\Lambda)$. Applying biduality  (Proposition \ref{p.D} \ref{p.D2}), we
obtain an isomorphism of functors $\Dcft(\cY,\Lambda)\to
\Dcft(\cX,\Lambda)$:
\[R f^! \simeq R f^! D_\cY D_\cY \simeq D_\cX f^* D_\cX.\]
\end{Remark}

\begin{Construction}\label{d.fushriek}
Let $f\colon \cX\to\cY$ be a morphism between finite-type Deligne-Mumford
$S$-stacks. Thanks to Proposition \ref{p.D}, we can  define a triangulated
functor
\[Rf^!\colon \Dcft(\cY,\Lambda)\to \Dcft(\cX,\Lambda), \quad N \mapsto D_\cX f^* D_\cY N.\]
By Remark \ref{s.fushr}, this definition is compatible with Construction
\ref{d.dist} and Proposition \ref{p.shradj}. If $f\colon \cX\to \cY$ and
$g\colon \cY\to \cZ$ are two such morphisms, then, by biduality, we have an
isomorphism of functors $\Dcft(\cZ,\Lambda)\to \Dcft(\cX,\Lambda)$:
\[
R(gf)^! = D_\cX (gf)^* D_\cZ  \simeq D_\cX f^* g^* D_\cZ
 \simeq D_\cX f^* D_\cY D_\cY g^* D_\cZ  = Rf^! Rg^!.
\]

For $L,M\in \Dcft(\cY,\Lambda)$, we have
\begin{multline*}
R\cHom_\cX(f^*M,Rf^!L)\simeq R\cHom_\cX(f^*M,D_{\cX}D_{\cX}Rf^!L) \simeq D_\cX(f^*M\otimes^L_\Lambda D_\cX Rf^! L)\simeq D_\cX(f^*M \otimes^L_\Lambda f^*D_\cY L)\\
\simeq D_\cX f^*(M\otimes^L_\Lambda D_\cY L)
\simeq Rf^!D_\cY(M\otimes^L_\Lambda D_\cY L) \simeq Rf^!R\cHom_\cY(M,D_\cY D_\cY L)\simeq Rf^!R\cHom_\cY(M,L).
\end{multline*}

If $f$ is smooth, it follows from the construction of $\Omega_\cX$ and $\Omega_\cY$ that $\Omega_\cX\simeq f^*\Omega_\cY(d)[2d]$, where $d$ is the relative dimension of $f$. It follows that \eqref{e.Homsm} induces an isomorphism $f^*(D_\cY N)(d)[2d]\simeq D_\cX f^*N$ for all $N\in D(\cX,\Lambda)$. Thus
\[Rf^!=D_\cX f^* D_\cY \simeq f^*(d)[2d].\]
\end{Construction}

\begin{Remark}\label{s.flshr}
Let $f\colon\cX\to \cY$ be a morphism of $m$-prime inertia between
finite-type Deligne-Mumford $S$-stacks. Assume either $f$ is representable
and \'etale, or $f$ is separated and $\cX$ and $\cY$ are of finite inertia.
The functor $Rf_!\colon D(\cX, \Lambda) \to D(\cY,\Lambda)$ as defined in
Constructions \ref{d.j} and \ref{s.fshr} preserves $\Dcft$. The isomorphisms
\eqref{e.jshradju} and \eqref{e.shradj1} give an isomorphism $D_\cY Rf_!
\simto Rf_* D_\cX$ of functors $D(\cX,\Lambda)\to D(\cY,\Lambda)$. Applying
biduality (Proposition \ref{p.D} \ref{p.D2}), we obtain an isomorphism of
functors $\Dcft(\cX,\Lambda)\to \Dcft(\cY,\Lambda)$:
\[Rf_!\simeq D_\cY D_\cY Rf_! \simeq D_\cY Rf_* D_\cX.\]
\end{Remark}

\begin{Construction}\label{d.flshriek}
Let $f\colon \cX\to \cY$ be a morphism of $m$-prime inertia between
finite-type Deligne-Mumford $S$-stacks. Thanks to Proposition \ref{p.D}, we
can define a triangulated functor
\[Rf_!\colon \Dcft(\cX,\Lambda)\to \Dcft(\cY,\Lambda), \quad M \mapsto D_\cY Rf_* D_\cX M.\]
By Remark \ref{s.flshr}, this definition is compatible with Constructions
\ref{d.j} and \ref{s.fshr}. If $f\colon \cX\to \cY$ and $g\colon \cY\to \cZ$
are two such morphisms, then, by biduality, we have an isomorphism of
functors $\Dcft(\cX,\Lambda)\to \Dcft(\cY,\Lambda)$:
\[
  R(gf)_! = D_\cZ R(gf)_* D_\cX \simeq D_\cZ Rg_* Rf_* D_\cX
  \simeq D_\cZ Rg_* D_\cY D_\cY Rf_* D_\cX \simeq Rg_! Rf_!.
\]
\end{Construction}

\begin{Remark}
If $\Lambda$ is Gorenstein of dimension 0, then Remark \ref{s.fushr} through
Construction \ref{d.flshriek} hold with $\Dcft$ replaced by $\Dbc$.
\end{Remark}

\section{Construction of $D_c(\cX,\cO)$}\label{s.topos}
In this section, we develop an $\ell$-adic formalism for general topoi.
Previous work on $\ell$-adic formalism includes Ekedahl \cite{Ekedahl},
Behrend \cite[2.2]{Behrend} for general topoi, Deligne \cite[1.1.2]{WeilII}
for \'etale topoi of algebraic spaces, and Laszlo-Olsson \cite{LO2} for
lisse-\'etale topoi of Artin stacks. Our methods are closely related to
\cite{Ekedahl} and \cite{LO2}. We fix a ring $\cO$ and a principal ideal
$\fm=\lambda\cO$ generated by a non-zero-divisor $\lambda$ such that $\cO\to
\varprojlim_{n\in \N}\Lambda_n$ is an isomorphism. Here $\N$ is the ordered
set of nonnegative integers and $\Lambda_n=\cO/\fm^{n+1}$.

\begin{Construction}
Let $\cX$ be a topos. Consider the topos $\cX^\N$ of projective systems
$(M_\bullet = (M_n)_{n\in \N})$ of sheaves on $\cX$ and the ring
$\Lambda_\bullet=(\Lambda_n)_{n\in \N}\in \cX^\N$ whose transition maps
$\Lambda_{n+1}\to\Lambda_n$ are induced by the identity map on $\cO$.
Consider the morphism of topoi $(\pi_*,\pi^{-1})\colon \cX^\N\to \cX$ with
$\pi_*(F_\bullet)=\varprojlim_n F_n$ such that $\pi^{-1}G$ is constant of
value $G$. The projection map $\pi^{-1}\cO\to \Lambda_\bullet$ then induces
a morphism of ringed topoi $\pi=(\pi_*, \pi^*)\colon
(\cX^\N,\Lambda_\bullet) \to (\cX,\cO)$, with
$\pi^*G=(\Lambda_n\otimes_{\cO} G)_{n\in \N}$.

For all $n$, let $(e_{n*},e_n^{-1})\colon \cX \to \cX^\N$ be the morphism of
topoi defined by $e_n^{-1}(G_\bullet) = G_n$, $(e_{n*} F)_q=F$ for $q\ge n$,
and $(e_{n*}F)_q=\{*\}$ for $q<n$. The functor $e_n^{-1}$ admits a left
adjoint $e_{n!}$ given by $(e_{n!}F)_q=F$ for $q\le n$ and
$(e_{n!}F)_q=\emptyset$ for $q>n$. The morphism of topoi $(e_{n*},e_n^{-1})$
induces a flat morphism of ringed topoi $e_n=(e_{n*},e_n^{-1})\colon (\cX,
\Lambda_n)\to (\cX^\N,\Lambda_\bullet)$. Note that $e_{n*}\colon
\Mod(\cX,\Lambda_n)\to \Mod(\cX^\N,\Lambda_\bullet)$ is exact and
$e_n^{-1}e_{n*}=\one$. For $M\in D(\cX^\N, \Lambda_\bullet)$, we denote
$e_n^{-1}M$ by $M_n$. The functor $e_n^{-1}\colon
\Mod(\cX^\N,\Lambda_\bullet)\to \Mod(\cX,\Lambda_n)$ admits a left adjoint,
still denoted by $e_{n!}$, given by
$(e_{n!}F)_q=\Lambda_q\otimes_{\Lambda_n} F$ for $q\le n$ and
$(e_{n!}F)_q=0$ for $q>n$. We have adjoint pairs of functors
$(Le_{n!},e_n^{-1})$ and $(e_n^{-1},e_{n*})$ between $D(\cX,\Lambda_n)$ and
$D(\cX^\N,\Lambda_\bullet)$ (see \cite[Theorem 14.4.5]{KSCat}). The
composition $\pi e_n\colon (\cX,\Lambda_n) \to (\cX,\cO)$ is the ring change
morphism and for brevity we will sometimes omit $(\pi e_n)_*$ from the
notation.
\end{Construction}

\begin{Definition}\label{s.null}
We say that $M\in\Mod(\cX^\N,\pi^{-1}\cO)$ is \emph{strict} if $M_{n+1}\to
M_n$ is an epimorphism for every $n\ge 0$. We say that
$M\in\Mod(\cX^\N,\pi^{-1}\cO)$ is \emph{essentially zero} if for every $n\ge
0$, there exists $r\ge 0$ such that the transition map $M_{n+r}\to M_n$ is
zero.  We say that $M$ is \emph{AR-null} \cite[V D\'efinition 2.2.1]{SGA5}
if, moreover, $r$ can be chosen independently of $n$. This applies to the
full subcategory $\Mod(\cX^\N,\Lambda_\bullet)\subset
\Mod(\cX^\N,\pi^{-1}\cO)$. The full subcategories
\[\cN \subset \Mod(\cX^\N,\Lambda_\bullet), \quad \cN'\subset \Mod(\cX^\N,\pi^{-1}\cO)\]
spanned by AR-null modules are Serre subcategories (Definition
\ref{d.thick}). We say that $M\in \Mod(\cX^\N,\Lambda_\bullet)$ is
\emph{preadic} (``$\fm$-adic'' in the terminology of \cite[V D\'efinition
3.1.1]{SGA5}) if, for all $n$, the natural map $\Lambda_n
\otimes_{\Lambda_{n+1}} M_{n+1} \to M_n$ is an isomorphism; $M\in
\Mod(\cX^\N,\Lambda_\bullet)$ is \emph{AR-preadic} (``AR-$\fm$-adic'' in the
terminology of \cite[V D\'efinition 3.2.2]{SGA5}) if there exists a preadic
module $N\in \Mod(\cX^\N,\Lambda_\bullet)$ and a homomorphism $N\to M$ with
AR-null kernel and cokernel. We let $\Mod_{pa}(\cX^\N,\Lambda_\bullet)$
(resp.\ $\Mod_{AR\text{-}pa}(\cX^\N,\Lambda_\bullet)$) denote the full
subcategory of $\Mod(\cX^\N,\Lambda_\bullet)$ spanned by the preadic (resp.\
AR-preadic) objects.
\end{Definition}

\begin{Notation}
Let $M\in \Mod(\cX^\N,\pi^{-1}\cO)$. For any integer $r$, we let $L(r)$
denote the translation of $L$ given by $L(r)_n=L_{r+n}$ for $r+n\ge 0$ and
$L(r)_n=0$ for $r+n<0$. For $r\ge 0$ (resp.\ $r\le 0$), the transition map
$L(r)\to L$ (resp.\ $L\to L(r)$) has AR-null kernel and cokernel. We have
$L(r)(r')=L(r+r')$ if $r\le0$ or $r'\ge 0$.
\end{Notation}

We refer the reader to \cite[1.11]{Tohoku} for the construction of quotient
categories of Abelian categories by Serre subcategories and to
\cite[Chap.~I]{GZ} for more general categories of fractions.

\begin{Lemma}\label{p.faith}
The functor
  \[\Mod(\cX^\N,\Lambda_\bullet)/\cN\to \Mod(\cX^\N,\pi^{-1}\cO)/\cN'\]
induced by the inclusion functor is fully faithful.
\end{Lemma}

\begin{proof}
Let $L,M\in \Mod(\cX^\N,\Lambda_\bullet)$. By \cite[V Proposition 2.4.4
(iv)]{SGA5}, we have
  \[\Hom_{\Mod(\cX^\N,\pi^{-1}\cO)/\cN'}(L,M) \simeq \varinjlim_{r\in \N} \Hom_{\Mod(\cX^\N,\pi^{-1}\cO)}(L(r),M).\]
For a morphism $a\colon L(r)\to M$, $a(-r)\colon L(r)(-r)\to M(-r)$ is a morphism of $\Mod(\cX^\N,\Lambda_\bullet)$. This defines an inverse of the map
\[\Hom_{\Mod(\cX^\N,\Lambda_\bullet)/\cN}(L,M) \to \Hom_{\Mod(\cX^\N,\pi^{-1}\cO)/\cN'}(L,M).\]
\end{proof}

\begin{Lemma}\label{p.pread}
Let $L,M\in \Mod(\cX^\N,\Lambda_\bullet)$, $L$ preadic. Then the localization map
\[\Hom_{\Mod(\cX^\N,\Lambda_\bullet)}(L,M)\to \Hom_{\Mod(\cX^\N,\Lambda_\bullet)/\cN}(L,M)\]
is an isomorphism.
\end{Lemma}

It follows that $M\in \Mod(\cX^\N,\Lambda_\bullet)$ is AR-preadic if and
only if there exists a preadic module $L\in \Mod(\cX^\N,\Lambda_\bullet)$,
isomorphic to $M$ in $\Mod(\cX^\N,\Lambda_\bullet)/\cN$. Moreover, the
functor $\Mod_{pa}(\cX^\N,\Lambda_\bullet)\to
\Mod(\cX^\N,\Lambda_\bullet)/\cN$ is fully faithful.

\begin{proof}
By Lemma \ref{p.faith}, it suffices to show that the localization map
  \[\Hom_{\Mod(\cX^\N,\Lambda_\bullet)}(L,M)\to \Hom_{\Mod(\cX^\N,\pi^{-1}\cO)/\cN'}(L,M)\]
  is an isomorphism. By \cite[V Proposition 2.4.4 (iv)]{SGA5}, we have
  \[\Hom_{\Mod(\cX^\N,\pi^{-1}\cO)/\cN'}(L,M) \simeq \varinjlim_{r\in \N} \Hom_{\Mod(\cX^\N,\pi^{-1}\cO)}(L(r),M).\]
Since $L$ is preadic and $M$ is a $\Lambda_\bullet$-module, the map
  \[\Hom_{\Mod(\cX^\N,\pi^{-1}\cO)}(L,M)\to \Hom_{\Mod(\cX^\N,\pi^{-1}\cO)}(L(r),M)\]
  induced by the transition map $L(r)\to L$ is an isomorphism for every $r$. In fact, for $n\ge m$, \[\Hom_\cO(L_n,M_m)\simeq \Hom_\cO(\Lambda_n\otimes_{\Lambda_{n+r}}L_{n+r},M_m)\simeq \Hom_\cO(L_{n+r},M_m).\]
\end{proof}

\begin{Definition}\label{s.DN}
We say that $M\in D(\cX^\N, \Lambda_\bullet)$ is \emph{essentially zero}
(resp.\ \emph{AR-null}) if $\fH^i M$ is essentially zero (resp.\ AR-null)
for all $i$. Let $D_\cN$ be the full subcategory of $D(\cX^\N,
\Lambda_\bullet)$ consisting of AR-null complexes, which is a thick
triangulated subcategory (Definition \ref{d.thick}).
\end{Definition}

\begin{Lemma}\label{p.Lpi}\leavevmode
\begin{enumerate}
\item \label{p.Lpi1} For $M\in \Mod(\cX,\cO)$, $\pi^* M$ is preadic,
    $\fH^{-1} L\pi^* M$ is essentially zero, and $\fH^{-q}L\pi^*M=0$ for
    $q>1$.

\item \label{p.Lpi2} For $N\in D(\cX,\Lambda_n)$, the adjunction map
    $L\pi^* (\pi e_n)_* N \to e_{n*}N$ has AR-null cone. In particular,
    for $N\in \Mod(\cX,\Lambda_n)$, $\fH^{-1}L\pi^*(\pi e_n)_* N$ is
    AR-null.
\end{enumerate}
\end{Lemma}

By \ref{p.Lpi1}, for $M\in D(\cX,\cO)$, the distinguished triangle
\[L\pi^* \tau^{\le q} M \to L\pi^* M \to L\pi^* \tau^{\ge q+1} M \to\]
induces a short exact sequence
\begin{equation}\label{e.Lpi}
0\to \pi^* \fH^q M \to \fH^q L\pi^* M \to \fH^{-1}L\pi^* \fH^{q+1} M \to 0,
\end{equation}
where $\pi^* \fH^q M$ is preadic and $\fH^{-1}L\pi^* \fH^{q+1} M$ is
essentially zero.

\begin{proof}
\begin{itemize}
  \item[\ref{p.Lpi1}] The first assertion follows from the isomorphism
      $\Lambda_n\otimes_{\Lambda_{n+1}}\Lambda_{n+1}\otimes_{\cO} M \simeq
      \Lambda_n\otimes_{\cO} M$. The last assertion holds because
      $e_n^{-1} \fH^{-q} L\pi^*M \simeq\cTor_q^\cO(\Lambda_n,M)=0$ for
      $q>1$. Consider the short exact sequence of $\pi^{-1}\cO$-modules
      $0\to F\to \pi^{-1}\cO\to \Lambda_\bullet\to 0$, where $F=(\cO)_{m
      \in\N}$ has transition maps $F_{m+1}\to F_{m}$ given by
      $\cO\xrightarrow{\times \lambda} \cO$, and $F_n\to (\pi^{-1}\cO)_n$
      is given by $\cO\xrightarrow{\times \lambda^{n+1}} \cO$. This
      sequence gives a $\pi^{-1}\cO$-flat resolution of $\Lambda_\bullet$.
      So $\fH^{-1}L\pi^*M$ is a sub-$\pi^{-1}\cO$-module of
      $F\otimes_{\pi^{-1}\cO} \pi^{-1} M$, hence is essentially zero.

  \item[\ref{p.Lpi2}] For $N\in \Mod(\cX,\Lambda_n)$, $\fH^{-1}L\pi^*N$ is
      a sub-$\pi^{-1}\cO$-module of $F\otimes_{\pi^{-1}\cO} \pi^{-1} N$,
      hence is AR-null. For the first assertion, by (a) we may assume
      $N\in \Mod(\cX,\Lambda_n)$. In this case $\pi^*N \to e_{n*}N$ is an
      epimorphism with AR-null kernel.
\end{itemize}
\end{proof}

The following is a variant of \cite[Lemma 1.4]{Ekedahl}.

\begin{Lemma}\label{p.null}
Let $M\in D^-(\cX^\N,\Lambda_\bullet)$ be AR-null, and let $N\in
D(\cX^\N,\Lambda_\bullet)$. Assume either $M\in D^b$ or $N\in D^-$. Then
$M\otimes^L_{\Lambda_\bullet} N$ is AR-null.
\end{Lemma}

\begin{proof}
We may assume $M\in \Mod(\cX^\N,\Lambda_\bullet)$. It then suffices to take
a quasi-isomorphism $N'\to N$, where $N'$ belongs to the smallest
triangulated subcategory of $K(\Mod(\cX^\N,\Lambda_\bullet))$ stable under
small direct sums and containing $K^-(\cP)$, where $\cP$ is the full
subcategory of $\Mod(\cX^\N,\Lambda_\bullet)$ spanned by flat modules.
\end{proof}

\begin{Lemma}\label{p.enpf}
  Suppose that $\cX$ has enough points. Let $M\in D(\cX,\Lambda_n)$ and $N\in D(\cX,\cO)$. Then the projection formula map
  \[e_{n*}M\otimes^L_{\Lambda_\bullet} L\pi^* N \to e_{n*}(M\otimes^L_{\Lambda_n} e_n^{-1}L\pi^* N)\]
  is an isomorphism.
\end{Lemma}

\begin{proof}
We may assume that $\cX$ is the punctual topos. Since $e_{n*}\colon
\Mod(\cX,\Lambda_n)\to \Mod(\cX^\N,\Lambda_\bullet)$ commutes with small
direct limits, we may repeat the arguments in the proof of Proposition
\ref{p.pbc} \ref{p.pbc2}.
\end{proof}

Let $f\colon \cX\to \cY$ be a morphism of topoi. It induces a flat morphism
of ringed topoi
\[(f^\N_*, f^{\N*})\colon (\cX^\N,\Lambda_\bullet) \to (\cY^\N,\Lambda_\bullet).\]
The base change map $Le_{n!} f^* \Rightarrow f^{\N*}Le_{n!}$ is a natural
isomorphism of functors $D(\cY,\Lambda_n)\to D(\cX^\N,\Lambda_\bullet)$. By
adjunction, we have the following.

\begin{Lemma}\label{p.en}
  For all $M\in D(\cX^\N,\Lambda_\bullet)$, the base change map
  \begin{equation}\label{e.en}
  e_n^{-1} Rf_*^\N M \to Rf_* e_n^{-1} M
  \end{equation}
  is an isomorphism.
\end{Lemma}

The case $M\in D^+$ is a particular case of \cite[V$^\text{bis}$
1.3.12]{SGA4}. This case follows also from the fact that the full
subcategory $I_\cX$ of $\Mod(\cX^\N,\Lambda_\bullet)$ spanned by modules of
the form $\prod_n e_{n*} I_{(n)}$, where $I_{(n)}$ is an injective
$\Lambda_n$-module, is $f_*^\N$-injective. For $I\in I_\cX$, $I$ is
injective and $e_n^{-1}I$ is injective for all~$n$.

The lemma implies that $Rf^{\N}_*$ preserves essentially zero complexes in $D^+$ and AR-null complexes in $D^+$.

The base change maps induce a natural transformation of functors $D(\cX,\Lambda_n)\to D(\cY^\N,\Lambda_\bullet)$
\begin{equation}\label{e.enshr}
  Le_{n!} Rf_* \Rightarrow Rf^\N_*Le_{n!}.
\end{equation}

\begin{Lemma}\label{p.Rpi}\leavevmode
\begin{enumerate}
\item \label{p.Rpi1} If $M\in D^+(\cX^\N, \Lambda_\bullet)$ is essentially
    zero, then we have $R\pi_*M = 0$.

\item \label{p.Rpi2} For $N\in D^+(\cX,\Lambda_n)$, the adjunction map
    $(\pi e_n)_* N\to R\pi_* L\pi^* (\pi e_n)_*N$ is an isomorphism.
\end{enumerate}
\end{Lemma}

Part \ref{p.Rpi1} is a particular case of \cite[Lemma 1.1]{Ekedahl}.

\begin{proof}
\begin{itemize}
\item[\ref{p.Rpi1}] Note that $R^q \pi_* M$ is the sheaf associated to the
    presheaf $(U\mapsto H^q(U^\N,M))$, where $U$ runs over objects of
    $\cX$. Let $a\colon U\to \mathrm{pt}$ be the morphism of topoi from
    $U$ to the punctual topos. Since $Ra_*^{\N} M$ is essentially zero,
    $R\Gamma(U^\N, M)\simeq R\varprojlim Ra^\N_* M=0$.

\item[\ref{p.Rpi2}] The composite
\[R\pi_* e_{n*} N\to R\pi_* L\pi^* R\pi_* e_{n*}N\to R\pi_* e_{n*}N\]
is the identity. By (a) and Lemma \ref{p.Lpi} (b), the second map is an
isomorphism.
\end{itemize}
\end{proof}

\begin{Definition}
We define a functor
\begin{equation}\label{e.Mhat}
D(\cX^\N,\Lambda_\bullet) \to D(\cX^\N,\Lambda_\bullet), \quad M\mapsto \hat{M}=L\pi^* R\pi_* M.
\end{equation}
Following Ekedahl \cite[Definition 2.1 (iii)]{Ekedahl} we say that a complex
$M\in D(\cX^\N,\Lambda_\bullet)$ is \emph{normalized} if the adjunction map
$\hat{M} \to M$ is an isomorphism in $D(\cX^\N,\Lambda_\bullet)$. We say
that $M\in D(\cX^\N,\Lambda_\bullet)$ is \emph{weakly normalized} if for all
$n$, the natural map
\begin{equation}\label{e.Eke}
\Lambda_n\otimes^L_{\Lambda_{n+1}}M_{n+1} \to M_n
\end{equation}
is an isomorphism.
\end{Definition}

The following criterion is a variant of \cite[Proposition 2.2 (ii)]{Ekedahl}
and plays an essential role in what follows.

\begin{Lemma}\label{p.Eke}
Let $M\in D(\cX^\N,\Lambda_\bullet)$. Consider the following conditions:
\begin{itemize}
\item[(a)] $M$ is normalized;
\item[(b)] $M\simeq L\pi^* N$ for some $N\in D(\cX,\cO)$;
\item[(c)] $M$ is weakly normalized.
\end{itemize}
Then (a)$\implies$(b)$\implies$(c). Moreover, if $M\in D^+$, then they are equivalent.
\end{Lemma}

In particular, for $M\in D^+(\cX^\N,\Lambda_\bullet)$, $\hat{M}$ is normalized.

\begin{proof}
  (a)$\implies$(b) By definition, $M\simeq L\pi^* R\pi_* M$.

  (b)$\implies$(c) We have $M_n\simeq \Lambda_n\otimes^L_\cO N$ and \eqref{e.Eke} is induced by the isomorphism
  \[\Lambda_{n}\otimes^L_{\Lambda_{n+1}}\Lambda_{n+1}\otimes^L_\cO N \simto \Lambda_n\otimes^L_\cO N.\]

(c)$\implies$(a) assuming $M\in D^+$. To prove that $M$ is normalized, it
suffices to show that the map $\alpha\colon e_n^{-1} \hat{M}\to e_n^{-1}M$
is an isomorphism for all $n$. This map sits in the commutative diagram
\begin{equation}\label{e.ca}
\xymatrix{(\pi e_n)_* \Lambda_n \otimes^L_\cO R\pi_* M\ar[d]_\gamma^\simeq \ar[r]^\beta_\sim
& (\pi e_n)_* e_n^{-1}L\pi^*R\pi_* M \ar[rd]^{(\pi e_n)_*
\alpha}\\
R\pi_*(L\pi^*(\pi e_n)_* \Lambda_n \otimes^L_{\Lambda_\bullet} M)\ar[r]
& R\pi_* e_{n*} e_n^{-1}(L\pi^*(\pi e_n)_* \Lambda_n \otimes_{\Lambda_\bullet}^L M)\ar[r]
& (\pi e_n)_* e_n^{-1} M,}
\end{equation}
where $\beta$ is projection formula map for $\pi e_n$, $\gamma$ is
projection formula map for $\pi$, and the lower horizontal arrows are
adjunction maps. The map $\beta$ is trivially an isomorphism. Taking the
resolution $F\to (\pi e_n)_*\Lambda_n$, where $F$ is the complex
$\cO\xrightarrow{\times \lambda^{n+1}}\cO$ concentrated in degrees $-1$ and
$0$, we see that $\gamma$ is an isomorphism. Thus it suffices to show that
the composite of the lower row of \eqref{e.ca} is an isomorphism. This row
is obtained by applying $R\pi_*$ to the upper row of the commutative diagram
\[\xymatrix{L\pi^*(\pi e_n)_* \Lambda_n \otimes^L_{\Lambda_\bullet} M\ar[r]\ar[d]_\delta
& e_{n*} e_n^{-1}(L\pi^*(\pi e_n)_* \Lambda_n \otimes_{\Lambda_\bullet}^L M)\ar[r]\ar[d]
& e_{n*} e_n^{-1} M,\\
\pi^*(\pi e_n)_* \Lambda_n \otimes^L_{\Lambda_\bullet} M\ar[r]^-{\epsilon}
& e_{n*} e_n^{-1}(\pi^*(\pi e_n)_* \Lambda_n \otimes_{\Lambda_\bullet}^L M),\ar[ru]_\sim
}
\]
where the vertical arrows are induced by $L\pi^*(\pi e_n)_*\Lambda_n\to
\pi^*(\pi e_n)_*\Lambda_n$. By Lemma \ref{p.Lpi} \ref{p.Lpi2} and Lemma
\ref{p.null}, $\delta$ has AR-null cone. By (c),
  \[\Lambda_n\otimes^L_{\Lambda_m} M_m \to M_n\]
is an isomorphism for all $m\ge n$. It follows that $\epsilon$ has AR-null
cone. Note that $L\pi^*(\pi e_n)_*\Lambda_n\otimes^L_{\Lambda_\bullet} M$
and $e_{n*}e_n^{-1}M$ are both in $D^+$. Hence $R\pi_* (\epsilon \delta)$ is
an isomorphism by Lemma \ref{p.Rpi} (a).
\end{proof}

\begin{Lemma}\label{p.tau}
Let $M\in D(\cX^\N,\Lambda_\bullet)$ be normalized. Then the map
\[\alpha\colon L\pi^* \tau^{\ge a} R\pi_* M \to L\pi^* \tau^{\ge a} R\pi_* \tau^{\ge a}M\simeq L\pi^* R\pi_*\tau^{\ge a} M\]
is an isomorphism.
\end{Lemma}

\begin{proof}
By Lemma \ref{p.Lpi} \ref{p.Lpi1}, the map
\[L\pi^* \tau^{\ge a} R\pi_* M \to \tau^{\ge a}L\pi^* \tau^{\ge a} R\pi_* M \simeq \tau^{\ge a} L\pi^* R\pi_* M\simto \tau^{\ge a} M\]
has essentially zero cone in $D^+$. This map equals the composite map
$\beta\alpha$ in the commutative diagram
\[\xymatrix{(L\pi^* \tau^{\ge a} R\pi_* M)\sphat \ar[r]^{\hat \alpha}\ar[d]_\gamma
& (L\pi^*  R\pi_* \tau^{\ge a}M)\sphat\ar[r]^-{\hat\beta}\ar[d]^\delta & \widehat{\tau^{\ge a}M}\ar[d]^\beta\\
L\pi^* \tau^{\ge a} R\pi_* M \ar[r]^\alpha & L\pi^*  R\pi_* \tau^{\ge a}M\ar[r]^-\beta& \tau^{\ge a}M}
\]
Thus by Lemma \ref{p.Rpi} \ref{p.Rpi1}, $\hat\beta\hat\alpha$ is an
isomorphism. By Lemma \ref{p.Eke}, $\gamma$ and $\delta$ are isomorphisms.
Since $\delta$ and $\hat\beta$ have a common section given by the adjunction
$\id\to R\pi_*L\pi^*$, it follows that $\hat\beta$ is an isomorphism.
Therefore, $\hat \alpha$ and $\alpha$ are isomorphisms.
\end{proof}

Objects of $D(\cX^\N,\Lambda_\bullet)$ satisfying condition (b) (resp.\ (c))
of Lemma \ref{p.Eke} are clearly closed under derived tensor product.
Applying Lemma \ref{p.Eke}, we obtain the following proposition.

\begin{Proposition}\label{p.tensornorm}
 Let $M, N\in D^+(\cX^\N,\Lambda_\bullet)$ be normalized. Then $M\otimes^L_{\Lambda_\bullet} N$ is normalized.
\end{Proposition}

\begin{Proposition}\label{p.tensornull}
Assume that $\cO$ is a regular ring. Let $N\in D^{\ge
a}(\cX^\N,\Lambda_\bullet)$ be normalized. Then
\begin{enumerate}
  \item \label{p.tensornull1} $-\otimes^L_{\Lambda_\bullet} N$ has
      cohomological amplitude $\ge a-\dim \cO$.

  \item \label{p.tensornull2} For all $M\in D^+(\cX^\N,\Lambda_\bullet)$
      AR-null, $M\otimes^L_{\Lambda_\bullet} N$ is AR-null.
\end{enumerate}
\end{Proposition}

\begin{proof}
By assumption, $N\simeq L\pi^* N'$ for some $N'\in D^{\ge a}(\cX,\cO)$. Then
$M\otimes_{\Lambda_\bullet}^L N \simeq M\otimes^L_{\pi^{-1} \cO} \pi^{-1}
N'$. Then \ref{p.tensornull1} follows from the fact that
$-\otimes_{\pi^{-1}\cO}^L-$ has cohomological amplitude $\ge -\dim \cO$. For
\ref{p.tensornull2}, we are then reduced to the case $M\in D^b$, which is
covered by Lemma \ref{p.null}.
\end{proof}

\begin{Proposition}\label{p.RcHomfiber}
Let $M,N\in D(\cX^\N,\Lambda_\bullet)$. Suppose that either
\begin{itemize}
\item[(a)] $M$ is weakly normalized; or
\item[(b)] the cohomology sheaves of $M$ are preadic, $M\in D^-$ and $N\in D^+$.
\end{itemize}
Then the natural map
\begin{equation}\label{e.RcHomfiber}
e_n^{-1} R\cHom_{\Lambda_\bullet}(M,N) \to R\cHom_{\Lambda_n} (M_n, N_n)
\end{equation}
is an isomorphism.
\end{Proposition}

This is a variant of \cite[3.1.2]{LO2}.

\begin{proof}
Let $j_n\colon \cX^{\le n} \to \cX^\N$ be the morphism of topoi defined by
$j_n^{-1}M=(M_m)_{m\le n}$, $(j_{n*} N)_m = N_m$ for $m\le n$,
$(j_{n*}N)_m=N_n$ for $m\ge n$. Let $e'_n\colon \cX\to \cX^{\le n}$ be the
morphism of topoi defined by ${e'_n}^{-1} N = N_n$, $(e'_{n*} F)_m = \{*\}$
for $m<n$, $(e'_{n*} F)_n=F$. Then $e_n = j_n e'_n$. Let
$\pi_n\colon(\cX^{\le n},\Lambda_{\le n})\to (\cX,\Lambda_n)$ be the
morphism of ringed topoi given by $\pi_{n*}=e'_n{}^{-1}$,
$(\pi_n^*F)_m=\Lambda_m\otimes_{\Lambda_n} F$.

In case (b), we may assume $M\in \Mod(\cX^\N,\Lambda_\bullet)$ is preadic,
then $\pi_n^* M_n \simto j_n^{-1} M$. As observed in the remark following
Lemma \ref{p.en}, there is a quasi-isomorphism $N\to I$ such that $I^k$ and
$e_n^{-1}I^k$ are injective for all $k$ and $n$. Hence it suffices to show
that for any $N\in \Mod(\cX^\N,\Lambda_\bullet)$,
$e_n^{-1}\cHom_{\Lambda_\bullet}(M,N)\to \cHom_{\Lambda_n}(M_n,N_n)$ is an
isomorphism. This is clear because the map is the composition
\begin{multline*}
e_n^{-1}\cHom_{\Lambda_\bullet}(M,N) = {e'_n}^{-1} j_n^{-1} \cHom_{\Lambda_\bullet}(M,N) \simto {e'_n}^{-1}\cHom_{j_n^{-1}\Lambda_\bullet}(j_n^{-1}M, j_n^{-1}N)\\
\simto \pi_{n*}\cHom_{\Lambda_{\le n}}(\pi_n^* M_n,j_n^{-1} N) \simeq \cHom_{\Lambda_n}(M_n, \pi_{n*} j_n^{-1} N)\simeq \cHom_{\Lambda_n}(M_n,N_n).
\end{multline*}

In case (a), $L\pi_n^* M_n \simto j_n^{-1} M$. Thus \eqref{e.RcHomfiber} is
the composite isomorphism
\begin{multline*}
e_n^{-1}R\cHom_{\Lambda_\bullet}(M,N) = {e'_n}^{-1} j_n^{-1} R\cHom_{\Lambda_\bullet}(M,N) \simto {e'_n}^{-1}R\cHom_{j_n^{-1}\Lambda_\bullet}(j_n^{-1}M, j_n^{-1}N)\\
\simto \pi_{n*}R\cHom_{\Lambda_{\le n}}(L\pi_n^* M_n,j_n^{-1} N) \simeq R\cHom_{\Lambda_n}(M_n, \pi_{n*} j_n^{-1} N)\simeq R\cHom_{\Lambda_n}(M_n,N_n).
\end{multline*}

\end{proof}

\begin{Corollary}\label{p.RcHomnorm}
Let $M\in D^-(\cX^\N,\Lambda_\bullet)$, $N\in D^+(\cX^\N,\Lambda_\bullet)$,
both satisfying condition (b) of Lemma \ref{p.Eke}. Then
\[R\cHom_{\Lambda_\bullet}(M,N)\in D^+(\cX^\N,\Lambda_\bullet)\]
is normalized.
\end{Corollary}

\begin{proof}
By Lemma \ref{p.Eke} and Proposition \ref{p.RcHomfiber}, we need to show
that
  \begin{equation}\label{e.RcHomnorm0}
  \Lambda_n \otimes^L_{\Lambda_{n+1}} R\cHom_{\Lambda_{n+1}}(M_{n+1}, N_{n+1})  \to R\cHom_{\Lambda_n}(M_n, N_n)
  \end{equation}
  is an isomorphism.  By assumption,
  $M=L\pi^* M'$, $N=L\pi^*N'$, where $M'$, $N'$ are objects of $D(\cX,\cO)$. Thus \eqref{e.RcHomnorm0} is isomorphic to
  \[\Lambda_n \otimes^L_{\Lambda_{n+1}} R\cHom_{\Lambda_{n+1}}(\Lambda_{n+1}\otimes^L_{\cO} M', \Lambda_{n+1}\otimes^L_{\cO} N')
  \to R\cHom_{\Lambda_n}(\Lambda_n\otimes^L_{\cO} M', \Lambda_n\otimes^L_{\cO}N').
  \]
  Therefore, it suffices to show that
  \[\Lambda_n \otimes^L_{\cO} R\cHom_{\cO}(M', N') \to R\cHom_{\Lambda_n}(\Lambda_n\otimes^L_{\cO} M', \Lambda_n\otimes^L_{\cO}N')\]
  is an isomorphism. This map is the composition of
  \begin{equation}\label{e.RcHomnorm}
    \Lambda_n \otimes^L_{\cO} R\cHom_{\cO}(M', N')\to R\cHom_\cO(M',\Lambda_n\otimes^L_{\cO}N')
  \end{equation}
  and the adjunction isomorphism
  \[R\cHom_\cO(M',\Lambda_n\otimes^L_{\cO}N')\simto R\cHom_{\Lambda_n}(\Lambda_n\otimes^L_{\cO} M', \Lambda_n\otimes^L_{\cO}N').\]
  Taking 
  the resolution $F\to \Lambda_n$, where $F$ is the complex $\cO\xrightarrow{\times \lambda^{n+1}} \cO$ concentrated in degrees $-1$ and $0$,
  we conclude that \eqref{e.RcHomnorm} is an isomorphism.
\end{proof}

\begin{Proposition}\label{p.RcHomnull}
  Let $N\in D^+(\cX^\N,\Lambda_\bullet)$ be AR-null and let $M\in D(\cX^\N,\Lambda_\bullet)$. Assume that one of the following conditions holds:
  \begin{itemize}
  \item[(a)] $M\in D^-$ is weakly normalized;
  \item[(a$'$)] $M$ is weakly normalized and $N\in D^b$;
  \item[(b)] $M\in D^-$ is of preadic cohomology sheaves.
  \end{itemize}
  Then $R\cHom_{\Lambda_\bullet}(M,N)$ is AR-null. Moreover, if (a) or (b) holds, then \[R\Hom_{\Lambda_\bullet}(M,N)=0.\]
\end{Proposition}

\begin{proof}
We may assume $N\in \Mod(\cX^\N,\Lambda_\bullet)$. In case (b), we may
assume $M\in \Mod(\cX^\N,\Lambda_\bullet)$ is preadic. In all cases, by
Proposition \ref{p.RcHomfiber},
\[e_n^{-1} R\cHom_{\Lambda_\bullet}(M,N) \simeq R\cHom_{\Lambda_n}(M_n,N_n),\]
hence $e_n^{-1} \cExt^q_{\Lambda_\bullet}(M,N) \simeq \cExt^q_{\Lambda_n}(M_n,N_n)$. The transition maps of $\cExt^q_{\Lambda_\bullet}(M,N)$ are induced by the canonical maps
\[  R\cHom_{\Lambda_{m}}(M_{m},N_{m})\to R\cHom_{\Lambda_{m}}(M_{m},N_{n})
   \xto{\alpha} R\cHom_{\Lambda_n}(M_n, N_n)
\]
for $m\ge n$. In cases (a) and (a$'$), $\alpha$ is the composition
\[R\cHom_{\Lambda_{m}}(M_{m},N_{n}) \simto R\cHom_{\Lambda_{n}}(\Lambda_{n}\otimes^L_{\Lambda_{m}} M_{m},N_{n})\simto R\cHom_{\Lambda_n}(M_n, N_n).\]
In case (b), $\alpha$ is the composition
\[\cHom^\bullet_{\Lambda_{m}}(M_{m},N') \simto \cHom^\bullet_{\Lambda_{n}}(\Lambda_{n}\otimes_{\Lambda_{m}} M_{m},N')\simto \cHom^\bullet_{\Lambda_n}(M_n, N'),\]
where $N'$ is an injective resolution of $N_n$.
In all cases, since $N$ is AR-null, $\cExt^q_{\Lambda_\bullet}(M,N)$ is AR-null for all $q$, hence $R\cHom_{\Lambda_\bullet}(M,N)$ is AR-null. Therefore, if $M\in D^-$, then
\[R\Hom_{\Lambda_\bullet}(M,N) \simeq R\Gamma(\cX,R\pi_* R\cHom_{\Lambda_\bullet}(M,N)) =0\]
by Lemma \ref{p.Rpi} \ref{p.Rpi1}.
\end{proof}

We refer the reader to \cite[C.D., Section I.2]{SGA4d} for the construction
of quotient categories of triangulated categories by thick subcategories.

\begin{Corollary}\label{p.DN}
Let $M\in D^b(\cX^\N,\Lambda_\bullet)$, $N\in D^+(\cX^\N,\Lambda_\bullet)$,
$M$ either normalized or of preadic cohomology sheaves. Let
$D_\cN^+=D_\cN\cap D^+(\cX^\N,\Lambda_\bullet)$, where $D_\cN$ is as in
Definition \ref{s.DN}. Then the localization map
  \[\Hom_{D^+(\cX^\N,\Lambda_\bullet)}(M,N)\to \Hom_{D^+(\cX^\N,\Lambda_\bullet)/D^+_\cN}(M,N)\]
  is an isomorphism.
\end{Corollary}

\begin{proof}
  We have
  \[\Hom_{D^+(\cX,\Lambda_\bullet)/D^+_\cN}(M,N)\simeq \varinjlim_{s\colon N\to N'}\Hom_{D^+(\cX^\N,\Lambda_\bullet)}(M,N'),\]
  where $s$ runs over maps in $D^+(\cX^\N,\Lambda_\bullet)$ with cone in $D^+_\cN$. For every such $s$, it follows from Proposition \ref{p.RcHomnull} that the map
  \[\Hom_{D^+(\cX^\N,\Lambda_\bullet)}(M,N) \to \Hom_{D^+(\cX^\N,\Lambda_\bullet)}(M,N')\]
  induced by $s$ is an isomorphism.
\end{proof}

Let $f\colon \cX\to \cY$ be a morphism of topoi. Then $f^{\N*}\colon
D(\cY^\N,\Lambda_\bullet) \to D(\cX^\N, \Lambda_\bullet)$ preserves weakly
normalized complexes.

\begin{Proposition}\label{p.fnorm}
Let $f\colon \cX\to \cY$ be a morphism of topoi. The functors
\[f^{\N*}\colon D^+(\cY^\N,\Lambda_\bullet) \to D^+(\cX^\N, \Lambda_\bullet), \quad Rf_*^\N\colon D(\cX^\N,\Lambda_\bullet) \to D(\cY^\N,\Lambda_\bullet)\]
preserve normalized complexes.
\end{Proposition}

\begin{proof}
The assertion for $f^{\N*}$ follows trivially from Lemma \ref{p.Eke}. For
$Rf_{\N*}$, consider the square
  of ringed topoi
  \[\xymatrix{(\cX^\N,\Lambda_\bullet) \ar[r]^{\pi_\cX}\ar[d]_{f^{\N}} &(\cX,\cO)\ar[d]^{f}\\
  (\cY^\N,\Lambda_\bullet) \ar[r]^{\pi_\cY}& (\cY,\cO).}
  \]
It suffices to show that, for all $M\in D(\cX,\cO)$, the base change map
  \[L\pi_\cY^* Rf_* M \to Rf^\N_* L\pi_\cX^* M\]
  is an isomorphism. By Lemma \ref{p.en}, it suffices to show that the projection formula map of $D(X,\Lambda_n)$
  \begin{equation}\label{e.fO}
  \Lambda_n \otimes_\cO^L Rf_* M \to Rf_* (\Lambda_n \otimes_\cO^L M)
  \end{equation}
  is an isomorphism.
  It then suffices to take the resolution $F\to \Lambda_n$, where $F$ is the complex $\cO\xrightarrow{\times \lambda^{n+1}} \cO$ concentrated in degrees $-1$ and $0$.
\end{proof}

\begin{Construction}\label{s.jO}
Let $U$ be an object of $\cX$ and $\cU=\cX/U$. Then $\cU^\N$ can be
identified with $\cX^\N/U^\N$. Consider the morphism of topoi $j\colon
\cU\to \cX$. The functor $j_!^{\N}\colon \Mod(\cU^\N,\Lambda_\bullet) \to
\Mod(\cX^\N, \Lambda_\bullet)$ is a left adjoint of $j^{\N*}$ and is exact.
It induces a triangulated functor $j_!^{\N}\colon D(\cU^\N,\Lambda_\bullet)
\to D(\cX^\N, \Lambda_\bullet)$. The base change map
  \begin{equation}\label{e.jen}
  j_!e_n^{-1} \Rightarrow e_n^{-1}j_!^\N
  \end{equation}
is a natural isomorphism of functors $D(\cU^\N,\Lambda_\bullet)\to
D(\cX,\Lambda_n)$.  In particular, $j_!^\N$ preserves AR-null complexes.
Moreover, we have a natural isomorphism of functors $D(\cU,\Lambda_n)\to
D(\cX^\N,\Lambda_\bullet)$
  \begin{equation}\label{e.jLen}
    j_!^\N Le_{n!} \Rightarrow Le_{n!}j_!.
  \end{equation}
It follows from projection formula \eqref{e.jpf} that $j^\N_!$ preserves
weakly normalized complexes, and thus preserves normalized complexes in
$D^+$ by Lemma \ref{p.Eke}.
\end{Construction}

Recall that for $M\in \Mod(\cX^\N,\Lambda_\bullet)$ satisfying the
Mittag-Leffler-Artin-Rees condition, if we let $M'$ denote its universal
image system, the map $M'\to M$ is a monomorphism with AR-null cokernel. For
$r\ge 0$, consider the functor $T_r\colon
\Mod(\cX^\N,\Lambda_\bullet)\to\Mod(\cX^\N,\Lambda_\bullet)$  defined by
$(T_rF)_n=\Lambda_{n} \otimes_{\Lambda_{n+r}} F_{n+r}$. Note that $T_r
T_s=T_{r+s}$. The map $T_r F \to F$ has AR-null kernel and is an epimorphism
(resp.\ isomorphism) for $F$ strict (resp.\ preadic). By \cite[V Proposition
3.2.3]{SGA5}, $F$ is AR-preadic if and only if $F$ satisfies the
Mittag-Leffler-Artin-Rees condition and $T_r F'$ is preadic for some $r\ge
0$. Here $F'$ denotes the universal image system of $F$. The functor
\[\Mod_{AR\text{-}pa}(\cX^\N,\Lambda_\bullet)\to
\Mod_{pa}(\cX^\N,\Lambda_\bullet)
\]
carrying $F$ to $T_r F'$ is a right adjoint of the inclusion functor. By
\cite[V Propositions 2.1.2, 3.1.3, 3.2.4]{SGA5},
$\Mod_{pa}(\cX^\N,\Lambda_\bullet)$,
$\Mod_{AR\text{-}pa}(\cX^\N,\Lambda_\bullet)$ are stable under cokernels and
$\Mod_{pa}(\cX^\N,\Lambda_\bullet)$ is stable under extensions in
$\Mod(\cX^\N,\Lambda_\bullet)$.

\begin{Proposition}\label{p.ARpreadic}
Let $0\to L\to M\to N \to 0$ be a short exact sequence in
$\Mod(\cX^\N,\Lambda_\bullet)$ with $L$ and $N$ AR-preadic. Then $M$ is
AR-preadic.
\end{Proposition}

This is a variant of \cite[V Proposition 5.2.4]{SGA5}, where one assumes
moreover that $L_n$ is Artinian for all~$n$. The Artinian condition is
seldom satisfied in \'etale cohomology. Indeed, if there exists a nonzero
Artinian \'etale sheaf of $\Lambda_n$-modules for some $n$ on a scheme $X$
of finite type over a field, then $X$ is necessarily an Artinian scheme.

\begin{proof}
As in \cite[V Proposition 5.2.4]{SGA5}, we use the reductions in the proof
of \cite[V Proposition 3.2.4 (ii)]{SGA5} as follows. Since $L$ and $N$
satisfy the Mittag-Leffler-Artin-Rees condition \cite[V Proposition
3.2.3]{SGA5}, $M$ satisfies this condition too \cite[V Proposition 2.1.2
(ii)]{SGA5}. Let $M'$ and $N'$ be the universal image systems of $M$ and
$N$, respectively.  Let $L'=\Ker(M'\to N')$. Applying the snake lemma to the
diagram
\[\xymatrix{0\ar[r] &L'\ar[r]\ar[d] & M'\ar[r]\ar[d]& N'\ar[r]\ar[d]& 0\\
0\ar[r]& L\ar[r] & M\ar[r] & N\ar[r]& 0,}\]
we see that $L'\to L$ is a monomorphism with AR-null cokernel. Therefore, we may assume that $M$ and $N$ are strict.

Since $N$ is AR-preadic, there exists an $r\ge 0$ such that $T_r N$ is
preadic. Let $L''=\Ker(T_r M\to T_r N)$. Applying the snake lemma to the
diagram
\[\xymatrix{0\ar[r] &L''\ar[r]\ar[d] & T_r M\ar[r]\ar[d]& T_r N\ar[r]\ar[d]& 0\\
0\ar[r]& L\ar[r] & M\ar[r] & N\ar[r]& 0,}\]
we see that $L''\to L$ has AR-null kernel and cokernel. Therefore, we may assume that $M$ is strict and $N$ is preadic.

In this case, $L$ is strict \cite[V Proposition 3.1.3 (ii)]{SGA5}. Choose
$r\ge 0$ such that $T_r L$ is preadic. Then we have an exact sequence
\[0\to T \to T_r L \to M \to N \to 0,\]
where $T=\Ker(T_r L\to L)$ is AR-null. We are therefore reduced to the following lemma.
\end{proof}

\begin{Lemma}\label{p.preadic}
Let $0\to T\to L\to M\to N \to 0$ be an exact sequence in
$\Mod(\cX^\N,\Lambda_\bullet)$ with $T$ AR-null, $L$ and $N$ preadic. Then
$M$ is AR-preadic.
\end{Lemma}

\begin{proof}
  We decompose the exact sequence into two short exact sequences
  \begin{gather}
  \label{e.preadic1}
    0\to T\to L\to Q\to 0,\\
\label{e.preadic2}
    0\to Q\to M\to N\to 0.
  \end{gather}
Using \eqref{e.preadic1} and Proposition \ref{p.RcHomnull} (b), we obtain an
isomorphism $\Ext^1_{\Lambda_\bullet}(N,L) \simto
\Ext^1_{\Lambda_\bullet}(N,Q)$. The preimage of the class of
\eqref{e.preadic2} under this isomorphism provides a 9-diagram
  \[\xymatrix{&0\ar[d]&0\ar[d]\\
  &T\ar[d]\ar@{=}[r]&T\ar[d] \\
  0\ar[r]&L\ar[d]\ar[r]& M'\ar[r]\ar[d] & N\ar[r]\ar@{=}[d]&0\\
  0\ar[r]&Q\ar[r]\ar[d]& M \ar[r]\ar[d] & N\ar[r]& 0.\\
  &0&0}\]
  Since $L$ and $N$ are preadic, $M'$ is also preadic \cite[V Proposition 3.1.3]{SGA5}.
\end{proof}

\begin{Definition}
We say that $M\in \Mod(\cX^\N,\Lambda_\bullet)$ is \emph{adic} (``Noetherian
$\fm$-adic'' in the terminology of \cite[V D\'efinition 5.1.1]{SGA5}) if $M$
is preadic and, for all $n$ or, equivalently, for some $n$, $M_n$ is
Noetherian; $M\in \Mod(\cX^\N,\Lambda_\bullet)$ is \emph{AR-adic} if there
exists an adic module $N\in \Mod(\cX^\N,\Lambda_\bullet)$ and a homomorphism
$N\to M$ with AR-null kernel and cokernel.
\end{Definition}

By Lemma \ref{p.pread}, $M$ is AR-adic if and only if there exists an adic
module $N\in \Mod(\cX^\N,\Lambda_\bullet)$, isomorphic to $M$ in
$\Mod(\cX^\N,\Lambda_\bullet)/\cN$. Note that we do not assume $M_n$ to be
Noetherian, hence our notion of AR-adic modules is more general than the
notion of ``Noetherian AR-$\fm$-adic projective systems'' in \cite[V
D\'efinition 5.1.3]{SGA5}. However, if $M$ is AR-adic and strict, then $M_n$
is Noetherian. Let $\Mod_{a}(\cX^\N,\Lambda_\bullet)$ (resp.\
$\Mod_{AR\text{-}a}(\cX^\N,\Lambda_\bullet)$) be the full subcategory of
$\Mod(\cX^\N,\Lambda_\bullet)$ consisting of adic (resp.\ AR-adic) modules.

\begin{Proposition}\label{p.thick}
The category $\Mod_{AR\text{-}a}(\cX^\N,\Lambda_\bullet)$ is a thick
subcategory of $\Mod(\cX^\N,\Lambda_\bullet)$ (Definition \ref{d.thick}).
\end{Proposition}

\begin{proof}
Using the same methods as in \cite[V Proposition 5.2.1]{SGA5}, one shows
that $\Mod_{AR\text{-}a}(\cX^\N,\Lambda_\bullet)$ is stable by kernel and
cokernel. Using the same reductions as in the proof of Proposition
\ref{p.ARpreadic}, one deduces from Lemma \ref{p.preadic} that this
subcategory is also stable by extension.
\end{proof}

\begin{Construction}
We define $\Mod_c(\cX,\cO)$ to be the quotient
\[\Mod_c(\cX,\cO) = \Mod_{AR\text{-}a}(\cX^\N,\Lambda_\bullet) / \cN.\]
By Lemma \ref{p.pread}, the composition
\[\Mod_a(\cX^\N,\Lambda_\bullet)\xto{\iota} \Mod_{AR\text{-}a}(\cX^\N,\Lambda_\bullet) \xto{\psi} \Mod_c(\cX,\cO)\]
of the inclusion functor $\iota$ and the localization functor $\psi$ is an equivalence of categories. Let
\[\phi\colon \Mod_c(\cX,\cO) \to \Mod_{a}(\cX^\N,\Lambda_\bullet)\]
be a quasi-inverse of $\psi\iota$. For $M\in \Mod_{AR\text{-}a}(\cX,\cO)$,
$\phi\psi M\simeq T_r M'$ for $r$ big enough, where $M'$ is the universal
image system of $M$ and $T_r$ is the functor defined in the proof of
Proposition \ref{p.ARpreadic}.  By Lemma \ref{p.pread}, $\iota\phi$ is a
left adjoint of $\psi$. We put $\cO=\psi \Lambda_\bullet$. Note that
$\cO\xrightarrow{\times \lambda} \cO$ is a monomorphism in
$\Mod_c(\cX,\cO)$, so that $\iota\phi$ is not left exact unless $\cX$ is an
initial topos. By \cite[V Th\'eor\`eme 5.2.3]{SGA5}, $\Mod_c(\cX,\cO)$ is a
Noetherian category.

We denote the category of Noetherian $\Lambda_n$-modules on $\cX$ by
$\Mod_c(\cX,\Lambda_n)$. For $M\in \Mod_c(\cX,\cO)$, $N\in
\Mod_c(\cX,\Lambda_n)$, we have $e_{n*}N\in
\Mod_{AR\text{-}a}(\cX,\Lambda_\bullet)$, $\phi \psi e_{n*}\simeq \pi^*(\pi
e_n)_*$, and
\[\Hom_{\Mod_c(\cX,\Lambda_n)}(e_n^{-1} \phi M,N) \simeq \Hom_{\Mod_c(\cX,\cO)}(M,\psi e_{n*}N).\]
Thus $e_n^{-1}\phi\psi e_{n*}\simeq e_n^{-1}\pi^*(\pi e_n)_* \simeq \one$
and $\psi e_{n*}\colon \Mod_c(\cX,\Lambda_n)\to \Mod_c(\cX,\cO)$ is fully
faithful. Define $\Lambda_n\otimes_\cO M= e_n^{-1}\phi M$. The family of
functors
\[\left(\Lambda_n\otimes_\cO-\colon \Mod_c(\cX,\cO)\to \Mod_c(\cX,\Lambda_n)\right)_{n\in \N}\]
is conservative.
\end{Construction}

\begin{Definition}\label{d.lisse}
We say $M\in \Mod_a(\cX^\N,\Lambda_\bullet)$ is \emph{locally constant}
 if $M_n$ is locally constant \cite[IX 2.0]{SGA4} for all $n$. We say $M\in
\Mod_c(\cX,\cO)$ is \emph{locally constant} if it is in the essential image
of locally constant adic modules. For any $n$, this is equivalent to
$\Lambda_n\otimes_\cO M$ being locally constant.
\end{Definition}

\begin{Lemma}\label{l.torsion}
Let $\cA$ be an Abelian category and let $M$ be a Noetherian object of
$\cA$. Let $F\colon M\to M$ be an endomorphism. Then there exists $n\ge 0$
such that the endomorphism on $\Img(M\xrightarrow{F^n} M)$ induced by $F$ is
a monomorphism.
\end{Lemma}

We say $M\in \Mod_{pa}(\cX^\N,\Lambda_\bullet)$ is \emph{torsion-free} if
$M\xrightarrow{\times\lambda} M$ is a monomorphism in $\Mod_{pa}$. This
definition does not depend on the choice of $\lambda$. By the lemma, every
adic object $M$ of $\Mod(\cX^\N,\Lambda_\bullet)$ sits in a short exact
sequence $0\to M'\to M \to M''\to 0$ in $\Mod_a$, where $M''$ is
torsion-free, and $M'$ is torsion, namely, annihilated by $\lambda^{n}$ for
some $n\ge 0$. In $\Mod(\cX^\N,\Lambda_\bullet)$ we have an exact sequence
$0\to N\to M'\to M \to M''\to 0$, where $N$ is AR-null.

\begin{proof}
Since $M$ is Noetherian, there exists $n\ge 0$ such that
$M'=\Ker(M\xrightarrow{F^n} M)=\Ker(M\xrightarrow{F^{n+1}} M)$. Then the
morphism $M''=\Img (M\xrightarrow{F^n} M)\xrightarrow{F} M$ is a
monomorphism.
\end{proof}

\begin{Construction}\label{c.Dc}
We say that $M\in D(\cX^\N, \Lambda_\bullet)$ is \emph{AR-adic} if $\fH^i M$
is AR-adic for all $i$. Let $D_{AR\text{-}a}(\cX^\N, \Lambda_\bullet)$ be
the full subcategory of $D(\cX^\N, \Lambda_\bullet)$ consisting of  AR-adic
complexes. We define $D_c^*(\cX,\cO)$ to be the quotient
\[D_c^*(\cX,\cO) = D_{AR\text{-}a}^*(\cX^\N, \Lambda_\bullet) / D^*_\cN,\]
where $*\in\{\emptyset, +, b\}$, $D_{AR\text{-}a}^* = D_{AR\text{-}a}\cap
D^*$, $D_\cN^* = D_\cN\cap D^*$, $D_\cN$ as in Definition \ref{s.DN}. The
inclusion functors induce fully faithful functors
\[\Mod_c(\cX,\cO)\to D^b_c(\cX,\cO)
\to D^+_c(\cX,\cO)\to D_c(\cX,\cO)
\]
by \cite[Proposition 10.2.6]{KSCat}. Moreover, truncation and cohomology
functors on $D_{AR\text{-}a}(\cX^\N, \Lambda_\bullet)$ induce truncation and
cohomology functors on $D_c(\cX,\cO)$.
\end{Construction}

In the rest of this section, we assume that $\cX$ has enough points and
$\cO$ is a complete discrete valuation ring with maximal ideal $\fm$.

\begin{Lemma}\label{p.tfflat}
 Let $M\in \Mod_{pa}(\cX^\N,\Lambda_\bullet)$ be torsion-free. Then $M$ is flat.
\end{Lemma}

\begin{proof}
We may assume that $\cX$ is the punctual topos. As $M$ is strict, $R^1\pi_*
M=0$. Moreover, $\Ker(M\xrightarrow{\times \lambda^{n+1}}M)$ is AR-null, so
that the exact sequence
\[M\xrightarrow{\times \lambda^{n+1}} M \to \pi^*\Lambda_n\otimes_\cO M\to 0\]
induces a short exact sequence
\[0\to \pi_* M\xrightarrow{\times \lambda^{n+1}} \pi_* M \to M_n\to 0\]
Thus $\pi_* M$ is a torsion-free, hence flat, $\cO$-module and $M_n\simeq
\Lambda_n\otimes_\cO \pi_* M$.
\end{proof}

\begin{Proposition}\label{p.Mhat}
Let $M\in D^+(\cX^\N, \Lambda_\bullet)$ be AR-adic. Then the cone of
$\hat{M}\to M$ is AR-null and $\fH^q \hat{M}_n \in \Mod(\cX,\Lambda_n)$ is
Noetherian for all $q$ and $n$. In particular, $\hat{M}$ is AR-adic. If,
moreover, $M\in D^{[a,b]}(\cX^\N, \Lambda_\bullet)$, then $\hat{M}$ belongs
to $D^{[a-1,b]}(\cX^\N, \Lambda_\bullet)$.
\end{Proposition}

This is similar to \cite[Theorem 3.0.14]{LO2}.

\begin{proof}
We may assume $M\in \Mod(\cX^\N,\Lambda_\bullet)$ AR-adic. Note that if
$F\to M$ is a homomorphism of AR-null kernel and cokernel with $F$ adic,
then the induced map $\hat{F} \to \hat{M}$ is an isomorphism. Thus we may
assume $M$ adic. By the remark following Lemma \ref{l.torsion}, we may
further assume that $M$ is either (a) torsion-free or (b) annihilated by
some $\lambda^n$.

In case (a), $M$ is flat by Lemma \ref{p.tfflat}, hence normalized by Lemma
\ref{p.Eke}, that is $\hat{M}\simto M$. In case (b), $M\simeq\pi^*N$, where
$N=M_n\in \Mod(\cX,\Lambda_n)$ is Noetherian. By Lemma \ref{p.Lpi}
\ref{p.Lpi2}, the cone $C$ of $L\pi^*N \to \pi^* N$ is AR-null. Hence, by
Lemma \ref{p.Rpi}, the composition $N\to R\pi_* L\pi^* N \to R\pi_* \pi^*
N$, where the first map is the adjunction map, is an isomorphism. Applying
$L\pi^*$ on both sides, we obtain $L\pi^* N\simto \hat{M}$. Note that
$L\pi^* N$ belongs to $D^{[-1,0]}$ and, by the proof of Lemma \ref{p.Lpi},
$e_n ^{-1} \fH^q L\pi^* N$ is a subquotient of $N_n$, hence is Noetherian,
for all $q$ and $n$. Hence $\hat{M}$ belongs to
$D^{[-1,0]}(\cX^\N,\Lambda_\bullet)$ and $\fH^q \hat{M}_n \in
\Mod(\cX,\Lambda_n)$ is Noetherian for all $q$ and $n$. The cone of
$\hat{M}\to M$ is isomorphic to $C$, which is AR-null.
\end{proof}

\begin{Corollary}\label{p.normadic}
Let $M\in D^+(\cX^\N, \Lambda_\bullet)$ be normalized. Then, for any $n$,
$M$ is AR-adic if and only if $\fH^q M_n$ is Noetherian for all $q$.
\end{Corollary}

\begin{proof}
If $M$ is AR-adic, then $\fH^q M_n \simeq \fH^q \hat{M}_n$ is Noetherian by
Proposition \ref{p.Mhat}. Conversely, if $\fH^q M_n$ is Noetherian for all
$q$, we first show that $\widehat{\fH^q M}$ belongs to $D^{[-1,0]}$ and is
AR-adic for all $q$. The short exact sequence \eqref{e.Lpi} for $R\pi_* M$
induces a short exact sequence
  \[ 0\to M'\to \fH^q M\to M'' \to 0,\]
where $M'$ is preadic and $M''$ is essentially zero. Since $M'_n$ is
Noetherian, $M'$ is adic. By Lemma \ref{p.Rpi} \ref{p.Rpi1}, $\widehat{\fH^q
M} \simeq \widehat{M'}$. The latter belongs to $D^{[-1,0]}$ and is AR-adic
by Proposition \ref{p.Mhat}. The distinguished triangle
  \[\widehat{\tau^{\le q} M} \to M \to \widehat{\tau^{\ge q+1} M }\to\]
induces a short exact sequence
  \[0\to \fH^0 \widehat{\fH^q M}\to \fH^q M \to \fH^{-1} \widehat{\fH^{q+1}M}\to 0.\]
Thus $\fH^q M$ is AR-adic.
\end{proof}

\begin{Construction}\label{s.endthree}
Let $D_{AR\text{-}a,\norm}^+(\cX^\N,\Lambda_\bullet)$ be the full
subcategory of $D^+_{AR\text{-}a}(\cX^\N,\Lambda_\bullet)$ consisting of
normalized complexes.  By Lemma \ref{p.Rpi} \ref{p.Rpi1} and Proposition
\ref{p.Mhat}, the restriction of the functor \eqref{e.Mhat} $M\mapsto
\hat{M}$ to $D^+$ factors to give a functor
\[D^+_c(\cX,\cO)\to D^+_{AR\text{-}a,\norm}(\cX^\N,\Lambda_\bullet)\]
called the \emph{normalization functor}, that we still denote by $M\mapsto
\hat{M}$.  The normalization functor is a quasi-inverse of the composition
\[D^+_{AR\text{-}a,\norm}(\cX^\N,\Lambda_\bullet) \xto{\iota} D^+_{AR\text{-}a}(\cX^\N,\Lambda_\bullet) \xto{\psi} D^+_c(\cX,\cO)\]
of the inclusion and localization functors. Moreover, the composition of the
normalization functor and $\iota$ is a left adjoint of $\psi$.

Let $D^{+,\le a}_c(\cX,\cO)$ (resp.\ $D^{\ge a}_c(\cX,\cO)$) be the
essential image of $D^{+,\le a}_{AR\text{-}a}(\cX^\N,\Lambda_\bullet)$
(resp.\ $D^{\ge a}_{AR\text{-}a}(\cX^\N,\Lambda_\bullet)$) in
$D^+_c(\cX,\cO)$. By Proposition \ref{p.Mhat}, $M\in D^+_c(\cX,\cO)$ belongs
to $D^{+,\le a}_c(\cX,\cO)$ if and only if $\hat M \in D^{\le
a}(\cX^\N,\Lambda_\bullet)$. For $M\in D^{+,\le a}_c(\cX,\cO)$ and $N\in
D^{\ge a+1}_c(\cX,\cO)$,
\[\Hom_{D^+_c(\cX,\cO)}(M,N)\simeq \Hom_{D^+_{AR\text{-}a}(\cX,\Lambda_\bullet)}(\hat M,N')=0,\]
where $N'\in D^{\ge a+1}_{AR\text{-}a}(\cX^\N,\Lambda_\bullet)$. Thus
$(D_c^{+,\le 0}, D_c^{\ge 0})$ is a t-structure, of heart $\Mod_c(\cX,\cO)$,
compatible with the truncation and cohomology functors of Construction
\ref{c.Dc}. The cohomological amplitude of the normalization functor
$M\mapsto \hat M$ is contained in $[-1,0]$ and is not $0$ unless $\cX$ is an
initial topos.

For $M\in \Dbc(\cX,\cO)$ and $N\in D_c^+(\cX,\cO)$, define
$R\Hom_{\cO}(M,N)=R\Hom_{\Lambda_\bullet}(\hat M, N)$. By Proposition
\ref{p.RcHomnull}, this gives a functor
\[R\Hom_\cO(-,-) \colon \Dbc(\cX,\cO)^\op \times D_c^+(\cX,\cO)\to D^+(\cO),\]
where $D^+(\cO)$ is the left-bounded derived category of $\cO$-modules. By
Corollary \ref{p.DN}, for $M\in \Dbc(\cX,\cO)$ and $N\in D^+_c(\cX,\cO)$, we
have
\[H^0R\Hom_\cO(M,N)\simeq \Hom_{D^+_c(\cX,\cO)}(M,N).\]

Let $D^+_c(\cX,\Lambda_n)$ be the full subcategory of $D^+(\cX,\Lambda_n)$
consisting of complexes with Noetherian cohomology sheaves. For $M\in
\Dbc(\cX,\cO)$, $N\in D^+_c(\cX,\Lambda_n)$, we have $\widehat{e_{n*}N}
\simeq L\pi^*(\pi e_n)_* N$,
\[\Hom_{D^+_c(\cX,\Lambda_n)}(e_n^{-1}\hat M,N)\simeq \Hom_{D_c^+(\cX,\cO)}(M,e_{n*}N).\]
For $N\in D_c^+(\cX,\cO)$, we define $\Lambda_n\otimes^L_\cO N = e_n^{-1}
\hat N$. For $m,n\ge 0$, the short exact sequence
\[0\to \Lambda_m\xto{\times \lambda^{n+1}} \Lambda_{m+n+1}\to \Lambda_n\to 0\]
induces a distinguished triangle
\[\Lambda_m\otimes_\cO^L N\xto{\times \lambda^{n+1}} \Lambda_{m+n+1}\otimes_\cO^L N\to \Lambda_n\otimes_\cO^L N\to.\]
The functor
\[\Lambda_n\otimes^L_\cO - \colon D_c^+(\cX,\cO)\to D_c^+(\cX,\Lambda_n)\]
is conservative. For an interval $I$, $\Lambda_0\otimes_\cO^L N\in
D^{+,I}_c(\cX,\Lambda_0)$ implies $N\in D^{+,I}_c(\cX,\cO)$.
\end{Construction}

\section{Operations on $D^+_c(\cX,\cO)$}\label{s.5}
In this section we fix a complete discrete valuation ring $\cO$ and we let
$\fm$ denote its maximal ideal. We apply the formalism of
Section~\ref{s.topos} to \'etale topoi $\cX_\et$ of Noetherian
Deligne-Mumford stacks $\cX$ and apply the gluing result in
Section~\ref{s.1} to complete the construction of the six operations for
$D^+_c(-,\cO)$. Note that $\cX_\et$ has enough points. Recall from
Proposition \ref{p.cN} that, for any $n$, $M\in \Mod(\cX,\Lambda_n)$ is
Noetherian if and only if $M$ is constructible. Here, as in
Section~\ref{s.topos}, we let $\Lambda_n=\cO/\fm^{n+1}$. We will let
$\Mod_c(\cX,\cO)$ and $D^+_c(\cX,\cO)$ denote exclusively the categories
constructed in Section~\ref{s.topos} (rather than the more naive categories
considered at the beginning of Section~\ref{s.3}).

\begin{Proposition}\label{p.tensoradic}
Let $M, N \in D^+(\cX^\N,\Lambda_\bullet)$ be AR-adic. Suppose that
$N$ is normalized. Then $M\otimes^L_{\Lambda_\bullet} N$ is AR-adic.
\end{Proposition}

\begin{proof}
By Proposition \ref{p.Mhat} and Proposition \ref{p.tensornull}
\ref{p.tensornull2}, we may assume $M$ normalized. Moreover, by Proposition
\ref{p.tensornull} \ref{p.tensornull1}, we may assume that $M=L\pi^* M'$,
$N=L\pi^* N'$, with $M', N'\in D^b(\cX,\cO)$. In this case,
$M\otimes^L_{\Lambda_\bullet} N$ belongs to $D^b$ and is normalized. By
Proposition \ref{p.normadic}, $M_n$ and $N_n$ belong to $D_c$ for all $n$.
Hence, by Proposition \ref{p.tensor}, $e_n^{-1}(M\otimes^L_{\Lambda_\bullet}
N) \simeq M_n\otimes^L_{\Lambda_n} N_n$ belongs to $D_c$ for all $n$.
Therefore $M\otimes^L_{\Lambda_\bullet} N$ is AR-adic by Corollary
\ref{p.normadic}.
\end{proof}

\begin{Construction}\label{c.tensor}
For $M\in D^+(\cX^\N,\Lambda_\bullet)$ AR-adic and $N\in D^+_c(\cX,\cO)$, define
\[M\otimes^L_{\cO} N = M\otimes^L_{\Lambda_\bullet} \hat{N}.\]
By Propositions \ref{p.tensoradic} and \ref{p.tensornull}, this gives a
functor
\[-\otimes^L_{\cO}-\colon D^+_c(\cX,\cO)\times D^+_c(\cX,\cO) \to D^+_c(\cX,\cO).\]
\end{Construction}

Using this definition, for $N\in D^+_c(\cX,\cO)$, the projection formula of
Lemma \ref{p.enpf}
\[(e_{n*}\Lambda_n)\otimes^L_{\Lambda_\bullet}
\hat{N}\simto e_{n*} e_{n}^{-1}\hat{N}
\]
can be reformulated as
\[(e_{n*}\Lambda_n)\otimes^L_{\cO} N \simto e_{n*}(\Lambda_n\otimes^L_\cO N). \]

\begin{Proposition}
For all $M,N\in D^+_c(\cX,\cO)$, the map
\[\Lambda_n\otimes^L_\cO(M\otimes^L_\cO N)\simeq(\Lambda_n \otimes_\cO^L M)\otimes^L_{\Lambda_n}(\Lambda_n\otimes^L_\cO N)\]
is an isomorphism.
\end{Proposition}

\begin{proof}
By definition $M\otimes^L_\cO N$ is represented by $\hat
M\otimes^L_{\Lambda_\bullet} \hat N$, which is normalized by Proposition
\ref{p.tensornorm}, and the assertion is trivial.
\end{proof}

\begin{Construction}\label{c.fus}
Let $f\colon\cX\to \cY$ be a morphism of Noetherian Deligne-Mumford stacks.
It induces a flat morphism of ringed topoi $(f^\N_*, f^{\N*})\colon
(\cX_\et^\N,\Lambda_\bullet) \to (\cY_\et^\N,\Lambda_\bullet)$. The functor
$f^{\N*}\colon D(\cY^\N,\Lambda_\bullet) \to D(\cX^\N, \Lambda_\bullet)$
preserves AR-null complexes and AR-adic complexes. It induces a t-exact
functor $f^*\colon D^+_c(\cY,\cO)\to D^+_c(\cX,\cO)$. For $N\in
D^+_c(\cY,\cO)$, the map
\[f^*(\Lambda_n\otimes^L_\cO N)\to \Lambda_n\otimes^L_\cO f^* N\]
is an isomorphism. In fact, the map $f^{\N*}\hat N \to \widehat{f^* N}$ is
an isomorphism by Proposition \ref{p.fnorm}.

The functor $Rf^\N_*\colon D(\cX^\N,\Lambda_\bullet) \to
D(\cY^\N,\Lambda_\bullet)$ preserves AR-null complexes in $D^+$ by Lemma
\ref{p.en}.
\end{Construction}

In the rest of this section, let $S$ be either a regular scheme of dimension
$\le 1$ or a quasi-excellent scheme. Assume that the residue characteristic
$\ell$ of $\cO$ is invertible on $S$.

Let $\cX$ be a finite-type Deligne-Mumford $S$-stack.

\begin{Proposition}\label{p.RcHomadic}
Let $M\in D^-(\cX^\N,\Lambda_\bullet)$ be normalized and AR-adic, and let
$N\in D^+(\cX^\N,\Lambda_\bullet)$ be AR-adic. Then
$R\cHom_{\Lambda_\bullet}(M,N)\in D^+(\cX^\N,\Lambda_\bullet)$ is AR-adic.
\end{Proposition}

\begin{proof}
Up to replacing $M$ by $\widehat{\tau^{\ge a} M}\simeq L\pi^* \tau^{\ge a}
R\pi_* M$ (Lemma \ref{p.tau}), we may assume $M\in D^b$.  By Propositions
\ref{p.Mhat} and \ref{p.RcHomnull}, we may assume $N$ normalized. Then, by
Corollary \ref{p.RcHomnorm}, $R\cHom_{\Lambda_\bullet}(M,N)$ is normalized.
By Corollary \ref{p.normadic}, $M_n$ and $N_n$ belongs to $D_c$ for all $n$.
Hence, by Propositions \ref{p.RcHomfiber} and \ref{p.cX},
$e_n^{-1}R\cHom_{\Lambda_\bullet}(M,N) \simeq R\cHom_{\Lambda_n} (M_n, N_n)$
belongs to $D_c$. Therefore $R\cHom_{\Lambda_\bullet}(M,N)$ is AR-adic by
Corollary \ref{p.normadic}.
\end{proof}

\begin{Construction}\label{c.RcHom}
For $M\in D^b_c(\cX,\cO)$ and $N\in D^+(\cX^\N,\Lambda_\bullet)$ AR-adic, define
\[R\cHom_{\cX,\cO}(M, N) = R\cHom_{\Lambda_\bullet}(\hat{M}, N).\]
By Propositions \ref{p.RcHomadic} and \ref{p.RcHomnull}, this gives a
functor
\[R\cHom_{\cX,\cO}(-,-)\colon D^b_c(\cX,\cO)^\op \times D^+_c(\cX,\cO) \to D^+_c(\cX,\cO).\]
We omit $\cX$ or $\cO$ from the notation when no confusion arises.
\end{Construction}

\begin{Proposition}\label{p.RcHomO}
For all $M\in D^b_c(\cX,\cO)$, $N\in D^+_c(\cX,\cO)$, the map
  \[\Lambda_n \otimes^L_\cO R\cHom_\cO(M,N) \to R\cHom_{\Lambda_n}(\Lambda_n\otimes^L_\cO M,\Lambda_n\otimes^L_\cO N)\]
  is an isomorphism.
\end{Proposition}

It follows that if $S$ is finite-dimensional, then  $R\cHom_\cO$ carries
$\Dbc(\cX,\cO)^\op\times \Dbc(\cX,\cO)$ to $\Dbc(\cX,\cO)$ by Propositions
\ref{p.RHomft} and \ref{p.cX}.

\begin{proof}
By definition, $R\cHom_\cO(M, N)$ is represented by
$R\cHom_{\Lambda_\bullet}(\hat M,\hat N)$, which is normalized by Corollary
\ref{p.RcHomnorm}. Hence the assertion follows from Proposition
\ref{p.RcHomfiber}.
\end{proof}

\begin{Proposition}\label{p.fadic}
Let $f\colon \cX\to \cY$ be a morphism of finite-type Deligne-Mumford
$S$-stacks. The functor $Rf^\N_*$ preserves AR-adic complexes in $D^+$.
\end{Proposition}

In particular, $Rf_*^\N$ induces a functor between $D^+_c$, which we denote
by
\[Rf_*\colon D^+_c(\cX,\cO)\to D^+_c(\cY,\cO).\]

\begin{proof}
Let $M\in D^+(\cX^\N,\Lambda_\bullet)$ be AR-adic. To prove that $Rf_{\N*}
M$ is AR-adic, we may assume $M$ normalized. Then $Rf^\N_* M$ is normalized
by Proposition \ref{p.fnorm}. By Corollary \ref{p.normadic}, $M_n$ belongs
to $D_c$ for all $n$. Hence, by Lemma \ref{p.en} and Proposition \ref{p.cf}
\ref{p.cf1}, $e_n^{-1} Rf^\N_* M \simeq Rf_* M_n$ belongs to $D_c$ for all
$n$. Therefore $Rf^\N_* M$ is AR-adic by Corollary \ref{p.normadic}.
\end{proof}

\begin{Proposition}\label{p.fO}
Let $f\colon \cX\to \cY$ be a morphism of finite-type Deligne-Mumford
$S$-stacks.  For all $M\in D_c^+(\cX,\cO)$, the map
  \[
    \Lambda_n\otimes^L_\cO Rf_* M \to Rf_*(\Lambda_n\otimes^L_\cO M)
  \]
  is an isomorphism.
\end{Proposition}

It follows that for $f$ of $\ell$-prime inertia, $Rf_*$ sends
$\Dbc(\cX,\cO)$ to $\Dbc(\cX,\cO)$ by Proposition \ref{p.cf} \ref{p.cf1}.
Moreover, for $f$ proper of $\ell$-prime inertia and fibers of dimension
$\le d$, the cohomological amplitude of $Rf_*$ is contained in $[0,2d]$.

\begin{proof}
By definition $Rf_* M$ is represented by $Rf^\N_* \hat M$, which is
normalized by Proposition \ref{p.fnorm}. Hence the assertion follows from
Lemma \ref{p.en}.
\end{proof}

\begin{Proposition}\label{p.pf}
Let $f\colon \cX\to \cY$ be a proper morphism
    between finite-type Deligne-Mumford $S$-stacks, and let $M\in
    D^+(\cX,\Lambda_\bullet)$, $N\in
    D^+_{AR\text{-}a,\norm}(\cY,\Lambda_\bullet)$. Then the projection
    formula map
\[N\otimes_{\Lambda_\bullet}^L Rf^{\N}_* M \to Rf^\N_* (f^{\N*} N\otimes^L_{\Lambda_\bullet} M)\]
is an isomorphism.
\end{Proposition}

\begin{proof}
By Lemma \ref{p.en} and proper base change, we may assume that $\cY$ is the
spectrum of a separably closed field. As in the proof of Proposition
\ref{p.Mhat}, we may assume $N\simeq L\pi^* K$, where $K$ is a finite
$\cO$-module. Replacing $K$ by a finite free resolution, the assertion
becomes clear.
\end{proof}

\begin{Proposition}\label{p.sbcO}\leavevmode
\begin{enumerate}
\item \label{p.sbcO1} (Base change) Let
\[\xymatrix{\cX'\ar[r]^{h}\ar[d]_{f'} & \cX\ar[d]^{f}\\
\cY'\ar[r]^{g} & \cY}\]
be a 2-Cartesian square of Deligne-Mumford stacks with $\cX$ and $\cY$ of finite type over $S$. Assume that one of the following conditions holds:
\begin{itemize}
  \item[(a)] $f$ is proper and $\cY'$ is of finite type over a scheme
      $S'$, either regular of dimension $\le 1$ or quasi-excellent.
  \item[(b)] $g$ is smooth and of finite type.
\end{itemize}
Then the base change map
\[g^* Rf_* M \to Rf'_* h^* M\]
is an isomorphism for all $M\in D_c^+(\cX,\cO)$.

\item \label{p.sbcO2} (Projection formula) Let $f\colon \cX\to \cY$ be a
    proper morphism between finite-type Deligne-Mumford $S$-stacks, and
    let $M\in D^+_c(\cX,\cO)$, $N\in D^+_c(\cY,\cO)$. Then the map
\[N\otimes_\cO^L Rf_* M \to Rf_* (f^* N\otimes^L_\cO M)\]
is an isomorphism.
\end{enumerate}
\end{Proposition}

\begin{proof}
\begin{itemize}
\item[\ref{p.sbcO1}] By Lemma \ref{p.en}, the proper or smooth
    (Proposition \ref{p.sbc}) base change map $g^{\N*} Rf^\N_* \hat M \to
    Rf'^{\N}_* h^{\N*} \hat M$ is an isomorphism.

\item[\ref{p.sbcO2}] This follows from Proposition \ref{p.pf}.
\end{itemize}
\end{proof}

\begin{Proposition}\label{p.fuO}
Let $f\colon \cX\to \cY$ be a morphism of Deligne-Mumford stacks, with $\cY$
of finite type over $S$. Let $M\in \Dbc(\cY,\cO)$, $L\in D^+_c(\cY,\cO)$.
Assume that one of the following conditions holds:
  \begin{itemize}
   \item[(a)] The cohomology sheaves of $M$ are locally constant
       (Definition \ref{d.lisse}) and $\cX$ is of finite type over some
       scheme $S'$ either regular of dimension $\le 1$ or quasi-excellent.
    \item[(b)] $f$ is smooth and of finite type.
  \end{itemize}
  Then the map
  \[f^* R\cHom_\cY(M,L)\to R\cHom_\cX(f^*M,f^*L)\]
  is an isomorphism.
\end{Proposition}

\begin{proof}
\begin{itemize}
\item[(a)] The assertion follows from Proposition \ref{p.RcHomO} and the
    fact that $\Lambda_n\otimes_\cO M\in \Dcft(\cY,\Lambda_n)$ has locally
    constant cohomology sheaves.

\item[(b)] The assertion follows from Propositions \ref{p.RcHomO} and
    \ref{p.Homsm}.
\end{itemize}
\end{proof}

Let $p\colon \cX\to \cY$ be a proper morphism of finite-type Deligne-Mumford
$S$-stacks. If $p$ has $\ell$-prime inertia and the fibers of $p$ have
dimension $\le d$, then $p_*^\N$ has cohomological dimension $\le 2d$ by
Lemma \ref{p.cd}.

\begin{Lemma}\label{p.enshr}
  Let $p\colon \cX\to \cY$ be a proper morphism of $\ell$-prime inertia between finite-type Deligne-Mumford $S$-stacks. For all $M\in D(\cX,\Lambda_n)$, the map \eqref{e.enshr}
  \[Le_{n!}Rp_*M \to Rp_*^\N Le_{n!}M\]
  is an isomorphism.
\end{Lemma}

\begin{proof}
By Lemma \ref{p.en}, it suffices to show that the map
\[\Lambda_m\otimes_{\Lambda_n}^L Rp_* M \to Rp_*(\Lambda_m\otimes^L_{\Lambda_n} M)\]
is an isomorphism for $m\le n$, which is projection formula (Proposition
\ref{p.pbc} \ref{p.pbc2}).
\end{proof}

\begin{Construction}\label{d.jO}
Let $j\colon \cU\to \cX$ be an \'etale representable morphism between
finite-type Deligne-Mumford $S$-stacks. Then Construction \ref{s.jO} applies
and we have $j_!^\N\colon D(\cU^\N,\Lambda_\bullet)\to
D(\cX^\N,\Lambda_\bullet)$.
\end{Construction}

\begin{Construction}\label{s.lsO}
We construct $Rf_!^\N$ by gluing as in Construction \ref{s.fshr}. Let $\cC$
be the 2-category whose objects are Deligne-Mumford $S$-stacks of finite
type and finite inertia and whose morphisms are the separated morphisms of
$\ell$-prime inertia, let $\cA$ be the arrowy 2-subcategory whose morphisms
are the open immersions, let $\cB$ be the arrowy 2-subcategory whose
morphisms are the proper morphisms, and let $\cD$ be the 2-category of
triangulated categories. Let $F_\cA\colon \cA\to \cD$ be the pseudofunctor
  \[\cX\mapsto D(\cX^\N,\Lambda_\bullet), \quad j\mapsto j_!^\N, \quad \alpha\mapsto \alpha_!^\N,\]
and let $F_\cB\colon \cB\to \cD$ be the pseudofunctor
  \[\cX\mapsto D(\cX^\N,\Lambda_\bullet), \quad p\mapsto Rp_*^\N, \quad \alpha\mapsto R\alpha_*^\N.\]
For any proper open immersion $f$, let $\rho(f)$ be the inverse of the
2-cell \eqref{e.rho} $f_!^\N\Rightarrow Rf_*^\N$. For any 2-Cartesian square
$D$ of the form \eqref{e.2}, let $G_D\colon i_!^\N Rq_*^\N\Rightarrow
Rp_*^\N j_!^\N$ be the 2-cell as in \eqref{e.shrstar}. By \cite[Proposition
8.8]{glue}, $(F_\cA,F_\cB,G,\rho)$ is an object of
$\GD^\Cart_{\cA,\cB}(\cC,\cD)$. By Proposition \ref{p.gluestack}, this
object defines a pseudofunctor $F\colon \cC\to \cD$. For any morphism $f$ of
$\cC$, we define
  \[Rf_!^\N\colon D(\cX^\N,\Lambda_\bullet)\to D(\cY^\N,\Lambda_\bullet)\]
to be $F(f)$. If the fibers of $f$ have dimension $\le d$, then $Rf_!^\N$
has cohomological amplitude contained in $[0,2d]$. If $f$ is representable
and \'etale, then $Rf_!^\N$ coincides with $f_!^\N$ (Construction
\ref{d.jO}) as in Remark \ref{s.Rshr=shr}.

We construct the support-forgetting map $Rf_!^\N\Rightarrow Rf_*^\N$, the
base change isomorphism and the projection formula isomorphism by gluing as
before. We construct two more isomorphisms by gluing. We define a
pseudonatural transformation $\epsilon\colon F\to F_n$ with
$\epsilon(\cX)=e_n^{-1}$ by gluing the inverse of \eqref{e.jen} for $\cA$
and \eqref{e.en} for $\cB$, which is possible by \cite[Propositions 8.9,
8.10]{glue}. We define a pseudonatural transformation $\eta\colon F_n\to F$
with $\eta(\cX)=Le_{n!}$ by gluing the inverse of \eqref{e.jLen} for $\cA$
and Lemma \ref{p.enshr} for $\cB$, which is possible by \cite[Proposition
8.8]{glue}. We have obtained the following.
\end{Construction}

\begin{Lemma}\label{p.shren}
Let $f\colon \cX\to \cY$ be a separated morphism of $\ell$-prime inertia
between finite-type and finite-inertia Deligne-Mumford stacks. Let $M\in
D(\cX^\N,\Lambda_\bullet)$, $N\in D(\cY,\Lambda_n)$. Then the maps
  \begin{gather}
    \label{e.shren} e_n^{-1}Rf_!^\N M\to Rf_!^\N e_{n}^{-1} M,\\
    \label{e.shrLen} Le_{n!}Rf_! N\to Rf_!^\N Le_{n!} N
  \end{gather}
  are isomorphisms.
\end{Lemma}

It follows from \eqref{e.shren}  that $Rf_!^\N$ preserves AR-null complexes.
It follows from \eqref{e.shren} and projection formula (Theorem \ref{p.bcc}
\ref{p.bcc2}) that $Rf_!^\N$ preserves weakly normalized complexes, thus
preserves normalized complexes in $D^+$. It also follows from
\eqref{e.shren} that $Rf_!^\N$ commutes with small direct sums, hence, by
Brown Representability Theorem \cite[Theorem 14.3.1 (ix)]{KSCat}, admits a
right adjoint
\[Rf^{\N!}\colon D(\cY^\N,\Lambda_\bullet)\to D(\cX^\N,\Lambda_\bullet).\]

\begin{Construction}\label{p.shrO}
Let $f\colon \cX\to \cY$ be a morphism of $\ell$-prime inertia between
finite-type Deligne-Mumford $S$-stacks. Assume either (a) $f$ is
representable and \'etale; or (b) $f$ is separated and $\cX$, $\cY$ are of
finite inertia. Then $Rf_!^\N$ is defined in Constructions \ref{d.jO} and
\ref{s.lsO}. We prove as in Proposition \ref{p.fadic} that it preserves
AR-adic complexes in $D^+$. Thus $Rf_!^\N$ induces
  \[Rf_!\colon D^+_c(\cX,\cO)\to D^+_c(\cY,\cO),\]
  which sends $\Dbc$ to $\Dbc$.
\end{Construction}

\begin{Construction}\label{s.dagO}
Let $f\colon \cX\to \cY$ be a separated morphism between finite-type and
finite-inertia Deligne-Mumford $S$-stacks. We define
  \[Rf_\dag^\N\colon D^+(\cX^\N,\Lambda_\bullet)\to D^+(\cY^\N,\Lambda_\bullet)\]
by gluing $j_!^\N$ for $j$ an open immersion and $Rp_*^\N$ for $p$ proper as
in Construction \ref{s.dag}. If $f$ is of $\ell$-prime inertia, then
$Rf_\dag^\N$ is isomorphic to the restriction of $Rf_!^\N$ to $D^+$. We
construct the support-forgetting map $Rf_\dag^\N \Rightarrow Rf_*^\N$ and
the base change isomorphism by gluing as before. For $N\in
D^+_{AR\text{-}a,\norm}$, we glue projection formula for $Rp^\N_*$
(Proposition \ref{p.pf}) and projection formula for $j_!^\N$ (which follows
from \eqref{e.jpf} and \eqref{e.jen}) to get projection formula isomorphism
for $Rf_\dag^\N$. Moreover, we glue \eqref{e.en} and the inverse of
\eqref{e.jen} and obtain the following.
\end{Construction}

\begin{Lemma}
Let $f\colon \cX\to \cY$ be a separated morphism between finite-type and
finite-inertia Deligne-Mumford $S$-stacks, and let $M\in
D^+(\cX^\N,\Lambda_\bullet)$. Then the map
  \[e_n^{-1}Rf_\dag^\N M \to Rf_\dag^\N e_n^{-1} M\]
  is an isomorphism.
\end{Lemma}

\begin{Construction}\label{c.dagO}
In particular, $Rf_\dag^\N$ preserves AR-null complexes. It follows from
Proposition \ref{p.fnorm} and projection formula \eqref{e.jpf} that
$Rf_\dag^\N$ preserves normalized complexes. One checks as in Proposition
\ref{p.fadic} that it preserves AR-adic complexes. Thus it induces
\[Rf_\dag \colon D^+_c(\cX,\cO)\to D^+_c(\cY,\cO),\]
endowed with a support-forgetting map $Rf_\dag \Rightarrow Rf_*$ which is a
natural isomorphism if $f$ is proper. If $f$ is of $\ell$-prime inertia,
then $Rf_\dag$ can be identified with $Rf_!$.
\end{Construction}

\begin{Construction}
Let $f\colon \cX \to \cY$ be a finite morphism between finite-type
Deligne-Mumford $S$-stacks. The functor $f_*^\N$ admits a right adjoint
$f^{\N!}\colon \Mod(\cY^\N,\Lambda_\bullet)\to
\Mod(\cX^\N,\Lambda_\bullet)$. Let $Rf^{\N!}\colon D(\cY^\N,\Lambda_\bullet)
\to D(\cX^\N,\Lambda_\bullet)$ be the right derived functor. This is
compatible with the definition following Lemma \ref{p.shren}.
\end{Construction}

Combining the isomorphisms induced from Lemma \ref{p.enshr} and
\eqref{e.shrLen} by adjunction, we have the following.

\begin{Lemma}\label{p.enus}
Let $f\colon \cX\to \cY$ be a separated morphism of $\ell$-prime inertia
between finite-type Deligne-Mumford $S$-stacks. Assume that either $f$ is a
closed immersion, or $\cX$ and $\cY$ are of finite inertia. Then, for all
$M\in D(\cY^\N,\Lambda_\bullet)$, the map
  \[e_n^{-1}Rf^{\N!} M\to Rf^! e_n^{-1}M\]
  is an isomorphism.
\end{Lemma}

It follows that $Rf^{\N!}$ preserves AR-null complexes in $D^+$. Moreover,
if $S$ is finite-dimensional, then $Rf^{\N!}$ preserves AR-null complexes
and weakly normalized complexes by Proposition \ref{p.cdim} \ref{p.cdim2}.

\begin{Construction}
Let $i\colon \cY \to \cX$ be a closed immersion between finite-type
Deligne-Mumford $S$-stacks.  Let $j\colon \cU\to \cX$ be the complementary
open immersion. For any complex $M$ of $\Lambda_\bullet$-modules on
$\cX^\N$, we have a natural short exact sequence
  \[0 \to j^\N_!j^{\N*} M \to M \to i^\N_* i^{\N*} M \to 0,\]
  hence a distinguished triangle in $D(\cX^\N,\Lambda_\bullet)$
  \begin{equation}\label{e.distO1}
  j_!^\N j^{\N*} M \to M \to i^\N_* i^{\N*} M \to.
  \end{equation}

  For any complex $N$ of injective $\Lambda_\bullet$-modules on $\cX^\N$, we have a natural short exact sequence
  \[0\to i_*^\N i^{!\N} N \to N \to j_*^\N j^{\N*} N \to 0,\]
  hence, for any $N\in D(\cX^\N,\Lambda_\bullet)$, a distinguished triangle
  \begin{equation}\label{e.distO2}
  i^\N_*Ri^{\N!}N \to N \to Rj^\N_* j^{\N*} N \to.
  \end{equation}
\end{Construction}

\begin{Proposition}
  In the situation of Proposition \ref{p.enus}, $Rf^{\N!}$ preserves normalized complexes in $D^+$.
\end{Proposition}

\begin{proof}
We easily reduce to two cases: (a) $f$ is a closed immersion; (b) $f$ is a
smooth morphism of schemes. Case (a) follows from the distinguished triangle
\eqref{e.distO2} and Proposition \ref{p.fnorm}. Case (b) follows from the
fact that $Rf^{\N!}\simeq f^{\N*}(d)[2d]$, where $d$ is the relative
dimension of $f$.
\end{proof}

\begin{Construction}\label{c.fushrO}
In the situation of Proposition \ref{p.enus}, one checks as in Proposition
\ref{p.fadic} that $Rf^{\N!}$ preserves AR-adic complexes in $D^+$. It
induces
\[Rf^!\colon D^+_c(\cY,\cO)\to D^+_c(\cX,\cO),\]
which sends $\Dbc$ to $\Dbc$ by Proposition \ref{p.cf} \ref{p.cf3} and
Proposition \ref{p.Lambdan} \ref{p.Lambdan3} below.
\end{Construction}

As in Proposition \ref{p.fO}, we have the following.

\begin{Proposition}\label{p.Lambdan}
  Let $f$ be a morphism between finite-type Deligne-Mumford $S$-stacks, and let $M\in D^+_c(\cX,\cO)$, $N\in D^+_c(\cY,\cO)$.
\begin{enumerate}
  \item If $f$ is representable and \'etale, then the map
  \[\Lambda_n\otimes^L_\cO Rf_! M \to Rf_!(\Lambda_n\otimes^L_\cO M)\]
  is an isomorphism.

  \item If $f$ is separated, $\cX$ and $\cY$ are of finite inertia, then the map
  \[\Lambda_n\otimes^L_\cO Rf_\dag M \to Rf_\dag(\Lambda_n\otimes^L_\cO M)\]
  is an isomorphism.

  \item \label{p.Lambdan3} If $f$ is as in Proposition \ref{p.enus}, then
      the map
  \[\Lambda_n\otimes^L_\cO Rf^! N \to Rf^! (\Lambda_n \otimes^L_\cO N)\]
  is an isomorphism.
\end{enumerate}
\end{Proposition}

\begin{Construction}\label{d.distO}
Let $i\colon\cY \to \cX$ be a closed immersion between finite-type
Deligne-Mumford $S$-stacks, and let $j\colon \cU \to \cX$ be the
complementary open immersion. For $M\in D^+_c(\cX,\cO)$, \eqref{e.distO1}
and \eqref{e.distO2} induce distinguished triangles
\begin{gather}
  j_!j^* M \to M \to i_* i^* M \to,\label{e.dist1O}\\
  i_*Ri^!M \to M \to Rj_* j^* M \to.\label{e.dist2O}
\end{gather}
\end{Construction}

Construction \ref{s.dagO} yields the following.

\begin{Theorem}\label{p.dagO}\leavevmode
\begin{enumerate}
\item \label{p.dagO1} (Base change) Let
  \[\xymatrix{\cX'\ar[r]^h \ar[d]_{f'} & \cX\ar[d]^f\\
  \cY'\ar[r]^g & \cY}\] be a 2-Cartesian square of Deligne-Mumford stacks,
  with $f$ separated, $\cX$ and $\cY$ of finite type and finite inertia
  over $S$, $\cY'$ of finite type and finite inertia over a scheme $S'$,
  either regular of dimension $\le 1$ or quasi-excellent. Then, for all
  $M\in D^+_c(\cX,\cO)$, the map
  \[g^* Rf_\dag M \to Rf'_\dag h^* M\]
  is an isomorphism.

  \item \label{p.dagO2} (Projection formula) Let $f\colon \cX\to \cY$ be a
      separated morphism between finite-type and finite-inertia
      Deligne-Mumford $S$-stacks, and let $M\in D^+_c(\cX,\cO)$, $N\in
      D^+_c(\cY,\cO)$. Then the map
  \[N\otimes^L_\cO Rf_\dag M \to Rf_\dag(f^*N\otimes^L_\cO M)\]
  is an isomorphism.
\end{enumerate}
\end{Theorem}

\begin{Corollary}[K\"unneth formula]
For every 2-commutative diagram
\[\xymatrix{\cX_1\ar[rd]\ar[d]_{f_1} && \cX_2\ar[d]^{f_2}\ar[ld]\\
\cY_1\ar[r] &\cS& \cY_2\ar[l]}
\]
of finite-type and finite-inertia Deligne-Mumford $S$-stacks such that $f_1$
and $f_2$ are separated, and for $M_i\in D^+_c(\cX_i,\cO)$, $i=1,2$, we have
a canonical isomorphism
\begin{equation}\label{e.KFdagger}
f_{1\dag}M_1\boxtimes^L_{\cS} f_{2\dag}M_2 \simto f_\dag (M_1\boxtimes^L_{\cS} M_2),
\end{equation}
where $f\colon \cX_1\times_{\cS}\cX_2 \to \cY_1\times_{\cS}\cY_2$.
\end{Corollary}

The base change isomorphism in Construction \ref{s.lsO} induces the
following by adjunction.

\begin{Proposition}\label{p.usbcO}
  Let
    \[\xymatrix{\cX'\ar[r]^h \ar[d]_{f'} & \cX\ar[d]^f\\
  \cY'\ar[r]^g & \cY}\]
  be a 2-Cartesian square of Deligne-Mumford $S$-stacks of finite type and finite inertia, $g$ separated of $\ell$-prime inertia. Let $M\in D^+_c(\cX,\cO)$. Then the map
  \[Rf'_*Rh^! M \to Rg^! Rf_* M\]
  is an isomorphism.
\end{Proposition}

The isomorphisms \eqref{e.RcHomadj}, \eqref{e.jshradju}, \eqref{e.shradj1},
\eqref{e.dist4} and \eqref{e.shradj2} induce the following.

\begin{Proposition}\label{p.shradjO}
Let $f$ be a morphism between finite-type Deligne-Mumford $S$-stacks, and
let $K\in \Dbc(\cX,\cO)$, $L\in D^+_c(\cY,\cO)$, $M\in \Dbc(\cY,\cO)$, $N\in
D^+_c(\cX,\cO)$.
\begin{enumerate}
\item We have a canonical isomorphism
\[R\cHom_\cY(M,Rf_*N)\simto Rf_*R\cHom_\cX(f^*M,N).
\]

  \item \label{p.shradjO1}  If $f$ is representable and \'etale, then we
      have a canonical isomorphism
  \begin{equation}\label{e.shradjO1}
  R\cHom_\cY(Rf_!K,L)\simto Rf_*R\cHom_\cX(K,f^* L).
  \end{equation}

  \item \label{p.shradjO2} If $f$ is separated of $\ell$-prime inertia,
      and $\cX$ and $\cY$ are of finite inertia, then we have a canonical
      isomorphism
      \begin{equation}\label{e.shradjO2}
  R\cHom_\cY(Rf_!K,L)\simto Rf_*R\cHom_\cX(K,Rf^! L).
  \end{equation}

  \item \label{p.shradjO3} If $f$ is as in Proposition \ref{p.enus}, then
      we have a canonical isomorphism
      \begin{equation}\label{e.shradjO3}
  R\cHom_\cX(f^*M,Rf^!L)\simto Rf^!R\cHom_\cY(M,L).
  \end{equation}
\end{enumerate}
\end{Proposition}

\begin{Remark}\label{s.Eke}
Let $X$ be a scheme separated of finite type over $S$. In
\cite[Section~6]{Ekedahl}, Ekedahl constructs a category
$\Dbc(X,\cO)_{\mathrm{Ek}}$ and the six operations on it. By Construction
\ref{s.endthree} and \cite[Proposition 2.7 ii)]{Ekedahl}, the normalization
functor induces an equivalence of categories $\Dbc(X,\cO)\to
\Dbc(X,\cO)_{\mathrm{Ek}}$. The six operations constructed earlier in this
section, when restricted to schemes, are compatible with the six operations
of Ekedahl via this equivalence.
\end{Remark}

In the rest of this section let $S$ be either a regular scheme of dimension
$\le 1$ or an excellent finite-dimensional scheme, endowed with a
nonnegative dimension function $\delta_S$. Recall that $\ell$ is assumed to
be invertible on $S$.

\begin{Construction}\label{s.DO}
Let $\cX$ be a finite-type Deligne-Mumford $S$-stack.  Let
$\Omega_{\cX,n}\in D^{[-2d_\cX,0]}(\cX,\Lambda_n)$ be the dualizing complex
in Construction \ref{c.D}. By construction, we have isomorphisms
$\Lambda_m\otimes^L_{\Lambda_n}\Omega_{\cX,n}\to \Omega_{\cX,m}$ for $m\le
n$. We use the notation introduced in the proof of Proposition
\ref{p.RcHomfiber}. The morphisms of topoi $(j_n\colon \cX^{\le n}\to
\cX^\N)$ give an open covering of $\cX^\N$. Let $\Omega_{\cX,\le
n}=L\pi_n^*\Omega_{\cX,n}\in D(\cX^{\le n},\Lambda_{\le n})$. Then we have
isomorphisms ${e'_n}^{-1}\Omega_{\cX,\le n}=\Omega_{\cX,n}$ and
$\Omega_{\cX,\le n}\res \cX^{\le m}\simeq \Omega_{\cX,\le m}$ for $m\le n$.
Moreover, $R\cHom_{\Lambda_{\le n}}(\Omega_{\cX,\le n},\Omega_{\cX,\le
n})\simeq \Lambda_{\le n}$. In fact,
\begin{multline*}
{e_n'}^{-1}R\cHom_{\Lambda_{\le
n}}(\Omega_{\cX,\le n},\Omega_{\cX,\le n}) \simeq \pi_{n*}R\cHom_{\Lambda_{\le
n}}(L\pi^*_{n}\Omega_{\cX,n},\Omega_{\cX,\le n})\\
\simeq R\cHom_{\Lambda_n}(\Omega_{\cX,n},\pi_{n*}\Omega_{\cX,\le n})
\simeq R\cHom_{\Lambda_n}(\Omega_{\cX,n},\Omega_{\cX,n})\simeq \Lambda_n.
\end{multline*}
By the Be\u\i linson-Bernstein-Deligne gluing theorem \cite[Th\'eor\`eme
3.2.4]{BBD}, there exists a unique $\Omega_\cX\in
D^b(\cX^\N,\Lambda_\bullet)$ such that for all $n$, $j_n^{-1}\Omega_\cX
\simeq \Omega_{\cX,\le n}$. By construction, $\Omega_\cX$ is AR-adic and
normalized. Thus
\[\Lambda_n\otimes^L_\cO \Omega_\cX \simeq e_n^{-1} \Omega_\cX \simeq \Omega_{\cX,n}.\]
Define a triangulated functor $D_\cX\colon \Dbc(\cX,\cO)^{\op} \to
D^+_c(\cX,\cO)$ by $D_\cX M = R\cHom_{\cO}(M,\Omega_\cX)$.
\end{Construction}

\begin{Proposition}\label{p.DO}
  Let $\cX$ be a finite-type Deligne-Mumford $S$-stack.
  The functor $D_\cX$ induces a functor $\Dbc(\cX,\cO)^\op \to \Dbc(\cX,\cO)$ of cohomological amplitude contained in $[-2d_\cX,1]$.
  Moreover, for $M\in \Dbc(\cX,\cO)$, the morphisms
  \[\Lambda_n\otimes^L_\cO D_\cX M \to D_{\cX, n} (\Lambda_n\otimes^L_{\cO} M), \qquad M\to D_{\cX}D_{\cX} M\]
  are isomorphisms.
\end{Proposition}

\begin{proof}
The first morphism above is an isomorphism by Proposition \ref{p.RcHomO}.
For $M\in D^{[a,b]}_c(\cX,\cO)$, $\Lambda_n\otimes^L_\cO M\in
D^{[a-1,b]}_c(\cX,\Lambda_n)$. By Proposition \ref{p.GD} \ref{p.GD1},
  $\Lambda_n\otimes^L_\cO D_\cX M$ belongs to $D^{[-b-2d_\cX,-a+1]}$ for all $n$. Hence $D_\cX M$ belongs to $D^{[-b-2d_\cX,-a+1]}_c(\cX,\cO)$. Moreover, for all $n$, the following diagram commutes
  \[\xymatrix{\Lambda_n\otimes^L_\cO M\ar[r]\ar[d] & D_{\cX,n}D_{\cX,n} (\Lambda_n\otimes^L_\cO M)\ar[d]^\simeq\\
  \Lambda_n\otimes^L_\cO D_\cX D_\cX M \ar[r]^\sim & D_{\cX,n} (\Lambda_n\otimes^L_\cO D_\cX M).} \]
By Proposition \ref{p.GD} \ref{p.GD2}, the top arrow is an isomorphism.
Therefore, the map $M \to D_{\cX}D_{\cX} M$ is an isomorphism.
\end{proof}

\begin{Remark}\label{s.fushrO}
Let $f\colon \cX\to \cY$ be a separated morphism of $\ell$-prime inertia
between finite-type Deligne-Mumford $S$-stacks. Assume either $f$ is a
closed immersion, or $\cX$ and $\cY$ are of finite inertia. The isomorphism
\eqref{e.shradjO3} gives an isomorphism $D_\cX f^* \simto R f^! D_\cY $ of
functors $D^b_c(\cY,\cO)\to \Dbc(\cX,\cO)$. Using biduality (Proposition
\ref{p.DO}), we obtain an isomorphism of functors $\Dbc(\cY,\cO)\to
\Dbc(\cX,\cO)$:
\[R f^! \simeq R f^! D_\cY D_\cY \simeq D_\cX f^* D_\cY.\]
\end{Remark}

\begin{Construction}\label{d.fushriekO}
Let $f\colon \cX\to\cY$ be a morphism between finite-type Deligne-Mumford
$S$-stacks. Thanks to Proposition \ref{p.DO}, we can define a triangulated
functor
\[Rf^!\colon \Dbc(\cY,\cO)\to \Dbc(\cX,\cO), \quad N \mapsto D_\cX f^* D_\cY N.\]
By Remark \ref{s.fushrO}, this definition is compatible with previous
definitions. For any $N\in \Dbc(\cY,\cO)$, Proposition \ref{p.DO} induces an
isomorphism
  \[\Lambda_n\otimes^L_\cO Rf^! N \to Rf^!(\Lambda_n\otimes^L_\cO N).\]
If $f\colon \cX\to \cY$ and $g\colon \cY\to \cZ$ are two such morphisms, then we have an isomorphism of functors $\Dbc(\cZ,\cO)\to \Dbc(\cX,\cO)$:
\[
R(gf)^! = D_\cX (gf)^* D_\cZ  \simeq D_\cX f^* g^* D_\cZ
 \simeq D_\cX f^* D_\cY D_\cY g^* D_\cZ  = Rf^! Rg^!.
\]

For $L,M\in \Dbc(\cY,\cO)$, we have
\begin{multline}\label{e.shradjO3p}
R\cHom_\cX(f^*M,Rf^!L)\simeq D_\cX(f^*M\otimes^L_\cO D_\cX Rf^! L)\simeq D_\cX(f^*M \otimes^L_\cO f^*D_\cY L)
\\
\simeq D_\cX f^*(M\otimes^L_\cO D_\cY L)\simeq Rf^!D_\cY(M\otimes^L_\cO D_\cY L) \simeq Rf^!R\cHom_\cY(M,L).
\end{multline}

If $f$ is smooth, it follows from the construction of $\Omega_\cX$ and
$\Omega_\cY$ that $\Omega_\cX\simeq f^*\Omega_\cY(d)[2d]$, where $d$ is the
relative dimension of $f$. It follows that Proposition \ref{p.fuO} induces
an isomorphism of functors $f^*(D_\cY N)(d)[2d]\simeq D_\cX f^*N$ for all
$N\in \Dbc(\cY,\cO)$. Thus
\[Rf^!=D_\cX f^* D_\cY \simeq f^*(d)[2d].\]
\end{Construction}

\begin{Remark}\label{s.flshrO}
Let $f\colon\cX\to \cY$ be a morphism of $\ell$-prime inertia between
finite-type Deligne-Mumford $S$-stacks. Assume either (a) $f$ is
representable and \'etale, or (b) $f$ is separated and $\cX$ and $\cY$ are
of finite inertia. The isomorphisms \eqref{e.shradjO1} and
\eqref{e.shradjO2} give an isomorphism $D_\cY Rf_! \simto Rf_* D_\cX$ of
functors $\Dbc(\cX,\cO)\to \Dbc(\cY,\cO)$. Using biduality (Proposition
\ref{p.DO}), we obtain an isomorphism of functors $\Dbc(\cX,\cO)\to
\Dbc(\cY,\cO)$:
\[Rf_!\simeq D_\cY D_\cY Rf_! \simeq D_\cY Rf_* D_\cX.\]
\end{Remark}

\begin{Construction}\label{d.flshriekO}
Let $f\colon \cX\to \cY$ be a morphism of $\ell$-prime inertia between
finite-type Deligne-Mumford $S$-stacks. Thanks to Proposition \ref{p.DO}, we
can define a triangulated functor
\[Rf_!\colon \Dbc(\cX,\cO)\to \Dbc(\cY,\cO), \quad M \mapsto D_\cY Rf_* D_\cX M.\]
By Remark \ref{s.flshrO}, this definition is compatible with previous
definitions. For $M\in D_c^b(\cX,\cO)$, Propositions \ref{p.fO} and
\ref{p.DO} induce an isomorphism
  \[\Lambda_n\otimes^L_{\cO} Rf_!M \to Rf_! (\Lambda_n\otimes^L_\cO M).\]
If $f\colon \cX\to \cY$ and $g\colon \cY\to \cZ$ are two such morphisms, then we have an isomorphism of functors $\Dbc(\cX,\cO)\to \Dbc(\cY,\cO)$:
\[
  R(gf)_! = D_\cZ R(gf)_* D_\cX \simeq D_\cZ Rg_* Rf_* D_\cX
  \simeq D_\cZ Rg_* D_\cY D_\cY Rf_* D_\cX = Rg_! Rf_!.
\]
\end{Construction}

\section{Operations on $D^+_c(\cX,E)$}\label{s.6}
In this section, let $E$ be a complete discrete valuation field, and let
$\cO$ be its ring of integers.

\begin{Construction}\label{c.E}
For an $\cO$-linear category $\cC$, we let $\cC\otimes_{\cO} E$ denote the
category obtained from $\cC$ by inverting multiplication by nonzero elements
of $\cO$. More precisely, $\cC\otimes_{\cO} E$ have the same objects as
$\cC$ and for objects $X$, $Y$, $\Hom_{\cC\otimes_\cO
E}(X,Y)=\Hom_{\cC}(X,Y)\otimes_\cO E$. Clearly $\cC\otimes_{\cO} E$ is an
$E$-linear category.

Now let $\cC$ be an $\cO$-linear Noetherian Abelian category. Let $\cT$ be
the full subcategory of $\cC$ spanned by torsion objects, namely, objects
annihilated by some power of the maximal ideal of $\cO$. Then $\cT$ is a
Serre subcategory (Definition \ref{d.thick}). By Lemma \ref{l.torsion}, the
quotient category $\cC/\cT$ is equivalent to $\cC\otimes_{\cO} E$.

Let $\cX$ be a topos. The above applies to the category
$\cC=\Mod_c(\cX,\cO)$. We define $\Mod_c(\cX, E)$ to be the quotient
$\Mod_c(\cX,\cO)/\cT$. We have $\Mod_c(\cX, E)\simeq
\Mod_c(\cX,\cO)\otimes_\cO E$.

Let $D_\cT$ be the full subcategory of $D_c(\cX,\cO)$ spanned by complexes
whose cohomology sheaves belong to $\cT$. Then $D_\cT$ is a thick
triangulated subcategory. For $*\in \{\emptyset,+,b\}$, we define
$D^*_c(\cX,E)$ to be the quotient $D^*_c(\cX,\cO)/D^*_\cT$, where
$D^*_\cT=D_\cT\cap D^*$. The inclusion functors induce fully faithful
functors
\[\Mod_c(\cX,E)\to D^b_c(\cX,E)
\to D^+_c(\cX,E)\to D_c(\cX,E).
\]
Truncation and cohomology functors on $D_c(\cX,\cO)$ induce truncation and
cohomology functors on $D_c(\cX,E)$. We denote the localization functor
$D_c(\cX,\cO)\to D_c(\cX,E)$ by $M\mapsto E\otimes_\cO M$. The functor
$D_c(\cX,\cO)\otimes_\cO E\to D_c(\cX,E)$ is not an equivalence in general,
but it induces an equivalence $D^b_c(\cX,E)\simeq D^b_c(\cX,\cO)\otimes_\cO
E$. If $\cX$ has enough points, the canonical t-structure on
$D^+_c(\cX,\cO)$ induces a canonical t-structure on $D^+_c(\cX,E)$, of heart
$\Mod_c(\cX,E)$, compatible with the truncation and cohomology functors.

This applies in particular to the \'etale topos of a Noetherian
Deligne-Mumford stack $\cX$. The functor $-\otimes_\cO -$ in Construction
\ref{c.tensor} induces a functor
\[-\otimes_E-\colon D^+_c(\cX,E) \times D^+_c(\cX,E) \to D^+_c(\cX,E).\]
For any morphism $f\colon \cX\to \cY$ of Noetherian Deligne-Mumford stacks,
the functor $f^*$ in Construction \ref{c.fus} induces a functor
\[f^*\colon
D_c(\cY,E)\to D_c(\cX,E).\]
\end{Construction}

In Construction \ref{d.E} through Proposition \ref{c.ED}, let $S$ be either
a regular scheme of dimension $\le 1$ or a quasi-excellent scheme. Assume
that $E$ has characteristic $0$ and residue characteristic $\ell$ invertible
on $S$.

\begin{Construction}\label{d.E}
For any finite-type Deligne-Mumford $S$-stack, the functor $R\cHom_\cO(-,-)$
in Construction \ref{c.RcHom} induces a functor
\[
R\cHom_E(-,-)\colon D^b_c(\cX,E)^\op \times D^+_c(\cX,E) \to
D^+_c(\cX,E).
\]
Let $f\colon \cX\to \cY$ be a morphism between finite-type Deligne-Mumford
$S$-stacks. The functor $Rf_*$ in Proposition \ref{p.fadic} induces a
functor
\[Rf_*\colon D^+_c(\cX,E)\to D^+_c(\cY,E).\]
If $f$ is separated and $\cX$ and $\cY$ are of finite inertia, then the
functors $Rf_\dag$ in Construction \ref{c.dagO} induces a functor
\[
Rf_!\colon D^+_c(\cX,E)\to D^+_c(\cY,E).
\]
If, moreover, $f$ has $\ell$-prime inertia, then the functor $Rf^!$ in
Construction \ref{c.fushrO} induces a functor
\[Rf^!\colon D_c^+(\cY,E) \to D_c^+(\cX,E).\]
\end{Construction}

\begin{Proposition}\label{p.finE}
Let $f\colon \cX\to \cY$ be a morphism between finite-type Deligne-Mumford
$S$-stacks. Assume that the fibers of $f$ have dimension $\le d$.
\begin{enumerate}
 \item \label{p.finE1} The functor $Rf_*$ sends $\Dbc(\cX,E)$ to
     $\Dbc(\cY,E)$. If $f$ is proper, then the cohomological amplitude of
     $Rf_*$ is contained in $[0,2d]$.

 \item \label{p.finE2} If $f$ is separated and $\cX$ and $\cY$ are of
     finite inertia, then $Rf_!$ has cohomological amplitude contained in
     $[0,2d]$, and, in particular, sends $\Dbc(\cX,E)$ to $\Dbc(\cY,E)$.
 \end{enumerate}
\end{Proposition}

\begin{proof}
\begin{itemize}
\item[\ref{p.finE1}] We may assume $\cY$ is a scheme. For the first
    assertion, using reductions as in Proposition \ref{p.cdim}
    \ref{p.cdim1}, we may assume $\cX$ separated over $S$. Then, for both
    assertions, since $f$ factors through the coarse space of $\cX$, we
    are reduced to two cases: (a) $f$ is a proper universal homeomorphism;
    (b) $f$ is representable. Case (b) is follows from the remark
    following Proposition \ref{p.fO}. In case (a), by proper base change
    (Proposition \ref{p.sbcO} \ref{p.sbcO1}), we may assume that $\cY$ is
    the spectrum of an algebraically closes field. Then $\cX^\red\simeq
    BG$ and the assertions are clear.

\item[\ref{p.finE2}] By base change (Proposition \ref{p.dagO}
    \ref{p.dagO1}), we may assume that $\cY$ is the spectrum of a field.
    The assertion then follows from the decomposition in Propositions
    \ref{p.comp} and \ref{p.finE1}.
\end{itemize}
\end{proof}

\begin{Proposition}\label{p.luE}
Let $f\colon \cX\to \cY$ be a proper universal homeomorphism (Definition
\ref{s.uh}) between finite-type Deligne-Mumford $S$-stacks, and let $L\in
D^+_c(\cY,E)$. Then the adjunction map $L\to Rf_*f^* L$ is an isomorphism.
\end{Proposition}

\begin{proof}
By proper base change, we may assume that $\cY$ is the spectrum of an
algebraically closed field. Then $\cX^\red\simeq BG$ and the
assertion is clear.
\end{proof}

\begin{Construction}\label{c.ED}
Let $S$ be either a regular scheme of dimension $\le 1$ or an excellent
finite-dimensional scheme, endowed with a dimension function. Recall that
$E$ has characteristic $0$ and residue characteristic $\ell$ invertible on
$S$. The functor $D_\cX$ in Construction \ref{s.DO} induces a reflexive
triangulated functor
\[D_\cX\colon \Dbc(\cX,E)^\op \to \Dbc(\cX,E).\]
Let $f$ be a morphism between finite-type Deligne-Mumford $S$-stacks. The
functor $Rf^!$ in Construction \ref{d.fushriekO} induces a functor
\[Rf^!\colon \Dbc(\cY,E) \to \Dbc(\cX,E),\]
which is the restriction of the $Rf^!$ in Construction \ref{d.E} to $D^b_c$
when the latter is defined.

The finiteness of $Rf_*$ (Proposition \ref{p.finE}) enables us to define a
functor
\[Rf_!\colon \Dbc(\cX,E)\to \Dbc(\cY,E), \quad M \mapsto D_\cY Rf_* D_\cX M.\]
If $f\colon \cX\to \cY$ and $g\colon \cY\to \cZ$ are two such
morphisms, then
\[
R(gf)_! = D_\cZ R(gf)_* D_\cX \simeq D_\cZ Rg_* Rf_* D_\cX
\simeq D_\cZ Rg_* D_\cY D_\cY Rf_* D_\cX = Rg_! Rf_!.
\]
For $K\in \Dbc(\cX,E)$, $L\in \Dbc(\cY,E)$, we have
\begin{multline}\label{e.shradjE}
R\cHom_{\cY}(Rf_!K,L)\simeq D_\cY(Rf_! K\otimes D_\cY L)\simeq D_\cY(Rf_!(K\otimes f^*D_\cY L))\\
\simeq Rf_*D_\cX(K\otimes D_\cX Rf^! L) \simeq Rf_*R\cHom_\cX(K,Rf^!L).
\end{multline}

If $f$ is separated and $\cX$ and $\cY$ are of finite inertia, then $Rf_!$
is the restriction of the $Rf_!$  in Construction \ref{d.E} to $\Dbc$. In
fact, using the decomposition of $f$ in Proposition \ref{p.comp}, we are
reduced to two cases: (a) $f$ proper and quasi-finite, (b) $f$
representable. In case (a), one can repeat the argument of \cite[Corollary
5.15]{OlssonFuji}. Case (b) follows from Remark \ref{s.flshrO}.
\end{Construction}

\begin{Construction}\label{d.Eplus}
Finally let $E'$ be an algebraic extension of $E$. We define
$\Mod_c(\cX,E')$ to be the inductive 2-limit of $\Mod_c(\cX,E'')$, and
define $D^+_c(\cX,E')$ to be the inductive 2-limit of $D^+_c(\cX,E'')$,
where $E''$ runs over all finite extensions of $E$ contained in $E'$. Then
$D^+_c(\cX,E')$ admits the same operations as $D^+_c(\cX,E)$.
\end{Construction}

\section{Pushforward of cohomological correspondences}\label{s.cc}
In this section, we construct several operations on cohomological
correspondences, including the extraordinary pushforward used in the
statement of the Lefschetz-Verdier formula. Throughout this section, we fix
a coefficient ring $\Lambda$ of one of the following types:
\begin{itemize}
\item[(a)] A ring annihilated by an integer $m$;
\item[(b)] A complete discrete valuation ring of residue characteristic
    $\ell>0$;
\item[(c)] An algebraic extension of a discrete valuation field of
    characteristic $0$ and residue characteristic $\ell>0$.
\end{itemize}
Accordingly, we fix a Noetherian base scheme $S$
\begin{itemize}
\item[(a)] with no additional assumption (in particular, we do \emph{not}
    assume $m$ invertible on $S$);
\item[(b)] regular of dimension $\le 1$ or quasi-excellent
    finite-dimensional, with $\ell$ invertible on $S$;
\item[(c)] regular of dimension $\le 1$ or excellent finite-dimensional,
    with $\ell$ invertible on $S$.
\end{itemize}

\begin{Convention}\label{c.Sstack}
In Sections \ref{s.cc} and \ref{s.LV}, all $S$-stacks are assumed to be
separated finite-type Deligne-Mumford $S$-stacks that are, according to the
type of $\Lambda$,
\begin{itemize}
\item[(a)] of $m$-prime inertia;
\item[(b)] of $\ell$-prime inertia;
\item[(c)] with no additional assumption.
\end{itemize}
\end{Convention}

For an $S$-stack $X$, we consider the category $D(X,\Lambda)$ in case (a)
and the category $\Dbc(X,\Lambda)$ in cases (b) and (c). In order to get
uniform statements for all cases, we adopt the following convention in cases
(b) and (c): $D(X,\Lambda):=\Dbc(X,\Lambda)$.

To simplify notation, we omit the base $S$ in products and exterior
products. We write $f_*$, $f_!$, $f^!$, $\otimes$, $\boxtimes_\cS$ for
$Rf_*$, $Rf_!$, $Rf^!$, $\otimes^L_\Lambda$, $\boxtimes^L_\cS$,
respectively. We let $\BC$, $\PF$, $\KF$, $\adj$, $\ev$ denote base change,
projection formula, K\"unneth formula, adjunction, and evaluation maps,
respectively.

We start by reviewing the K\"unneth formula maps (Construction \ref{c.KFH}
through Proposition \ref{p.KFadj}).

\begin{Construction}\label{c.KFH}
Let
\begin{equation}\label{e.square}
\xymatrix{X_2\ar[d] & X\ar[l]_{p_2}\ar[d]^{p_1}\\
\cS & X_1\ar[l]}
\end{equation}
be a 2-commutative diagram of $S$-stacks. Let $L_i, M_i\in D(X_i,\Lambda)$,
$i=1,2$. We consider the map
\[H^0(X_1,M_1)\otimes H^0(X_2,M_2)\to H^0(X,p_1^*M_1\otimes p_2^*M_2)\]
carrying $a\otimes b$, $a\colon \Lambda\to M_1$, $b\colon \Lambda\to M_2$ to
$\Lambda\simeq p_1^*\Lambda\otimes p_2^*\Lambda\xrightarrow{p_1^*a\otimes
p_2^*b}p_1^*M_1\otimes p_2^*M_2$. We also consider the morphism
\[p_1^*R\cHom_{X_1}(L_1,M_1)\otimes p_2^* R\cHom_{X_2}(L_2,M_2)\to
R\cHom_{X}(p_1^* L_1\otimes p_2^*L_2,p_1^* M_1\otimes p_2^*M_2),
\]
composite of the obvious morphisms
\begin{multline*}
p_1^*R\cHom_{X_1}(L_1,M_1)\otimes p_2^*R\cHom_{X_2}(L_2,M_2) \to R\cHom_{X}(p_1^*L_1,p_1^*M_1)\otimes R\cHom_X(p_2^*L_2,p_2^*M_2)\\
\to R\cHom_{X}(p_1^* L_1\otimes p_2^*L_2,p_1^* M_1\otimes p_2^*M_2).
\end{multline*}
In particular, if \eqref{e.square} is 2-Cartesian, we obtain
\begin{gather}
\KF_{H^0}\colon H^0(X_1,M_1)\otimes H^0(X_2,M_2)\to H^0(X,M_1\boxtimes_\cS M_2),\notag\\
\KF\colon R\cHom_{X_1}(L_1,M_1)\boxtimes_\cS R\cHom_{X_2}(L_2,M_2)\to
R\cHom_{X}(L_1\boxtimes_{\cS}L_2,M_1\boxtimes_\cS M_2).\label{e.KF}
\end{gather}
\end{Construction}

\begin{Construction}\label{c.KFstar}
Let
\begin{equation}\label{e.cube}
\xymatrix{&X_2\ar[ld]\ar[dd]|\hole^(.3){f_2} && X\ar[ld]^{p_1}\ar[dd]^f\ar[ll]_{p_2}\\
\cS\ar[dd]_s &&\ar[ll] X_1\ar[dd]_(.3){f_1} \\
& Y_2\ar[ld] && Y\ar[ld]^{q_1}\ar[ll]|\hole_(.3){q_2}\\
\cT && Y_1\ar[ll]}
\end{equation}
be a 2-commutative diagram of $S$-stacks. Let $M_i\in D(X_i,\Lambda)$,
$i=1,2$. We consider the morphism
\[q_1^*f_{1*}M_1\otimes q_2^*f_{2*}M_2\to f_*(p_1^*M_1\otimes p_2^*M_2)\]
adjoint to the composite
\[f^*(q_1^* f_{1*}M_1\otimes q_2^* f_{2*}M_2) \simto p_1^* f_1^*f_{1*}M_1 \otimes p_2^* f_2^* f_{2*} M_2
\xrightarrow{\adj\otimes \adj} p_1^* M_1\otimes p_2^* M_2.
\]
In particular, if the top and bottom squares of \eqref{e.cube} are
2-Cartesian, we obtain
\[\KF_*\colon f_{1*}M_1\boxtimes_\cS f_{2*}M_2 \to f_*(M_1\boxtimes_\cT M_2).\]
This construction is compatible with composition of morphisms of squares.
\end{Construction}

\begin{Construction}\label{c.KFls}
For a 2-commutative diagram of $S$-stacks of the form \eqref{e.cube} with
top and bottom squares 2-Cartesian satisfying $\cS=\cT$ and $s=\id$, and
$N_i\in D(Y_i,\Lambda)$, we consider the map
\[\KF^!\colon f_1^! N_1 \boxtimes_\cS f_2^! N_2 \to f^!(N_1\boxtimes_\cS N_2)\]
adjoint to the composite
\[f_!(f_1^! N_1 \boxtimes_\cS f_2^! N_2) \xrightarrow[\sim]{\KF_!} f_{1!}f_1^!N_1 \boxtimes_\cS f_{2!}f_2^! N_2
\xrightarrow{\adj\boxtimes_{\cS} \adj} N_1\boxtimes_\cS N_2,
\]
where $\KF_!$ is given by the K\"unneth formula isomorphism for $f_!$
(\eqref{e.Kunnethlshr} or \eqref{e.KFdagger}). The construction of $\KF^!$
is compatible with composition of morphisms of 2-Cartesian squares.
\end{Construction}

K\"unneth formula maps constructed above are compatible with the
isomorphisms \eqref{e.shradj1}, \eqref{e.shradj2}, \eqref{e.shradjO2},
\eqref{e.shradjO3} in the following sense.

\begin{Proposition}\label{p.KFadj}
In the situation of Construction \ref{c.KFls}, for $M_i\in D(X_i,\Lambda)$,
$L_i,N_i\in D(Y_i,\Lambda)$, $i=1,2$, the diagram
\[
\xymatrix{\boxtimes_\cS R\cHom_{Y_i}(f_{i!}M_i,L_i)\ar[d]_{\KF}
\ar[r]^\sim
&\boxtimes_\cS f_{i*}R\cHom_{X_i}(M_i,f_i^!L_i)\ar[d]^{\KF_*}\\
R\cHom_Y(\boxtimes_\cS f_{i!}M_i,\boxtimes_\cS L_i)\ar[d]_{\KF_!}^{\simeq} & f_*\boxtimes_\cS R\cHom_{X_i}(M_i,f_i^!L_i)\ar[r]^{\KF}
& f_*R\cHom_X(\boxtimes_\cS M_i, \boxtimes_\cS f_i^! L_i)\ar[d]^{\KF^!}\\
R\cHom_Y(f_!\boxtimes_\cS M_i,\boxtimes_\cS L_i)\ar[rr]^\sim && f_*R\cHom_X(\boxtimes_\cS M_i,f^!\boxtimes_\cS L_i)}
\]
commutes. Here the top and bottom horizontal arrows are both of type
\eqref{e.shradj1} in case (a), of type \eqref{e.shradjO2} in case (b), and
of type \eqref{e.shradjE} in case (c). Moreover, the diagram
\[
\xymatrix{\boxtimes_\cS R\cHom_{X_i}(f_i^*N_i,f_i^!L_i)\ar[d]_{\KF}\ar[rr]^\sim
&& \boxtimes_\cS f_i^! R\cHom_{Y_i}(N_i,L_i)\ar[d]^{\KF^!}\\
R\cHom_X(\boxtimes_\cS f_i^* N_i, \boxtimes_\cS f_i^! L_i)\ar[r]^{\KF^*}_\sim
& R\cHom_X(f^*\boxtimes_\cS N_i, \boxtimes_\cS f_i^! L_i)\ar[d]_{\KF^!}
& f^!\boxtimes_\cS R\cHom_{Y_i}(N_i,L_i)\ar[d]^{\KF}\\
& R\cHom_X(f^*\boxtimes_\cS N_i,f^! \boxtimes_\cS L_i)\ar[r]^\sim & f^!R\cHom_Y(\boxtimes_\cS N_i,\boxtimes_\cS L_i)}
\]
commutes. Here the top and bottom horizontal arrows are both of type
\eqref{e.shradj2} in case (a), of type \eqref{e.shradjO3} in case (b), and
of type \eqref{e.shradjO3p} in case (c).
\end{Proposition}

\begin{proof}
The verification by adjunction is straightforward.
\end{proof}

We will need the following generalizations of the base change and K\"unneth
formula maps for $f_!$.

\begin{Construction}\label{c.GBC}
Let
\[\xymatrix{X'\ar[rrd]^h\ar[rd]^(.8)r\ar[ddr]_{f'} \\
& Z\ar[r]_q\ar[d]^p & X\ar[d]^f\\
&Y'\ar[r]^g & Y}
\]
be a 2-commutative diagram of $S$-stacks with 2-Cartesian square such that
$r$ is proper. We define the generalized base change map
\[\GBC\colon g^* f_! \to f'_! h^*\]
of functors $D(X,\Lambda)\to D(Y',\Lambda)$ to be the composite
\[g^*f_!\xrightarrow[\sim]{\BC} p_! q^* \xrightarrow{\adj} p_!r_* r^* q^* \simeq p_!r_!r^* q^* \simeq f'_! h^*.\]
By adjunction, we obtain the generalized base change map
\[\GBC\colon h_*f'^! \to f^! g_*\]
of functors $D(Y',\Lambda)\to D(X, \Lambda)$. The maps $\GBC$ are compatible
with horizontal and vertical compositions of squares.
\end{Construction}

\begin{Construction}\label{c.GKF}
Consider a 2-commutative diagram of $S$-stacks of the form \eqref{e.cube}
satisfying $\cS=\cT$, $s=\id$. Assume that the morphisms $X\to
X_1\times_{\cS}X_2$, $Y\to Y_1\times_{\cS} Y_2$ are proper. Then we have
2-commutative diagrams with 2-Cartesian squares
\[\xymatrix{X\ar@/^1pc/[rrrd]^{p_2}\ar@/_1pc/[dddr]_{p_1}\ar[rd]^(.7)r\ar[rrd]^{h_2}\ar[rdd]_{h_1} &&&& X\ar@/^.6pc/[rr]^{f}\ar[r]_r\ar[rd] & Z\ar[r]_{f'}\ar[d] & Y\ar[d]\\
&Z\ar[rd]^(.7){f'}\ar[r]_{h'_2}\ar[d]^{h'_1}& X_2\times_{Y_2} Y\ar[r]^{t_2}\ar[d]^{g_2} & X_2\ar[d]^{f_2}&& X_1\times_\cS X_2\ar[r] & Y_1\times_\cS Y_2\\
&X_1\times_{Y_1} Y\ar[r]^{g_1}\ar[d]_{t_1} & Y \ar[r]_{q_2}\ar[d]^{q_1} & Y_2\\
&X_1\ar[r]^{f_1} & Y_1}
\]
where $r$ is proper. Let $M_i\in D(X_i,\Lambda)$, $i=1,2$. We define the
generalized K\"unneth formula map
\[q_1^*f_{1!}M_1\otimes q_2^* f_{2!}M_2\to f_!(p_1^* M_1\otimes p_2^* M_2)\]
to be the map given by the commutative diagram
\[\xymatrix{ & g_{1!}(t_1^*M_1\otimes g_1^*g_{2!}t_2^* M_2)\ar[r]^\sim_{\BC}\ar@/^1pc/[rr]^{\GBC}
&g_{1!}(t_1^*M_1\otimes h'_{1!}p'^*_2M_2) \ar[r]_{\adj_r}\ar[d]^\simeq_{\PF_{h'_1}}
&g_{1!}(t_1^*M_1\otimes h_{1!}p^*_2M_2) \ar[d]^\simeq_{\PF_{h_1}}\\
q_1^*f_{1!}M_1\otimes q_2^* f_{2!}M_2 \ar[r]^{\BC\otimes \BC}_\sim
& g_{1!} t_1^* M_1 \otimes g_{2!}t_2^* M_2\ar[u]_{\PF_{g_1}}^\simeq \ar[d]^{\PF_{g_2}}_\simeq\ar[r]^{\KF_!}_\sim
& f'_!(p'^*_1 M_1\otimes p'^*_2 M_2)\ar[r]^{\adj_r}
& f_!(p_1^* M_1\otimes p_2^* M_2)\\
& g_{2!}(g_2^*g_{1!}t_1^* M_1\otimes t_2^* M_2)\ar[r]_\sim^{\BC}\ar@/_1pc/[rr]_{\GBC}
&g_{2!}(h'_{2!}p'^*_1M_1 \otimes t_2^* M_2) \ar[r]^{\adj_r}\ar[u]_\simeq^{\PF_{h'_2}}
&g_{2!}(h_{2!}p^*_1M_1 \otimes t_2^* M_2) \ar[u]_\simeq^{\PF_{h_2}},}
\]
where $p'_i=t_ih'_i\colon Z\to X_i$, $i=1,2$, and $\KF_!$ is the K\"unneth
formula isomorphism for $f'_!(t_1^*M_1\boxtimes_Y t_2^*M_2)$.
\end{Construction}

\begin{Definition}\label{d.corr}
A \emph{correspondence} (or an \emph{$S$-correspondence} to be more precise)
is a morphism of $S$-stacks $b=(b_1,b_2)\colon B\to X_1\times X_2$, where
$X_1$ and $X_2$ are $S$-stacks. Recall that $X_1\times X_2$ stands for
$X_1\times_S X_2$ by the convention of this section.

Let $b=(b_1,b_2)\colon B\to X_1\times X_2$, $c=(c_1,c_2)\colon C\to
Y_1\times Y_2$ be correspondences. A morphism of correspondences $b\to c$ is
a triple $g^\sharp=(g,(f_1,f_2))$ of morphisms of $S$-stacks such that the
following diagram commutes
\begin{equation}\label{e.corr}
\xymatrix{X_1\ar[d]_{f_1}&\ar[l]_{b_1} B \ar[d]^g \ar[r]^{b_1} & X_2\ar[d]^{f_2}\\
Y_1 & \ar[l]_{c_1} C \ar[r]^{c_2} & Y_2.}
\end{equation}

The morphism of correspondences $g^\sharp$ is called \emph{left-proper}
(resp.\ \emph{right-proper}) if the morphism $B\to X_1\times_{Y_1} C$
(resp.\ $B\to X_2\times_{Y_2} C$) is proper. The morphism $g^\sharp$ is
called of type (i) if it is left-proper and $f_2$ is an identity morphism.
The morphism $g^\sharp$ is called of type (ii) if $f_1$ and $g$ are both
identities.

The \emph{transpose} of a correspondence $b=(b_1,b_2)\colon B\to X_1\times
X_2$ is
\[b^T=(b_2,b_1)\colon B\xrightarrow{b} X_1\times X_2\simeq X_2\times X_1.\]
The transpose of a morphism of correspondences $g^\sharp=(g,(f_1,f_2))\colon
b\to c$ is $g^{\sharp T}=(g,(f_2,f_1))\colon b^T\to c^T$.
\end{Definition}

\begin{Remark}
If $g$ is proper, then $g^\sharp$ is left-proper and right-proper. Type (i)
morphisms and type (ii) morphisms of correspondences are left-proper.
Conversely, every left-proper morphism of correspondences $g^\sharp$ is
canonically the composite of a morphism of type (ii) followed by a morphism
of type (i):
\[\xymatrix{X_1\ar@{=}[d] & \ar[l]_{b_1} B\ar@{=}[d]\ar[r]^{b_1} & X_2\ar[d]^{f_2}\\
X_1\ar[d]_{f_1} & B\ar[l]_{b_1}\ar[d]^g \ar[r] & Y_2\ar@{=}[d]\\
Y_1 & C\ar[l]_{c_1} \ar[r]^{c_2} & Y_2.}
\]
\end{Remark}

\begin{Definition}\label{d.decomp}
Let $b'\colon B'\to X_1\times X_2$, $b''\colon B''\to X_2\times X_1$,
$c'\colon C'\to Y_1\times Y_2$, $c''\colon C''\to Y_2\times Y_1$ be
correspondences. A morphism $(b',b'')\to (c',c'')$ is a datum
$(g',g'')^\sharp= (g',g'',(f_1,f_2))$ such that $(g',(f_1,f_2))\colon b'\to
c'$ and $(g'',(f_2,f_1))\colon b''\to c''$ are morphisms of correspondences.

The morphism $(g',g'')^\sharp$ is called of type (I) if $g'^\sharp$ is of
type (i) and $g''^\sharp$ is of type (ii). The morphism $(g',g'')^\sharp$ is
called of type (II) if $g'^\sharp$ is of type (ii) and $g''^\sharp$ is of
type (i). The morphism $(g',g'')^\sharp$ is called \emph{left-decomposable}
if it is isomorphic to a composite $(g'_1,g''_1)^\sharp \dotsm
(g'_n,g''_n)^\sharp$ such that for each $1\le i\le n$, $(g'_i,g''_i)^\sharp$
is of type (I) or (II). The morphism $(g',g'')^\sharp$ is called
\emph{right-decomposable} if its transpose $(g', g'',(f_2,f_1))\colon
(b'^T,b''^T)\to (c'^T,c''^T)$ is left-decomposable.
\end{Definition}

\begin{Proposition}
Let $b'\colon B'\to X_1\times X_2$, $b''\colon B''\to X_2\times X_1$,
$c'\colon C'\to Y_1\times Y_2$, $c''\colon C''\to Y_2\times Y_1$ be
correspondences and let $(g',g'')^\sharp\colon (b',b'')\to (c',c'')$ be a
morphism. If either $g'$ is proper and $g''^\sharp$ is left-proper, or $g''$
is proper and $g'^\sharp$ is left-proper, then $(g',g'')^\sharp$ is
left-decomposable. In particular, if $g'$ and $g''$ are proper, then
$(g',g'')^\sharp$ is left-decomposable and right-decomposable.
\end{Proposition}

\begin{proof}
By symmetry, it suffices to treat the case $g'$ proper and $g''^\sharp$
left-proper. In this case $(g',g'')^\sharp$ can be decomposed into three
morphisms, of types (I), (II), (I), respectively:
\[
\xymatrix{X_1\ar[d]_{f_1} &  B'\ar[l]\ar@{=}[d] \ar[r] & X_2\ar@{=}[d] & B''\ar[l]\ar[r]\ar@{=}[d] & X_1\ar[d]^{f_1}\\
Y_1\ar@{=}[d] & B'\ar[l]\ar@{=}[d]\ar[r] & X_2\ar[d]^{f_2} & B''\ar[l]\ar[r]\ar[d]^{g''} & Y_1\ar@{=}[d]\\
Y_1\ar@{=}[d] & B'\ar[l]\ar[d]^{g'}\ar[r] & Y_2\ar@{=}[d] & C''\ar[l]\ar@{=}[d]\ar[r] & Y_1\ar@{=}[d]\\
Y_1 & C'\ar[l]\ar[r] & Y_2 & C''\ar[l]\ar[r] & Y_1.}
\]
\end{proof}

\begin{Definition}\label{d.cohcorr}
Let $b=(b_1,b_2)\colon B\to X_1\times X_2$ be a correspondence, and let
$L_i\in D(X_i,\Lambda)$, $i=1,2$. A \emph{$b$-morphism} from $L_1$ to $L_2$
is a map $b_1^* L_1\to b_2^! L_2$ in $D(B,\Lambda)$.
\end{Definition}

A $b$-morphism is also known as a cohomological correspondence \cite[III
3.2]{SGA5}.

\begin{Notation}\label{n.K}
For an $S$-stack $X$, we let $K_X=a^!\Lambda_S$, where $a\colon X\to S$ is
the structure morphism. In this section, we do not use the duality functors
defined in previous sections (except implicitly in the definition of $f^!$
in case (c)). Instead, we let $D_X=R\cHom_X(-,K_X)$.
\end{Notation}

\begin{Remark}\label{r.P}
For $b$ and $L_i\in D(X_i,\Lambda)$, $i=1,2$ as in Definition
\ref{d.cohcorr}, the $b$-morphisms from $L_1$ to $L_2$ form a
$\Lambda$-module $\Hom_B(b_1^*L_1,b_2^!L_2)$. We will use the isomorphism
(\eqref{e.shradj2} in case (a), \eqref{e.shradjO3} in case (b),
\eqref{e.shradjO3p} in case (c))
\begin{equation}\label{e.P1}
R\cHom_B(b_1^*L_1,b_2^!L_2)\simeq b^!R\cHom_{X_1\times
X_2}(p_1^*L_1,p_2^!L_2),
\end{equation}
where $p_i\colon X_1\times X_2\to X_i$, $i=1,2$ are the projections. We will
also use the composite morphism
\begin{equation}\label{e.P2}
D_{X_1} L_1 \boxtimes L_2 \xrightarrow{\KF} R\cHom_{X_1\times X_2} (L_1\boxtimes \Lambda_{X_2}, K_{X_1}\boxtimes L_2)
\xrightarrow{\KF^!} R\cHom_{X_1\times X_2}(p_1^*L_1,p_2^!L_2).
\end{equation}
\end{Remark}

\begin{Construction}[Dual]\label{c.dual}
Let $b=(b_1,b_2)\colon B\to X_1\times X_2$ be a correspondence, and let
$L_i\in D(X_i,\Lambda)$, $i=1,2$. We define
\[D_b\colon R\cHom_B(b_1^*L_1,b_2^!D_{X_2}L_2)\to R\cHom_B(b_2^*L_2,b_1^!D_{X_1}L_1)\]
to be the composite
\begin{multline*}
R\cHom_B(b_1^*L_1,b_2^!D_{X_2}L_2)\xrightarrow{\alpha} R\cHom_B(D_{B}b_2^!D_{X_2}L_2, D_B b_1^*L_1) \simeq R\cHom_B(D_BD_B b_2^*L_2,b_1^!D_{X_1}L_1) \\
\xrightarrow{\beta} R\cHom_B(b_2^*L_2,b_1^!D_{X_1}L_1),
\end{multline*}
where $\alpha$ is the evaluation morphism $R\cHom_B(-,-)\to
R\cHom_B(D_B-,D_B-)$ and $\beta$ is induced by the evaluation morphism
$\id\to D_BD_B$. We have $D_{b^T} D_b=\id$ and $D_bD_{b^T}=\id$, so that
$D_b$ is an isomorphism. Applying $H^0(B,-)$, we obtain an isomorphism
\[D_b\colon \Hom_B(b_1^*L_1,b_2^!D_{X_2}L_2)\simto \Hom_B(b_2^*L_2,b_1^!D_{X_1}L_1).\]
\end{Construction}

\begin{Construction}[Extraordinary pushforward]\label{c.shriek}
Let $g^\sharp$ be a left-proper morphism of correspondences of the form
\eqref{e.corr}, and let $L_i\in D(X_i,\Lambda)$, $i=1,2$. We define
\[g^\sharp_!\colon g_*R\cHom_B(b_1^*L_1,b_2^!L_2)\to R\cHom_C(c_1^*f_{1!}L_1,c_2^!f_{2!}L_2)\]
to be the composite
\begin{multline*}
  g_*R\cHom_B(b_1^*L_1,b_2^!L_2) \xrightarrow{\adj} g_*R\cHom_B(b_1^* L_1,b_2^!f_2^!f_{2!}L_2)
  \simeq g_*R\cHom_B(b_1^*L_1,g^!c_2^!f_{2!}L_2)\\
  \simeq R\cHom_C(g_!b_1^*L_1,c_2^!f_{2!}L_2)\xrightarrow{\GBC} R\cHom_C(c_1^*f_{1!}L_1,c_2^!f_{2!}L_2),
\end{multline*}
where $\GBC$ is defined in Construction \ref{c.GBC}. Taking $H^0(C,-)$, we
obtain a map
\[g^\sharp_!\colon \Hom_B(b_1^*L_1,b_2^!L_2)\to \Hom_C(c_1^*f_{1!}L_1,c_2^!f_{2!}L_2).\]
The maps $g^\sharp_!$ are compatible with composition of morphisms of
correspondences.
\end{Construction}

\begin{Construction}[Pushforward]\label{c.star}
Let $g^\sharp$ be a right-proper morphism of correspondences of the form
\eqref{e.corr}, and let $L_i\in D(X_i,\Lambda)$, $i=1,2$. We define
\[g^\sharp_*\colon g_*R\cHom_B(b_1^*L_1,b_2^!L_2)\to R\cHom_C(c_1^*f_{1*}L_1,c_2^!f_{2*}L_2)\]
to be the composite
\begin{multline*}
g_*R\cHom_B(b_1^*L_1,b_2^!L_2)\xrightarrow{\adj} g_*R\cHom_B(b_1^*f_1^*f_{1*}L_1,b_2^!L_2)\simeq g_*R\cHom_B(g^*c_1^*f_{1*}L_1,b_2^!L_2) \\
\simeq R\cHom_C(c_1^*f_{1*}L_1,g_*b_2^!L_2) \xrightarrow{\GBC} R\cHom_C(c_1^*f_{1*}L_1,c_2^!f_{2*}L_2).
\end{multline*}
Taking $H^0(C,-)$, we obtain a map
\[g^\sharp_*\colon \Hom_B(b_1^*L_1,b_2^!L_2)\to \Hom_C(c_1^*f_{1*}L_1,c_2^!f_{2*}L_2).\]
The maps $g^\sharp_*$ are compatible with composition of morphisms of
correspondences.
\end{Construction}

Extraordinary pushforward and pushforward (Constructions \ref{c.shriek} and
\ref{c.star}) are dual to each other via Construction \ref{c.dual}. More
precisely, we have the following.

\begin{Proposition}\label{r.dualpush}
Let $g^\sharp$ be a left-proper morphism  of correspondences of the form
\eqref{e.corr}, and let $L_i\in D(X_i,\Lambda)$, $i=1,2$. Then the diagram
\[
\xymatrix{g_*R\cHom_B(b_1^*L_1,b_2^!D_{X_2}L_2)\ar[r]^{D_b}_\sim\ar[d]_{g^{\sharp }_!}
& g_*R\cHom_B(b_2^*L_2,b_1^!D_{X_1}L_1)\ar[d]^{g^{\sharp T}_*}\\
R\cHom_C(c_1^*f_{1!}L_1,c_2^!f_{2!}D_{X_2}L_2) \ar[d]_{\alpha}
& R\cHom_C(c_2^*f_{2*}L_2,c_1^!f_{1*}D_{X_1}L_1)\ar[d]_\simeq\\
R\cHom_C(c_1^*f_{1!}L_1,c_2^!D_{Y_2}f_{2*}L_2)\ar[r]^{D_c}_\sim
&R\cHom_C(c_2^*f_{2*}L_2,c_1^!D_{Y_1}f_{1!}L_1),}
\]
commutes. Here $\alpha$ is induced by the composite map
\begin{equation}\label{e.dualpush2}
f_{2!}D_{X_2}\xrightarrow{\ev} D_{Y_2}D_{Y_2}f_{2!}D_{X_2}\simeq D_{Y_2}f_{2*}D_{X_2}D_{X_2}\xrightarrow{\ev}D_{Y_2}f_{2*}
\end{equation}
of functors $D(X_2,\Lambda)^\op\to D(Y_2,\Lambda)$.
\end{Proposition}

\begin{proof}
The diagram can be decomposed as\tiny
\[\xymatrix@C=1em@M=0pt{g_*R\cHom(b_1^*L_1,b_2^!DL_2)\ar[dd]_{\adj}\ar[rd]_{\adj}\ar[rr]^\ev
&&g_* R\cHom(Db_2^!DL_2,Db_1^*L_1)\ar[r]^\sim\ar[d]^{\adj}
& g_*R\cHom(DDb_2^*L_2,b_1^!DL_1)\ar[r]^\ev\ar[dd]^{\adj}
& g_*R\cHom(b_2^*L_2,b_1^!DL_1)\ar[dd]^{\adj}\\
&g_*R\cHom(b_1^*L_1,b_2^!Df_2^*f_{2*}L_2)\ar[r]^\ev
&g_*R\cHom(Db_2^!Df_2^*f_{2*}L_2,Db_1^*L_1)\ar[rd]^\simeq\\
g_*R\cHom(b_1^*L_2,b_2^!f_2^!f_{2!}DL_2)\ar[r]^\beta\ar[d]_\simeq
&g_*R\cHom(b_1^*L_1,b_2^!f_2^!Df_{2*}L_1)\ar[r]^\ev\ar[d]^\simeq\ar[u]^\simeq
&g_*R\cHom(Db_2^!f_2^!Df_{2*}L_2,Db_1^*L_1)\ar[u]^\simeq\ar[d]^\simeq
&g_*R\cHom(DDb_2^*f_2^*f_{2*}L_2,b_1^!DL_1)\ar[r]^\ev\ar[d]^\simeq
&g_*R\cHom(b_2^*f_2^*f_{2*}L_2,b_1^!DL_1)\ar[d]^\simeq\\
g_*R\cHom(b_1^*L_1,g^!c_2^!f_{2!}DL_2)\ar[r]^\beta\ar[d]_\simeq
&g_*R\cHom(b_1^*L_1,g^!c_2^!Df_{2*}L_2)\ar[r]^\ev\ar@{}[rrrd]|{\text{(C)}}\ar[d]^\simeq
&g_*R\cHom(Dg^!c_2^!Df_{2*}L_1,Db_1^*L_1)\ar[r]^\sim
&g_*R\cHom(DDg^*c_2^*f_{2*}L_2,b_1^!DL_1)\ar[r]^\ev
&g_*R\cHom(g^*c_2^*f_{2*}L_2,b_1^!DL_1)\ar[d]^\simeq\\
R\cHom(g_!b_1^*L_1,c_2^!f_{2!}DL_2)\ar[r]^\beta\ar[d]_\GBC
&R\cHom(g_!b_1^*L_1,c_2^!Df_{2*}L_2)\ar[r]^\ev\ar[d]^\GBC
&R\cHom(Dc_2^!Df_{2*}L_2,Dg_!b_1^*L_1)\ar[r]^\sim\ar[d]^\GBC
&R\cHom(DDc_2^*f_{2*}L_2,g_*b_1^!DL_1)\ar[r]^\ev\ar[d]^\GBC
&R\cHom(c_2^*f_{2*}L_2,g_*b_1^!DL_1)\ar[d]^\GBC\\
R\cHom(c_1^*f_{1!}L_1,c_2^!f_{2!}DL_2)\ar[r]^\alpha
& R\cHom(c_1^*f_{1!}L_1,c_2^!Df_{2*}L_2)\ar[r]^\ev
&R\cHom(Dc_2^!Df_{2*}L_2,Dc_1^*f_{1!}L_1)\ar[rd]^\simeq
&R\cHom(DDc_2^*f_{2*}L_2,c_1^!f_{1*}DL_1)\ar[r]^\ev
&R\cHom(c_2^*f_{2*}L_2,c_1^!f_{1*}DL_1)\\
&&&R\cHom(DDc_2^*f_{2*}L_2,c_1^!Df_{1!}L_1)\ar[r]^\ev\ar[u]^\simeq
&R\cHom(c_2^*f_{2*}L_2,c_1^!Df_{1!}L_1),\ar[u]_\simeq}
\]
\normalsize where the arrows marked with $\beta$ are induced by
\eqref{e.dualpush2}. The square (C) commutes by the following lemma while
the other inner cells trivially commute.
\end{proof}

\begin{Lemma}
Let $g\colon B\to C$ be a morphism of $S$-stacks, and let $M\in
D(B,\Lambda)$, $L\in D(C,\Lambda)$. Then the following diagram commutes
\[\xymatrix{R\cHom_C(g_!M, D_L)\ar[r]^\sim\ar[d]_\ev
&g_*R\cHom_B(M,g^!D_C L)\ar[d]^\ev\\
R\cHom_C(D_CD_CL,D_Cg_!M)\ar[d]_\simeq
&g_*R\cHom_B(D_Bg^!D_CL,D_BM)\ar[d]^\simeq\\
R\cHom_C(D_CD_CL,g_*D_BM)\ar[d]_\ev
&g_*R\cHom_B(D_BD_Bg^*L,D_BM)\ar[d]^\ev\\
R\cHom_C(L,g_*D_BM)\ar[r]^\sim
&g_*R\cHom_B(g^*L,D_BM).}
\]
\end{Lemma}

\begin{proof}
The diagram is adjoint to
\[\xymatrix{g_*R\cHom_B(M,g^!D_CL)\otimes L\ar[r]^{\KF_*}
&g_*(R\cHom_B(M,g^!D_C L)\otimes g^*L)\ar[d]^{\KF}\\
R\cHom_C(g_!M,D_CL)\otimes L\ar[u]^\simeq\ar[d]_{\KF}
&g_*R\cHom_B(M,g^!D_CL\otimes g^*L)\ar[d]^{\KF^!}\\
R\cHom_C(g_!M,D_C L\otimes L)\ar[r]^\sim\ar[d]_\ev
&g_*R\cHom_B(M,g^!(D_C L\otimes L))\ar[d]^\ev\\
D_C g_!M\ar[r]^\sim
&g_*D_B M,}
\]
where the upper square commutes by Proposition \ref{p.KFadj} and the lower
square trivially commutes.
\end{proof}

To conclude this section, we prove compatibility results for duality and
extraordinary pushforward (Constructions \ref{c.dual} and \ref{c.shriek})
with the morphisms introduced in Remark \ref{r.P}. In the case where
\eqref{e.P2} is an isomorphism, these results give interpretations of
duality and extraordinary pushforward in terms of external tensor products.

\begin{Proposition}\label{p.interdual}
Let $b\colon B\to X=X_1\times X_2$ be a correspondence, and let $L_i\in
D(X_i,\Lambda)$, $i=1,2$. Let $P=D_{X_1}L_1\boxtimes D_{X_2}L_2$,
$Q=D_{X_2}L_2\boxtimes D_{X_1}L_1$, $\tilde
P=R\cHom_X(p_1^*L_1,p_2^!D_{X_2}L_2)$, $\tilde
Q=R\cHom_X(p_2^*L_2,p_1^!D_{X_1}L_1)$. Then the following diagram commutes
\[\xymatrix{R\cHom_B(b_1^*L_1,b_2^!D_{X_2}L_2)\ar[r]^-\sim\ar[d]_{D_b}^\simeq &
b^!\tilde P\ar[d]_{D_{\id_X}}^\simeq & b^! P\ar[l]\ar[d]^\simeq\\
R\cHom_B(b_2^*L_2,b_1^!D_{X_1}L_1)\ar[r]^-\sim & b^!\tilde Q & b^!Q.\ar[l]}
\]
Here the left horizontal arrows are given by \eqref{e.P1} and the right
horizontal arrows are given by \eqref{e.P2}.
\end{Proposition}

\begin{proof}
By the following lemma, the left square commutes. For the right square, the
two composite maps $P\to \tilde Q$ are both adjoint to the map
$D_{X_1}L_1\boxtimes (D_{X_2}L_2\otimes L_2)\to D_{X_1}L_1\boxtimes
K_{X_2}$.
\end{proof}

\begin{Lemma}
Let $b\colon B\to X$ be a morphism of $S$-stacks, and let $L_i\in
D(X,\Lambda)$, $i=1,2$. Then the following diagram commutes
\[\xymatrix{R\cHom_B(b^*L_1,b^!D_{X}L_2)\ar[d]_\ev\ar[r]^-\sim & b^!R\cHom_X(L_1,D_{X}L_2)\ar[d]^\ev\\
R\cHom_B(D_B b^! D_X L_2, D_B b^* L_1)\ar[d]_\simeq & b^!R\cHom_X(D_XD_XL_2,D_X L_1)\ar[dd]^\ev\\
R\cHom_B(D_BD_B b^*L_2,b^!D_XL_1)\ar[d]_\ev\\
R\cHom_B(b^*L_2,b^!D_{X}L_1)\ar[r]^-\sim & b^!R\cHom_X(L_2,D_X L_1).}
\]
\end{Lemma}

\begin{proof}
The diagram is adjoint to
\[\xymatrix{R\cHom_B(b^*L_1,b^!D_XL_2)\otimes b^* L_2\ar[r]^-\sim\ar[d]_{\KF}
& b^!R\cHom_X(L_1,D_XL_2)\otimes b^*L_2\ar[d]^{\KF^!}\\
R\cHom_B(b^*L_1,b^!D_XL_2\otimes b^*L_2)\ar[d]_{\KF^!}
&b^!(R\cHom_X(L_1,D_X L_2)\otimes L_2)\ar[d]^{\KF}\\
R\cHom_B(b^*L_1,b^!(D_XL_2\otimes L_2))\ar[r]^-\sim\ar[d]_\ev
&b^!R\cHom_X(L_1,D_XL_2\otimes L_2)\ar[d]^\ev\\
D_Bb^*L_1\ar[r]^\sim & b^! D_XL_1,}
\]
where the upper square commutes by Proposition \ref{p.KFadj} and the lower
square trivially commutes.
\end{proof}

\begin{Proposition}\label{p.inter}
Let $g^\sharp$ be a morphism of correspondences of the form \eqref{e.corr},
and let $L_i\in D(X_i,\Lambda)$, $i=1,2$. Let $f=f_1\times f_2\colon
X=X_1\times X_2 \to Y_1\times Y_2=Y$, $P=D_{X_1}L_1\boxtimes L_2$,
$E=D_{X_1}(f_{1!}L_1)\boxtimes L_2$, $\tilde
P=R\cHom_{X}(p_{1,X}^*L_1,p_{2,X}^!L_2)$, $\tilde
E=R\cHom_{Y}(p_{1,Y}^*f_{1!}L_1,p_{2,Y}^!L_2)$. If $g^\sharp$ is of type (i)
(Definition \ref{d.corr}), then $f^\sharp_!\colon f_*\tilde P\to \tilde E$
is an isomorphism, where $f^\sharp=(f,(f_1,f_2))$, and the diagram
\begin{equation}\label{e.inter1}
\xymatrix{g_*R\cHom_B(b_1^*L_1,b_2^!L_2)\ar[r]^-\sim\ar[dd]_{g^\sharp_!}
& g_*b^!\tilde P\ar[d]^{\GBC} & g_* b^! P\ar[d]^{\GBC}\ar[l]\\
& c^!f_*\tilde P\ar[d]^{f^\sharp_!}_\simeq
&c^!f_* P \ar[l]\\
R\cHom_C(c_1^*f_{1!}L_1,c_2^!L_2)\ar[r]^-\sim
& c^!\tilde E
& c^!E\ar[l]\ar[u]_\alpha}
\end{equation}
commutes. Here the left horizontal arrows are given by \eqref{e.P1}, the
right horizontal arrows are given by \eqref{e.P2}, and $\alpha$ is the
composite
\[E=D_{Y_1}(f_{1!}L_1)\boxtimes L_2 \simto f_{1*} D_{X_1} L_1\boxtimes L_2
\xrightarrow{\KF_*} f_*(D_{X_1} L_1\boxtimes L_2)= f_*P.
\]
If $g^\sharp$ is of type (ii) (Definition \ref{d.corr}), then the diagram
\begin{equation}\label{e.inter2}
\xymatrix{R\cHom_B(b_1^*L_1,b_2^!L_2)\ar[d]_{g^\sharp_!}
& b^!\tilde P\ar[l]_-\sim & b^!P\ar[r]^{\adj}\ar[l]
& b^!f^!f_!P\ar[d]^{\beta}_\simeq\\
R\cHom_C(c_1^*L_1,c_2^!f_{2!}L_2) &c^!\tilde E\ar[l]_-\sim &&c^!E\ar[ll]}
\end{equation}
commutes. Here the leftmost horizontal arrows are given by the inverse of
\eqref{e.P1}, the unmarked arrows are given by \eqref{e.P2}, and $\beta$ is
given by $\KF_!\colon f_!P\simto E$.
\end{Proposition}

\begin{proof}
In case (i), the first assertion follows immediately from the definitions.
The upper-right square of \eqref{e.inter1} trivially commutes and the
lower-right square is induced from
\[\xymatrix{
 f_*R\cHom_X(p_{1,X*}L_1,p_{2,X^!}L_2)\ar[d]^\simeq
& f_*(DL_1\boxtimes L_2)\ar[l]\\
 R\cHom_Y(f_!p_{1,X}^*L_1,p_{2,Y!} L_2)\ar[d]^\simeq_{\BC}
& f_{1*} DL_1\boxtimes L_2\ar[u]_{\KF_*}\\
R\cHom_Y(p_{1,Y}^*f_{1!}L_1,p_{2,Y}^!L_2)
&D(f_{1!}L_1)\boxtimes L_2,\ar[l]\ar[u]_{\simeq}}
\]
which commutes by Proposition \ref{p.KFadj}. Moreover, the left square of
\eqref{e.inter1} can be decomposed into
\begin{equation}\label{e.inter3}
\xymatrix{g_*R\cHom_B(b_1^*L_1,b_2^!L_2)\ar[r]^\sim\ar[d]_\simeq
&g_*b^!R\cHom_X(p_{1,X}^*L_1,p_{2,X}^!L_2)\ar[r]^{\GBC}
& c^!f_* R\cHom_X(p_{1,X}^*L_1,p_{2,X}^!L_2)\ar[d]^{\simeq}\\
R\cHom_C(g_!b_1^*L_1,c_2^!L_2)\ar[r]^{\GBC}\ar[rd]_{\GBC}
&R\cHom_C(c^*f_!p_{1,X}^*L_1,c_2^! L_2)\ar[r]^\sim\ar[d]^{\BC}_\simeq
&c^!R\cHom_Y(f_!p_{1,X}^*L_1,p_{2,Y}^!L_2)\ar[d]^{\BC}_\simeq\\
&R\cHom_C(c_1^*f_{1!}L_1,c_2^!L_2)\ar[r]^\sim
& c^!R\cHom_Y(p_{1,Y}^*f_{1!}L_1,p_{2,Y}^!L_2).}
\end{equation}
In fact, by definition, $g^\sharp_!$ is the composite of the left vertical
arrow and the oblique arrow of \eqref{e.inter3} while $f^\sharp_!$ is the
composite of the right vertical arrows of \eqref{e.inter3}. The triangle and
the lower right square of \eqref{e.inter3} trivially commute. The upper
square of \eqref{e.inter3} corresponds by adjunction to the square
\[\xymatrix{b_{2!}(g^*M\otimes b^*M') &\ar[l]_{\PF_b}^\sim p_{2,X!}(b_! g^*M\otimes M')
& p_{2,X!}(f^*c_!M\otimes M')\ar[l]_{\GBC}\\
c_{2!}(M\otimes g_!b^*M')\ar[u]^{\PF_g}_\simeq
& c_{2!}(M\otimes c^*f_!M')\ar[l]_{\GBC}
& p_{2,Y!}(c_! M\otimes f_!M')\ar[l]_{\PF_c}^\sim \ar[u]_{\PF_f}^\simeq,}
\]
which is given by applying $p_{2,Y!}$ to two equivalent definitions of the
generalized K\"unneth formula map (Construction \ref{c.GKF}) associated to
the morphism of squares from
\[\xymatrix{B\ar[r]^b\ar[d]_g & X\ar[d]^f\\ C\ar[r]^c & Y}\]
to the identity square of value $Y$. Here $M\in D(C,\Lambda)$, $M'=p_{1,X}^*
L_1\in D(X,\Lambda)$.

In case (ii), the diagram \eqref{e.inter2} can be decomposed as
\begin{equation}\label{e.inter4}
\xymatrix{R\cHom_B(b_1^*L_1,b_2^!L_2)\ar[d]_{\adj_{f_2}}\ar[r]^-\sim
&b^!R\cHom_X(p_{1,X}^*L_1,p_{2,X}^!L_2)\ar[d]^{\adj_{f_2}}
&b^!(DL_1\boxtimes L_2)\ar[d]^{\adj_f}\ar[l]\\
R\cHom_C(c_1^*L_1,c_2^!f_{2!}L_2)\ar[r]^-\sim\ar[rd]_\simeq
& b^!R\cHom_X(p_{1,X}^*L_1,f^!p_{2,Y}^!f_{2!}L_2)\ar[d]^\simeq
& b^!f^!f_!(DL_1\boxtimes L_2)\ar[d]^{\KF_!}_\simeq\\
& c^!R\cHom_Y(p_{1,Y}^*L_1,p_{2,Y}^!f_{2!}L_2)
& c^!(DL_1\boxtimes f_{2!}L_2).\ar[l]}
\end{equation}
The triangle and the upper left square of \eqref{e.inter4} trivially
commute. The right square of \eqref{e.inter4} can be further decomposed as
\[\xymatrix{&R\cHom_X(p_{1,X}^*L_1,p_{2,X}^!L_2)\ar[dl]_{\adj_{f_2}}
&DL_1\boxtimes L_2\ar[d]^{\adj_f}\ar[ld]^{\adj_{f_2}}\ar[l]\\
R\cHom_X(p_{1,X}^*L_1,f^!p_{2,Y}^!f_{2!}L_2)\ar[d]_\simeq
& DL_1\boxtimes f_2^!f_{2!}L_2\ar[d]^{\KF^!}\ar[l]
& f^!f_!(DL_1\boxtimes L_2)\ar[ld]^{\KF_!}_\simeq\\
f^!R\cHom_Y(p_{1,Y}^*L_1,p_{2,Y}^!f_{2!}L_2)
& f^!(DL_1\boxtimes f_{2!}L_2),\ar[l]}
\]
where the lower-left square commutes by Proposition \ref{p.KFadj} and the
other inner cells trivially commute.
\end{proof}

\section{Lefschetz-Verdier formula}\label{s.LV}
In this section, we state and prove a Lefschetz-Verdier formula for
Deligne-Mumford stacks (Theorem \ref{t.lv}). Throughout this section, we fix
$S$, the spectrum of a field, and a ring $\Lambda$ of one of the following
types:
\begin{itemize}
\item[(a)] A Noetherian ring annihilated by an integer $m$ invertible on
    $S$;
\item[(b)] A complete discrete valuation ring of residue characteristic
    $\ell$ invertible on $S$;
\item[(c)] An algebraic extension of a discrete valuation field of
    characteristic $0$ and residue characteristic invertible on $S$.
\end{itemize}
We adopt Convention \ref{c.Sstack} for $S$-stacks. For an $S$-stack $X$, we
consider the category $\Dcft(X,\Lambda)\subset D^-_c(X,\Lambda)$ in case
(a), and the category $\Dbc(X,\Lambda)$ in cases (b) and (c). In order to
get uniform statements for all cases, we make the following definition in
cases (b) and (c): $\Dcft(X,\Lambda):=\Dbc(X,\Lambda)$. We use the same
simplified notation as in the previous section. Note that $K_X=\Omega_X$ (we
endow $S$ with the dimension function $0$) so that $D_X$ in Notation
\ref{n.K} coincides with the duality functors in Sections \ref{s.3},
\ref{s.5}, \ref{s.6}.

We start by establishing a few K\"unneth formulas. Proposition \ref{p.KFls}
through Remark \ref{r.PP} below are stated in case (a). In cases (b) and
(c), one should replace $D^-_c$ by $D^b_c$.

\begin{Proposition}\label{p.KFls}
In the situation of Construction \ref{c.KFstar}, if $\cS=\cT=S$, $s=\id$,
and $M_i\in D^-_c(X_i,\Lambda)$, $i=1,2$, then the morphism $\KF_*$ is an
isomorphism.
\end{Proposition}

\begin{proof}
We easily reduce to case (a). We may assume $M_i\in D^b_c$. Using the
distinguished triangle \eqref{e.dist2}, we reduce by induction to the case
$M_i=Rj_{i*}L_i$, where $j_i\colon U_i\to X_i$ is an immersion, $L_i\in
D^b_c(U_i,\Lambda)$, and the cohomology sheaves of $L_i$ are locally
constant. Note that the assertion for $(f_i,Rj_{i*} L_i)$ follows from the
assertions for $(j_i, L_i)$ and $(f_i j_i,L_i)$. Up to replacing $M_i$ by
$L_i$, we may thus assume that the cohomology sheaves of $M_i$ are locally
constant. We may further assume $M_i\in \Mod_c(X_i,\Lambda)$. As
$(f_1,f_2)=(f_1,\id)(\id, f_2)$, we may assume that $f_1$ or $f_2$ is
identity. By symmetry we may assume $X_2=Y_2$, $f_2=\id$. As the problem is
local on $Y_2$, we may assume $M_2$ constant. Then $\KF_*$ is the composite
\[q_1^*f_{1*}M_1\otimes q_2^*M_2 \xrightarrow{\BC} f_* p_1^* M_1\otimes q_2^* M_2\to f_*(p_1^* M_1\otimes p_2^* M_2),\]
where the first morphism is an isomorphism by generic base change
(Proposition \ref{p.gbc}) and the second morphism is an isomorphism by
projection formula (Proposition \ref{p.cdim} \ref{p.cdim1}).
\end{proof}

\begin{Proposition}\label{p.KF}
In the situation of Construction \ref{c.KFH}, if $\cS=S$, $L_i\in
\Dcft(X_i,\Lambda)$, $M_i\in D^-_c(X_i,\Lambda)$, $i=1,2$, then the morphism
$\KF$  \eqref{e.KF} is an isomorphism.
\end{Proposition}

\begin{proof}
We reduce easily to case (a) for separated schemes, which is \cite[III
Proposition 2.3]{SGA5}. We may also imitate the proof of \cite[III
Proposition 2.3]{SGA5} to reduce to K\"unneth formula for $f_*$ (Proposition
\ref{p.KFls}).
\end{proof}

\begin{Proposition}\label{p.KFus}
In the situation of Construction \ref{c.KFls}, if $\cS=S$, $N_i\in
D^-_c(Y_i,\Lambda)$, $i=1,2$, then the morphism $\KF^!$ is an isomorphism.
\end{Proposition}

\begin{proof}
Again we reduce easily to case (a) for separated schemes, which is \cite[III
Proposition 1.7.4]{SGA5}. As before, we may also imitate the proof of
\cite[III Proposition 1.7.4]{SGA5} to reduce to K\"unneth formula for $f_*$
(Proposition \ref{p.KFls}).
\end{proof}

\begin{Remark}\label{r.PP}
In the situation of Remark \ref{r.P}, \eqref{e.P2} is an isomorphism for
$L_1\in \Dcft(X_1,\Lambda)$, $L_2\in D^-_c(X_2,\Lambda)$ by Propositions
\ref{p.KF} and \ref{p.KFus}. In this case, composing \eqref{e.P1} with
\eqref{e.P2}, we obtain an isomorphism $b^!(D_{X_1}L_1\boxtimes L_2)\simeq
R\cHom_B(b_1^*L_1,b_2^! L_2)$.
\end{Remark}

\begin{Construction}[Pairing]\label{c.pair}
Let $b'=(b'_1,b'_2)\colon B'\to X_1\times X_2$, $b''=(b''_2,b''_1)\colon
B''\to X_2\times X_1$ be correspondences, and let $L_i\in
\Dcft(X_i,\Lambda)$, $i=1,2$. We form a 2-Cartesian square
\[\xymatrix{B\ar[r]\ar[d] & B''\ar[d]^{b''}\\ B'\ar[r]^{b'} & X_1\times X_2}\]
and define
\[\langle-,-\rangle\colon R\cHom_{B'}(b'^*_1L_1,b'^!_2 L_2)\boxtimes_{X_1\times X_2} R\cHom_{B''}(b''^*_2L_2,b''^!_1L_1) \to K_B\]
to be the composite
\begin{multline*}
R\cHom_{B'}(b'^*_1L_1,b'^!_2 L_2)\boxtimes_{X_1\times X_2} R\cHom_{B''}(b''^*_2L_2,b''^!_1L_1)
\xrightarrow[\sim]{\alpha} b'^! (D_{X_1}L_1\boxtimes L_2)\boxtimes_{X_1\times X_2} b''^!(D_{X_2}L_2\boxtimes L_1)\\
\xrightarrow{\KF^!} b^!((D_{X_1}L_1\boxtimes L_2)\otimes(D_{X_2}L_2\boxtimes L_1))
\simeq b^!((D_{X_1}L_1\otimes L_1)\boxtimes(D_{X_2}L_2\otimes L_2))\\
\xrightarrow{\ev\boxtimes \ev} b^!(K_{X_1}\boxtimes K_{X_2}) \xrightarrow[\sim]{\KF^!} b^!K_{X_1\times X_2}\simeq K_B,
\end{multline*}
where $\alpha$ is given by Remark \ref{r.PP}. The map $\langle -,-\rangle$
is clearly symmetric. It induces a map (similar to \cite[III
(4.2.4)$'$]{SGA5})
\[\langle-,-\rangle\colon b'_*R\cHom_{B'}(b'^*_1L_1,b'^!_2 L_2)\otimes b''_* R\cHom_{B''}(b''^*_2L_2,b''^!_1L_1) \to b_*K_B,\]
composite of
\begin{multline*}
b'_*R\cHom_{B'}(b'^*_1L_1,b'^!_2 L_2)\otimes b''_* R\cHom_{B''}(b''^*_2L_2,b''^!_1L_1)\\
\xrightarrow{\KF_*} b_*(R\cHom_{B'}(b'^*_1L_1,b'^!_2 L_2)\boxtimes_{X_1\times X_2} R\cHom_{B''}(b''^*_2L_2,b''^!_1L_1))
\xrightarrow{b_*\langle-,-\rangle} b_* K_B.
\end{multline*}
This further induces a map
\[\langle-,-\rangle\colon \Hom_{B'}(b'^*_1L_1,b'^!_2 L_2)\otimes \Hom_{B''}(b''^*_2L_2,b''^!_1L_1) \to H^0(B,K_B),\]
composite of
\begin{multline*}
\Hom_{B'}(b'^*_1L_1,b'^!_2 L_2)\otimes \Hom_{B''}(b''^*_2L_2,b''^!_1L_1)\\
\xrightarrow{\KF_{H^0}}
H^0(X_1\times X_2,b'_*R\cHom_{B'}(b'^*_1L_1,b'^!_2 L_2)\otimes b''_* R\cHom_{B''}(b''^*_2L_2,b''^!_1L_1)) \xrightarrow{H^0(X_1\times X_2,\langle -,-\rangle)} H^0(B,K_B).
\end{multline*}
\end{Construction}

\begin{Remark}\label{r.dualpair}
The pairing defined in Construction \ref{c.pair} is compatible with duality
(Construction \ref{c.dual}). More precisely, for $b'$, $b''$ and $L_i\in
\Dcft(X_i,\Lambda)$, $i=1,2$ as above, the diagram
\[\xymatrix{R\cHom_{B'}(b'^*_1L_1,b'^!_2 L_2)\boxtimes_{X_1\times X_2} R\cHom_{B''}(b''^*_2L_2,b''^!_1L_1) \ar[rd]^-{\langle-,-\rangle}\ar[d]_{D_{b'}\boxtimes_{X_1\times X_2} D_{b''}}\\
R\cHom_{B'}(b'^*_2D_{X_2}L_2,b'^!_1 D_{X_1}L_1)\boxtimes_{X_1\times X_2} R\cHom_{B''}(b''^*_1D_{X_1}L_1,b''^!_2D_{X_2}L_2) \ar[r]_-{\langle-,-\rangle} & K_B}
\]
commutes by Proposition \ref{p.interdual}. In particular, for $u\colon
b'^*_1L_1\to b'^!_2 L_2$, $v\colon b''^*_2L_2\to b''^!_1 L_1$, we have the
identity $\langle u,v\rangle =\langle D_{b'}u,D_{b''}v\rangle$ in
$H^0(B,K_B)$. In the case of schemes, such an identity is stated in
\cite[III (5.1.6)]{SGA5}.
\end{Remark}

\begin{Remark}
The pairing defined in Construction \ref{c.pair} is related to
noncommutative traces as follows. Recall that for a finite group $G$, the
inertia stack $I_{BG}:=BG\times_{\Delta_{BG},BG\times BG,\Delta_{BG}} BG$ of
the classifying stack $BG$ can be identified with $[G/G]$, where $G$ acts by
conjugation. Let $g\colon I_{BG}\to G^\natural$ be the coarse moduli
morphism, where $G^\natural$ is the space of conjugacy classes of $G$. Let
$C$ and $Y$ be (separated finite-type) $S$-schemes and let $c\colon C\to
Y\times Y$ be a correspondence. If in Construction \ref{c.pair},
$X_1=X_2=Y\times BG$ (in particular $G$ has order prime to $m$ in case (a)
and prime to $\ell$ in case (b) by Convention \ref{c.Sstack}), and
\[b'=c\times \Delta_{BG}\colon C\times BG \to (Y\times BG) \times (Y\times BG), \quad
b''=\Delta_{Y\times BG}\colon Y\times BG \to (Y\times BG) \times (Y\times BG),
\]
then $B$ can be identified with $Y^c\times I_{BG}$, where
$Y^c=C\times_{c,Y\times Y, \Delta_Y} Y$. For $L\in \Dcft(Y\times
BG,\Lambda)$, the composite
\[\Hom_{C\times BG}((c_1\times \id_{BG})^* L,(c_2\times \id_{BG})^! L)\xrightarrow{\langle -,\id_L\rangle} H^0(Y^c\times I_{BG}, K_{Y^c\times I_{BG}}) \xrightarrow{(\id_{Y^c}\times g)_!} H^0(Y^c\times G^\natural, K_{Y^c\times G^\natural})\]
can be identified with the local trace over the not necessarily commutative
ring $\Lambda[G]$ defined in \cite[III~B (6.10.1)]{SGA5}, by \cite[III~B
(6.10.2)]{SGA5}. The noncommutative local trace is a crucial ingredient in
the divisibility results of \cite[III~B Section 6]{SGA5} (see also
\cite[Rapport, Section 5]{SGA4d}).
\end{Remark}

The following Lefschetz-Verdier formula expresses the compatibility of
pairing (Construction \ref{c.pair}) with extraordinary pushforward
(Construction \ref{c.shriek}).

\begin{Theorem}\label{t.lv}
Let $f_i\colon X_i\to Y_i$ be morphisms of $S$-stacks (see Convention
\ref{c.Sstack}), $i=1,2$ and let $f=f_1\times f_2\colon X=X_1\times X_2\to
Y_1\times Y_2=Y$. Let
\begin{equation}\label{e.lvcube}
\xymatrix{&B''\ar[ld]_{b''}\ar[dd]|\hole^(.3){g''} && B\ar[ld]\ar[dd]^g\ar[ll]\\
X\ar[dd]_{f} &&\ar[ll]_(.7){b'} B'\ar[dd]^(.3){g'} \\
& C''\ar[ld]_{c''} && C\ar[ld]\ar[ll]|\hole\\
Y && C'\ar[ll]_{c'}}
\end{equation}
be a 2-commutative diagram of $S$-stacks with 2-Cartesian top and bottom
squares. We assume that $(g',g'')^\sharp$ is left-decomposable (Definition
\ref{d.decomp}). Let $L_i\in \Dcft(X_i,\Lambda)$, $i=1,2$. Then $g$ is
proper and the following diagram commutes
\begin{equation}\label{e.lv}
\xymatrix{g_*(R\cHom_{B'}(b'^*_1L_1,b'^!_2 L_2)\boxtimes_X R\cHom_{B''}(b''^*_2L_2,b''^!_1L_1))\ar[r]^-{\langle-,-\rangle}
& g_*K_B\ar[dd]^{\adj}\\
g'_{*}R\cHom_{B'}(b'^*_1L_1,b'^!_2 L_2)\boxtimes_Y g''_{*}R\cHom_{B''}(b''^*_2L_2,b''^!_1L_1)\ar[u]^{\KF_*}\ar[d]_{g'^\sharp_!\boxtimes_Y g''^\sharp_!} \\
R\cHom_{C'}(c'^*_1f_{1!}L_1,c'^!_2f_{2!}L_2) \boxtimes_Y R\cHom_{C''}(c''^*_2f_{2!}L_2,c''^!_1f_{1!}L_1)\ar[r]^-{\langle-,-\rangle}
& K_C.}
\end{equation}
\end{Theorem}

The case of schemes of the following was proven by Illusie \cite[III
Th\'eor\`eme 4.4, Corollaire 4.5]{SGA5} under the additional assumption that
$f_1,f_2,g',g'',b',b'',c',c''$ are proper.

\begin{Corollary}\label{c.lv}
In the situation of Theorem \ref{t.lv}, the diagram
\begin{equation}\label{e.lv1}
\xymatrix{f_*(b'_*R\cHom_{B'}(b'^*_1L_1,b'^!_2 L_2)\otimes b''_* R\cHom_{B''}(b''^*_2L_2,b''^!_1L_1))\ar[r]^-{\langle-,-\rangle}& f_*b_*K_B\ar[dd]^{\adj}\\
f_{*}b'_*R\cHom_{B'}(b'^*_1L_1,b'^!_2 L_2)\otimes f_{*}b''_* R\cHom_{B''}(b''^*_2L_2,b''^!_1L_1)\ar[u]^{\KF_*}\ar[d]_{g'^\sharp_!\otimes g''^\sharp_!} \\
c'_*R\cHom_{C'}(c'^*_1f_{1!}L_1,c'^!_2f_{2!}L_2) \otimes c''_*R\cHom_{C''}(c''^*_2f_{2!}L_2,c''^!_1f_{1!}L_1)\ar[r]^-{\langle-,-\rangle}
& c_*K_C,}
\end{equation}
where $b\colon B\to X$, $c\colon C\to Y$, commutes. In particular, the
diagram
\begin{equation}\label{e.lv2}
\xymatrix{\Hom_{B'}(b'^*_1L_1,b'^!_2 L_2)\otimes \Hom_{B''}(b''^*_2L_2,b''^!_1L_1)\ar[r]^-{\langle-,-\rangle}\ar[d]_{g'^\sharp_!\otimes g''^\sharp_!}
& H^0(B,K_B)\ar[d]^{g_!}\\
\Hom_{C'}(c'^*_1f_{1!}L_1,c'^!_2f_{2!}L_2) \otimes \Hom_{C''}(c''^*_2f_{2!}L_2,c''^!_1f_{1!}L_1)\ar[r]^-{\langle-,-\rangle}
& H^0(C,K_C)}
\end{equation}
commutes. In other words, for $u\colon b'^*_1L_1\to b'^!_2 L_2$ and $v\colon
b''^*_2L_2\to b''^!_1L_1$,
\[g_!\langle u,v\rangle=\langle g'^\sharp_!
u,g''^\sharp_! v\rangle.\]
\end{Corollary}

\begin{proof}
The diagram \eqref{e.lv1} can be decomposed as\tiny
\[\xymatrix@C=1em{f_*(b'_*R\cHom_{B'}(b'^*_1L_1,b'^!_2 L_2)\otimes b''_* R\cHom_{B''}(b''^*_2L_2,b''^!_1L_1))\ar[r]^{\KF_*}
&f_*b_*(R\cHom_{B'}(b'^*_1L_1,b'^!_2 L_2)\boxtimes_X R\cHom_{B''}(b''^*_2L_2,b''^!_1L_1))\ar[r]^-{\langle-,-\rangle}
& f_*b_*K_B\ar[dd]^{\adj}\\
f_{*}b'_*R\cHom_{B'}(b'^*_1L_1,b'^!_2 L_2)\otimes f_{*}b''_* R\cHom_{B''}(b''^*_2L_2,b''^!_1L_1)\ar[u]^{\KF_*}\ar[d]_{g'^\sharp_!\otimes g''^\sharp_!}\ar[r]^{\KF_*}
&c_*(g'_{*}R\cHom_{B'}(b'^*_1L_1,b'^!_2 L_2)\boxtimes_Y g''_{*}R\cHom_{B''}(b''^*_2L_2,b''^!_1L_1))\ar[u]^{\KF_*}\ar[d]_{g'^\sharp_!\otimes g''^\sharp_!}\\
c'_*R\cHom_{C'}(c'^*_1f_{1!}L_1,c'^!_2f_{2!}L_2) \otimes c''_*R\cHom_{C''}(c''^*_2f_{2!}L_2,c''^!_1f_{1!}L_1)\ar[r]^{\KF_*}
&c_*(c'_*R\cHom_{C'}(c'^*_1f_{1!}L_1,c'^!_2f_{2!}L_2) \otimes c''_*R\cHom_{C''}(c''^*_2f_{2!}L_2,c''^!_1f_{1!}L_1))\ar[r]^-{\langle-,-\rangle}
& c_*K_C,}
\]
\normalsize where the left squares trivially commute and the right square is
obtained by applying $c_*$ to \eqref{e.lv}. The diagram \eqref{e.lv2} can be
decomposed as\tiny
\[\xymatrix@C=1em{
&H^0(X,b'_*R\cHom_{B'}(b'^*_1L_1,b'^!_2 L_2)\otimes b''_* R\cHom_{B''}(b''^*_2L_2,b''^!_1L_1))\ar[r]^-{\langle-,-\rangle}
&H^0(B,K_B)\ar[dd]^{g_!}\\
\Hom_{B'}(b'^*_1L_1,b'^!_2 L_2)\otimes
\Hom_{B''}(b''^*_2L_2,b''^!_1L_1)\ar[d]_{g'^\sharp_!\otimes
g''^\sharp_!}\ar[ru]^-{\KF_{H^0}}\ar[r]_-{\KF_{H^0}}
&H^0(Y,f_{*}b'_*R\cHom_{B'}(b'^*_1L_1,b'^!_2 L_2)\otimes f_{*}b''_* R\cHom_{B''}(b''^*_2L_2,b''^!_1L_1))\ar[u]^{\KF_*}\ar[d]_{g'^\sharp_!\otimes
g''^\sharp_!}\\
\Hom_{C'}(c'^*_1f_{1!}L_1,c'^!_2f_{2!}L_2) \otimes \Hom_{C''}(c''^*_2f_{2!}L_2,c''^!_1f_{1!}L_1)\ar[r]^-{\KF_{H^0}}
&H^0(Y,c'_*R\cHom_{C'}(c'^*_1f_{1!}L_1,c'^!_2f_{2!}L_2) \otimes c''_*R\cHom_{C''}(c''^*_2f_{2!}L_2,c''^!_1f_{1!}L_1))\ar[r]^-{\langle-,-\rangle}
& H^0(C,K_C),}
\]
\normalsize where the triangle and the lower left square trivially commute
and the right square is obtained by applying $H^0(Y,-)$ to \eqref{e.lv1}.
\end{proof}

Applying Corollary \ref{c.lv} to the case where $f_1=f_2=g''$, $b'', c''$
are the diagonal morphisms, and $L_1=L_2$, and using the fact that
$g''^\sharp_! \id=\id$, we obtain the following corollary. The case of
algebraic spaces was shown by Varshavsky \cite[Proposition
1.2.5]{Varshavsky} under the additional assumption that $f$ is proper. Note
that $X^b$ is not a substack of $B$ in general, unlike the case of algebraic
spaces.

\begin{Corollary}
Let
\[\xymatrix{X\ar[d]_f & B\ar[d]^g\ar[l]_{b_1}\ar[r]^{b_2} & X\ar[d]^f\\
Y &\ar[l]_{a_1}\ar[r]^{a_2} C & Y}
\]
be a morphism of correspondences such that $g$ is proper. Let $L\in
\Dcft(X,\Lambda)$. Then
\[h\colon X^b:=B\times_{b,X\times X, \Delta_X} X\to
C\times_{c,Y\times Y,\Delta_Y} Y=:Y^c
\]
is proper and the following diagram
commutes
\[\xymatrix{\Hom_{B}(b^*_1L,b^!_2 L)\ar[r]^-{\langle-,\id_{L}\rangle}\ar[d]_{g^\sharp_!}
& H^0(X^b,K_{X^b})\ar[d]^{h_!}\\
\Hom_{C}(c^*_1f_{!}L,c^!_2f_{!}L) \ar[r]^-{\langle-,\id_{f_! L}\rangle}
& H^0(Y^c,K_{Y^c}).}
\]
\end{Corollary}

Applying duality to Theorem \ref{t.lv} and Corollary \ref{c.lv}, we obtain
the following compatibility of pairing (Construction \ref{c.pair}) with
pushforward (Construction \ref{c.star}), by Proposition \ref{r.dualpush} and
Remark \ref{r.dualpair}.

\begin{Corollary}
Let $f_i\colon X_i\to Y_i$ be morphisms of $S$-stacks, $i=1,2$ and let
$f=f_1\times f_2\colon X=X_1\times X_2\to Y_1\times Y_2=Y$. Consider a
2-commutative diagram of $S$-stacks of the form \eqref{e.lvcube} with
2-Cartesian top and bottom squares such that $(g',g'')^\sharp$ is
right-decomposable. Let $L_i\in \Dcft(X_i,\Lambda)$, $i=1,2$. Then $g$ is
proper and the following diagram commutes
\[\xymatrix{g_*(R\cHom_{B'}(b'^*_1L_1,b'^!_2 L_2)\boxtimes_X R\cHom_{B''}(b''^*_2L_2,b''^!_1L_1))\ar[r]^-{\langle-,-\rangle}
& g_*K_B\ar[dd]^{\adj}\\
g'_{*}R\cHom_{B'}(b'^*_1L_1,b'^!_2 L_2)\boxtimes_Y g''_{*}R\cHom_{B''}(b''^*_2L_2,b''^!_1L_1)\ar[u]^{\KF_*}\ar[d]_{g'^\sharp_*\boxtimes_Y g''^\sharp_*} \\
R\cHom_{C'}(c'^*_1f_{1*}L_1,c'^!_2f_{2*}L_2) \boxtimes_Y R\cHom_{C''}(c''^*_2f_{2*}L_2,c''^!_1f_{1*}L_1)\ar[r]^-{\langle-,-\rangle}
& K_C.}
\]
In particular, for $u\colon b'^*_1L_1\to b'^!_2 L_2$ and $v\colon
b''^*_2L_2\to b''^!_1L_1$,
\[g_!\langle u,v\rangle=\langle g'^\sharp_*
u,g''^\sharp_* v\rangle
\]
in $H^0(C,K_C)$.
\end{Corollary}

A related formula is Fujiwara's theorem, previously conjectured by Deligne,
which expresses the trace of a cohomological correspondence, after
sufficient twisting by Frobenius, as the sum of naive local terms. Olsson
extended this trace formula to Deligne-Mumford stacks \cite[Theorem
12.3]{OlssonFuji} by reducing the problem to the coarse moduli spaces. By
contrast, it is unclear how to reduce the Lefschetz-Verdier formula for
Deligne-Mumford stacks to the case of algebraic spaces. Instead, we will
give a direct proof of Theorem \ref{t.lv}.

As mentioned earlier, a Lefschetz-Verdier formula for schemes was proven by
Illusie and later another one was proven by Varshavsky. Although
Varshavsky's statement is closer to ours in the sense that it makes fewer
additional properness assumptions, his proof makes essential use of
compactifications of correspondences. It is unclear how to extend
Varshavsky's method to Deligne-Mumford stacks. Our proof of Theorem
\ref{t.lv} is closer in spirit to Illusie's original proof.

\begin{proof}[Proof of Theorem \ref{t.lv}]
We will use the notation $b\colon B\to X$, $c\colon C\to Y$, $P=D_{X_1}
L_1\boxtimes L_2$, $Q=D_{X_2}L_2\boxtimes L_1$, $E=D_{Y_1}f_{1!}L_1\boxtimes
f_{2!}L_2$, $F=D_{Y_2}f_{2!}L_2\boxtimes f_{1!}L_1$.

By compatibility of the diagram with composition, we may assume that
$(g',g'')^\sharp$ is of type (I) or (II). By symmetry, it suffices to treat
case (I). Recall that this is the case where $g'^\sharp$ is of type (i) and
$g''^\sharp$ is of type (ii). In other words, $f_2=\id$, $g''=\id$, and
$B'\to X_1\times_{Y_1} C'$ is proper. Thus we can decompose
$(g',g'')^\sharp$ as follows
\[\xymatrix{X_1\ar@{=}[d] & B'\ar[l]\ar[r]\ar[d] & X_2\ar@{=}[d] & B''\ar@{=}[d]\ar[l]\ar[r] & X_1\ar@{=}[d]\\
X_1\ar[d]_{f_1} & X_1\times_{Y_1} C'\ar[l]\ar[r]\ar[d] & X_2\ar@{=}[d] & B''\ar@{=}[d]\ar[l]\ar[r] & X_1\ar[d]^{f_1}\\
Y_1 & C'\ar[l]\ar[r] & X_2 & B''\ar[l]\ar[r] & Y_1.}
\]
Therefore it suffices to treat two special cases of type (I):
\begin{enumerate}
\item $f_1=\id$;
\item The morphism $B'\to X_1\times_{Y_1} C'$ is an equivalence.
\end{enumerate}

For (1), we will show more generally that the theorem holds for
$(g',g'')^\sharp$, not necessarily of type (I), whenever $f_i=\id$, $i=1,2$
and $g', g''$ are proper. Indeed, in this case, $g$ is clearly proper, and,
by Proposition \ref{p.inter}, the diagram \eqref{e.lv} can be decomposed as
\[\xymatrix{g_*(b'^!P\boxtimes_X b''^! Q)\ar[r]^{\KF^!}\ar[d]_{\KF_!}^\simeq
& g_*g^!(c'^!P\boxtimes_X c''^! Q)\ar[r]^{\KF^!}\ar[d]^{\adj_g}
& g_*b^!(P\otimes Q)\ar[d]^{\adj_g}\ar[r]^{\langle-,-\rangle} & g_* c^! K_X\ar[d]^{\adj_g}\\
g'_*b'^!P\boxtimes_X g''_*c''^!Q \ar[r]^-{\adj_{g'}\boxtimes_X \adj_{g''}}
& c'^!P\boxtimes_X c''^!Q\ar[r]^{\KF^!}
&c^!(P\otimes Q)\ar[r]^{\langle-,-\rangle} & c^! K_X,}
\]
where all inner squares commute.

In case (2), the front square of \eqref{e.lvcube} is 2-Cartesian. It follows
that the back square of \eqref{e.lvcube} is 2-Cartesian as well and $g$ is
an equivalence. We may assume that $g=\id$, so that \eqref{e.lvcube} is of
the form
\[\xymatrix{&B''\ar[ld]_{b''}\ar@{=}[dd]|\hole && B\ar[ld]^{\bar b}\ar@{=}[dd]\ar[ll]_{d}\\
X\ar[dd]_{f} &&\ar[ll]_(.7){b'} B'\ar[dd]^(.3){g'} \\
& B''\ar[ld]_{c''} && B\ar[ld]^{\bar c}\ar[ll]|\hole_(.7){d}\\
Y && C'\ar[ll]_{c'}}
\]
Then, by Proposition \ref{p.inter}, the diagram \eqref{e.lv} can be
decomposed as
\[\xymatrix{g'_*b'^! P\boxtimes_Y b''^!Q\ar[d]_{\BC\boxtimes_Y
\adj_f}\ar[r]^{\KF_*}\ar@{}[rrd]|{\text{(A)}}
& b'^!P\boxtimes_X b''^!Q\ar[r]^{\KF^!}
& b^!(P\otimes Q)\ar[d]^{\adj_f}\ar[r]^{\langle -,- \rangle}
& b^! K_X\ar[d]^{\adj_f}\\
c'^!f_*P\boxtimes_Y c''^!f_!Q\ar[r]^{\KF^!}\ar[d]_\simeq
& c^!(f_*P\otimes f_! Q)\ar[r]\ar@{}[rrd]|{\text{(B)}}\ar[d]^\simeq
& c^!f_!(P\otimes Q)\ar[r]^{\langle-,-\rangle}
& c^!f_!K_X\ar[d]^{\adj_f}\\
c'^!E\boxtimes_Y c''^!F\ar[r]^{\KF^!} & c^!(E\otimes F)\ar[rr]^{\langle -,-\rangle}
&& c^! K_Y,}
\]
where the unmarked arrow is induced by the composite
\[f_*P\otimes f_!Q\xrightarrow{\PF} f_!(f^*f_*P\otimes Q) \xrightarrow{\adj} f_!(P\otimes Q).\]
The upper-right and lower-left squares of the above diagram clearly commute.
The square (A) can be decomposed as
\[\xymatrix{g'_*b'^!P\boxtimes_Y b''^!Q\ar[r]^{\KF_*}\ar[d]^\simeq_{\BC}\ar@{}[rd]|{\text{(A1)}}
& b'^!P\boxtimes_X b''^!Q\ar[r]^{\KF^!}
&d^!(P\boxtimes_X  b''^!Q)\ar[r]^{\KF^!}
&b^!(P\otimes Q)\ar[d]^{\adj_f}\\
c'^!f_*P\boxtimes_Y b''^!Q\ar[r]^{\KF^!}\ar[d]_{\adj_f}
& d^!(f_*P\boxtimes_Y b''^!Q)\ar[ru]^{\adj_f}\ar[r]_{\KF^!}\ar[d]^{\adj_f}
& b^!(f_*P\boxtimes_Y Q)\ar[ru]_{\adj_f}\ar[d]^{\adj_f}\ar@{}[rd]|{\text{(A2)}}
& c^! f_!(P\otimes Q)\\
c'^!f_*P\boxtimes_Y c''^!f_!Q\ar[r]^{\KF^!}
& d^!(f_*P\boxtimes_Y c''^!f_!Q)\ar[r]^{\KF^!}
& b^!(f_*P\boxtimes_Y f^!f_! Q)\ar[r]^{\KF^!}
& c^!(f_*P\otimes f_! Q).\ar[u]}
\]
Here (A1) corresponds by adjunction to the commutative diagram
\[\xymatrix{d_!(\bar c^*c'^!f_*P\otimes d^* R)\ar[d]_{\PF}^\simeq
&d_!(\bar c^* g'_*b'^!P\otimes d^*R)\ar[l]_{\BC}^\sim\ar[r]^{\adj_{g'}}\ar[d]^{\PF}_\simeq
&d_!(\bar b^* b'^!P\otimes d^*R)\ar[d]_\simeq^{\PF}\\
d_!\bar c^* c'^!f_*P\otimes R\ar[d]_{\BC}^\simeq
& d_!\bar c^*g'_*b'^!P\otimes R\ar[r]^{\adj_{g'}}\ar[l]_{\BC}^\sim\ar[d]^{\BC}_{\simeq}
& d_!\bar b^*b'^!P\otimes d^*R\ar[d]^{\BC}_\simeq\\
b''^*b'_!g'^*c'^!f_*P\otimes R\ar[d]_{\BC}^\simeq
& b''^*b'_!g'^*g'_*b'^!P\otimes R\ar[r]^{\adj_{g'}}\ar[l]_{\BC}^\sim \ar[d]^{\BC}_{\simeq}
& b''^*b'_!b'^!P\otimes R\ar[dd]^{\adj_{c'}}\\
b''^*f^*c'_!c'^!f_*P\otimes R\ar[d]_{\adj_{d'}}
&b''^*f^*c'_!g'_*b'^!P\otimes R\ar[l]_{\BC}^\sim\\
b''^*f^*f_*P\otimes R\ar[rr]^{\adj_f} && b''^*P\otimes R,}
\]
where $R=c''^!Q$, and (A2) is obtained by applying $b^!$ to the commutative
diagram
\[\xymatrix{&P\otimes Q\ar[r]^{\adj} & f^!f_!(P\otimes Q)\\
f^*f_*P\otimes Q\ar[r]^{\adj}\ar[ur]^{\adj}\ar[d]_{\adj}
& f^!f_!(f^*f_*P\otimes Q)\ar[ur]^{\adj}\ar[d]^{\adj}\ar[r]_\sim^{\PF}
& f^!(f_*P\otimes f_!Q)\ar[rd]_{\id}\ar[d]_{\adj}\\
f^*f_*P\otimes f^!f_!Q\ar[r]^{\adj}
& f^!f_!(f^*f_*P\otimes f^!f_!Q)\ar[r]^{\PF}_\sim
&f^!(f_*P\otimes f_!f^!f_!Q)\ar[r]_{\adj}
&f^!(f_*P\otimes f_!Q).\ar[uul]}
\]
The diagram (B) can be identified with the exterior product of the
evaluation map $DL_2\otimes L_2\to K_{X_2}$ and the diagram
\[\xymatrix{f_{1*}DL_1\otimes f_{1!}L_1\ar[r]\ar[d]_\simeq & f_{1!}(DL_1\otimes
L_1)\ar[r]^-{\ev} & f_{1!}K_{X_1}\ar[d]^{\adj_{f_1}}\\
D(f_{1!}L_1)\otimes L_1\ar[rr]^{\ev} && K_{Y_1},}
\]
which corresponds by adjunction to the commutative diagram
\[\xymatrix{f_{1*}R\cHom_{X_1}(L_1,K_{X_1})\ar[r]^{\adj}\ar[rd]_{\id}
& f_{1*}R\cHom_{X_1}(L_1,f_1^!f_{1!}K_{X_1})\ar[r]^\sim\ar[d]^{\adj}
& R\cHom_{Y_1}(f_{1!}L_1,f_{1!}K_{X_1})\ar[d]^{\adj}\\
&f_{1*}R\cHom_{X_1}(L_1,K_{X_1})\ar[r]^{\sim} & R\cHom_{Y_1}(f_{1!}L_1,K_{Y_1}).}\]
\end{proof}

\section{Appendix: Brauer theory and Laumon's theorem}\label{s.7}
The six operations constructed in this article are used in joint work with
Illusie \cite[Section~3]{Illusie-Zheng} to give a generalization of Laumon's
theorem comparing direct image and extraordinary direct image. The
generalization works for both torsion-free coefficients and torsion
coefficients. In this appendix, we show that the case of torsion
coefficients can be deduced from the case of torsion-free coefficients by
extending results in Brauer's modular representation theory to Grothendieck
groups of categories constructed in this article. Moreover, we deduce
comparison results for other ordinary and extraordinary operations.

Let $\Lambda$ be either (a) a field of characteristic $\ell>0$ or (b) an
algebraic extension of a complete valuation field of characteristic $0$ and
residue characteristic $\ell>0$. To get uniform statements for both cases,
we let $m=\ell$ in case (a) and $m=1$ is case (b). Let $S$ be a regular
scheme of dimension $\le 1$ such that every finite-type $S$-scheme has a
dense open subscheme that is geometrically unibranch. Assume $\ell$
invertible on $S$.

\begin{Notation}\label{n.Kt}
Let $\cX$ be a Noetherian Deligne-Mumford stack. We let $K(\cX,\Lambda)$
denote the Grothendieck ring of $\Mod_c(\cX,\Lambda)$. For an object $M$ of
$\Mod_c(\cX,\Lambda)$, we let $[M]$ denote its image in $K(\cX,\Lambda)$. We
let $\Kt(\cX,\Lambda)$ denote the quotient of $K(\cX,\Lambda)$ by the ideal
generated by $[\Lambda(1)]-[\Lambda]$, where $[\Lambda]$, $[\Lambda(1)]$
denote the classes of the constant sheaf $\Lambda$ and the Tate twist
$\Lambda(1)$ in the Grothendieck group, respectively.
\end{Notation}

Let $f\colon \cX \to \cY$ be a morphism of $m$-prime inertia between
finite-type separated-diagonal Deligne-Mumford $S$-stacks. The triangulated
functors
\[Rf_*, Rf_! \colon \Dbc(\cX,\Lambda) \to \Dbc(\cY,\Lambda) \]
induce group homomorphisms
\[f_*, f_! \colon K(\cX,\Lambda) \to K(\cY, \Lambda), \]
which in turn induce group homomorphisms
\[f_*\sptilde, f_!\sptilde \colon \Kt(\cX,\Lambda) \to \Kt(\cY, \Lambda)\]
by passing to quotients.

\newcommand{\cit}{\cite[Section 3]{Illusie-Zheng}}
\begin{Theorem}[\cit]\label{t.Laumon}
We have $f_*\sptilde=f_!\sptilde$.
\end{Theorem}

Case (b) of Theorem \ref{t.Laumon} restricted to the case of schemes is a
theorem of Laumon \cite{Laumon}. In \cite[Section 3]{Illusie-Zheng}, cases
(a) and (b) are proven using similar methods. Here we show that case (a) can
be deduced from case (b). We start by reviewing some facts from modular
representation theory.

Let $\cO$ be a complete discrete valuation ring with fraction field $E$ of
characteristic $0$ and residue field $F$ of characteristic $\ell >0$, and
let $\fm$ be the maximal ideal of $\cO$. Note that for any field $F$ of
characteristic $\ell>0$, any Cohen ring of $F$ satisfies the condition
for~$\cO$. For any $\F$-module (resp.\ $\cO$-module, resp.\ $E$-module) $M$
of finite type, we endow $M$ with the unique Hausdorff topology such that
$M$ is a topological module. For a profinite group~$G$, we define a
\emph{coherent $F[G]$-module} (resp.\ \emph{$\cO[G]$-module}, resp.\
\emph{$E[G]$-module}) to be an $\F$-module (resp.\ $\cO$-module, resp.\
$E$-module) of finite type endowed with a continuous $\F$-linear (resp.\
$\cO$-linear, resp.\ $E$-linear) action of $G$, and we denote by $K(G, \F)$
(resp.\ $K(G,\cO)$, resp.\ $K(G,E)$) the Grothendieck ring of the category
of such modules. Let $i\colon \Spec \F\to \Spec \cO$, $j\colon \Spec E\to
\Spec \cO$. Then we have a group homomorphism $i_*\colon K(G,\F)\to
K(G,\cO)$ defined by restriction of scalars, and a ring homomorphism
$j^*\colon K(G, \cO)\to K(G,E)$ defined by extension of scalars $[M]\mapsto
[E\otimes_\cO M]$.

\begin{Lemma}\label{p.Brauer}
We have $i_*=0$ and $j^*$ is an isomorphism.
\end{Lemma}

\begin{proof}
We claim that the localization sequence
\[K(G,\F)\xrightarrow{i_*} K(G,\cO)\xrightarrow{j^*} K(G,E) \to 0 \]
is exact.  It is clear that $j^* i_* = 0$. We define a homomorphism $s\colon
K(G,E)\to K(G,\cO)/\Img i_*$ sending the class of a coherent $E[G]$-module
$M$ to the class of any $G$-stable $\cO$-lattice $L$ of $M$. If $L_1$ and
$L_2$ are two $G$-stable lattices of $M$, then $[L_1]-[L_2]$ belongs to
$\Img i_*$. Indeed, we may assume $L_1\subset L_2$, in which case $L_1/L_2$
is killed by a power of $\fm$, hence its class belongs to $\Img i_*$. If
$0\to M' \xrightarrow{f} M \xrightarrow{g} M''\to 0$ is a short exact
sequence of coherent $E[G]$-modules, $L$ is a $G$-stable $\cO$-lattice of
$M$, then $f^{-1}(L)$ is a lattice of $M$ and $g(L)$ is a lattice of $M''$.
Hence $s$ is a well-defined homomorphism. It is clearly an inverse of the
map $K(G,\cO)/\Img i_* \to K(G,E)$ induced by $j^*$.

Consider the decomposition map $d\colon K(G,E) \to K(G,\F)$ sending the
class of a coherent $E[G]$-module $M$ to the class of $[L/\fm L]$, where $L$
is a $G$-stable $\cO$-lattice $L$ of $M$. On shows as in the case of a
finite group \cite[Section 15.2, Th\'eor\`eme 32]{Serre} that this does not
depend on the choice of the lattice. The decomposition map is surjective. In
fact, since every coherent $F[G]$-module comes from a coherent $F[H]$-module
for some finite quotient group $H$ of $G$, the surjectivity follows from the
case of a finite group \cite[Section 16.1, Th\'eor\`eme 33]{Serre}.

By construction $i_* d=0$. It follows that $i_*=0$ and $j^*$ is an
isomorphism.
\end{proof}

Next we establish an analogue of Lemma \ref{p.Brauer} for Deligne-Mumford
stacks. Let $\cX$ be a Noetherian Deligne-Mumford stack. Notation \ref{n.Kt}
applies to $\Lambda=E,F$. We extend it to $\cO$. We let $i_*\colon
K(\cX,F)\to K(\cX,\cO)$ denote the group homomorphism induced by the exact
restriction of scalars functor $e_{0*}\colon
\Mod_c(\cX,F)\to\Mod_c(\cX,\cO)$, let $j^*\colon K(\cX,\cO)\to K(\cX,E)$
denote the ring homomorphism induced by the exact functor (Convention
\ref{c.E})
\[\Mod_c(\cX,\cO)\to \Mod_c(\cX,E), \quad M\mapsto E\otimes_\cO M,\]
and let $i^*\colon K(\cX,\cO)\to K(\cX,F)$ denote the ring homomorphism
induced by the triangulated functor (\ref{s.endthree})
\[\Dbc(\cX,\cO)\to \Dbc(\cX,\F), \quad M\mapsto F\otimes^L_\cO M.\]

\begin{Proposition}\label{p.BrauerX}\leavevmode
\begin{enumerate}
\item\label{p.BrauerX1} $i_*=0$ and $j^*$ is an isomorphism.

\item\label{p.BrauerX2} $i^*$ is a surjection.
\end{enumerate}
\end{Proposition}

\begin{proof}
\begin{itemize}
\item[\ref{p.BrauerX1}] We first show that the localization sequence
\[K(\cX,\F)\xrightarrow{i_*} K(\cX,\cO)\xrightarrow{j^*} K(\cX,E) \to 0 \]
is exact. Since $E\otimes_\cO (e_{0*}M) =0$ for all $M\in \Mod_c(\cX,F)$,
we have $j^*i_* = 0$. We define a homomorphism
\[s\colon K(\cX,E)\to K(\cX,\cO)/\Img i_*, \quad [E\otimes_\cO M] \mapsto [M].\]
Note that if $M\in \Mod_c(\cX,\cO)$ satisfies $E\otimes_\cO M= 0$ in
$\Mod_c(\cX,E)$, then its class in $K(\cX,\cO)$ belongs to $\Img i_*$. If
\[0\to E\otimes_\cO M'\xrightarrow{f} E\otimes_\cO M \xrightarrow{g} E\otimes_\cO M''\to 0\]
is a short exact sequence in $\Mod_c(\cX,E)$, then there exists an integer
$n\ge 0$, $f_\cO\colon M'\to M$ and $g_\cO\colon M\to M''$ such that
$\ell^n f=E\otimes_\cO f_\cO$ and $\ell^n g=E\otimes_\cO g_\cO$. Since
$E\otimes_\cO \Ker f_\cO$, $E\otimes_\cO (\Img f_\cO/\Img f_\cO \cap \Ker
g_\cO)$, $E\otimes_\cO (\Ker g_\cO/\Img f_\cO \cap \Ker g_\cO)$ and
$E\otimes_\cO \Coker g_\cO$ are all zero objects of $\Mod_c(\cX,E)$, it
follows that $[M']-[M]+[M'']$ belongs to $\Img i_*$. Hence $s$ is a
well-defined homomorphism. It is clearly an inverse of the map
$K(\cX,\cO)/\Img i_* \to K(\cX,E)$ induced by $j^*$.

It remains to show $i_*=0$. Note that the Abelian group $K(\cX,F)$ is
generated by elements of the form $[f_! M]$, where $f\colon \cY\to \cX$ is
an immersion, $\cY$ is integral, $M\in \Mod_c(\cY,F)$ is locally constant.
For such $f$, we have a 2-commutative diagram
\[\xymatrix{\Mod_\lisse(\cY,F)\ar[r]^{e_{0*}}\ar[d]_{f_!}& \Mod_\lisse(\cY,\cO)\ar[d]^{f_!}\\
\Mod_c(\cX,F)\ar[r]^{e_{0*}}& \Mod_c(\cX,\cO),}
\]
where $\Mod_\lisse(\cY,-)$ denote the full subcategory of $\Mod_c(\cY,-)$
spanned by locally constant sheaves, which is equivalent to $\Mod_c(G,-)$.
Here $G$ is a fundamental group of $\cY$. The above diagram induces a
commutative diagram
\[\xymatrix{K(G,F)\ar[r]^{i_*}\ar[d] & K(G,\cO)\ar[d]\\
K(\cX,F)\ar[r]^{i_*}& K(\cX,\cO).}
\]
We conclude by applying the fact that the top horizontal arrow vanishes
(Lemma \ref{p.Brauer}).

\item[\ref{p.BrauerX2}] For $f$ and $G$ as before, we have a commutative
    diagram
\[\xymatrix{K(G,\cO)\ar[r]^{j^*}_\sim \ar[d] & K(G,E)\ar[r]^d & K(G,F)\ar[d]\\
K(\cX,\cO)\ar[rr]^{i^*}&& K(\cX,F)}
\]
where the vertical maps are induced by $f_!$. It then suffices to apply
the fact that $d$ is surjective.
\end{itemize}
\end{proof}

With the help of Proposition \ref{p.BrauerX}, we can deduce case (a) of
Theorem \ref{t.Laumon} from case (b) as follows. Let $f\colon \cX \to \cY$
be a morphism of $\ell$-prime inertia between finite-type separated-diagonal
Deligne-Mumford $S$-stacks. We have a 2-commutative diagram of triangulated
categories and triangulated functors
\[\xymatrix{\Dbc(\cX,E)\ar[d]_{Rf_*} & \Dbc(\cX,\cO)\ar[l]_{E\otimes_\cO -} \ar[r]^{F\otimes^L_\cO-} \ar[d]^{Rf_*} & \Dbc(\cX,\F)\ar[d]^{Rf_*}\\
  \Dbc(\cY,E) & \Dbc(\cY,\cO)\ar[l]_{E\otimes_\cO -} \ar[r]^{F\otimes^L_\cO-} & \Dbc(\cY,\F).}
  \]
The 2-commutativity of the left square is trivial and the 2-commutativity of
the right square is Proposition \ref{p.fO}. The above diagram induces a
commutative diagram of Abelian groups and homomorphisms
\[\xymatrix{K(\cX,E)\ar[d]_{f_*} & K(\cX,\cO)\ar[l]_{j^*} \ar[r]^{i^*}\ar[d]^{f_*} & K(\cX,\F)\ar[d]^{f_*}\\
  K(\cY,E) & K(\cY,\cO)\ar[l]_{j^*} \ar[r]^{i^*} & K(\cY,\F) }\]
and a similar diagram for $K\sptilde{}$. Similar statements hold for $Rf_!$.
By Proposition \ref{p.BrauerX}, $j^*$ is an isomorphism and $i^*$ is a
surjection. It follows that $f_*\sptilde=f_!\sptilde$ for $E$ implies
$f_*\sptilde=f_!\sptilde$ for $F$.

Now we turn to consequences of Theorem \ref{t.Laumon} on other operations.
Let $f\colon \cX \to \cY$ be a morphism of between finite-type
separated-diagonal Deligne-Mumford $S$-stacks. The triangulated functors
\[f^*,Rf^!\colon \Dbc(\cY,\Lambda)\to \Dbc(\cX,\Lambda)\]
induce group homomorphisms
\begin{gather*}
f^*,f^!\colon K(\cY,\Lambda)\to K(\cX,\Lambda),\\
{f^*}\sptilde,{f^!}\sptilde\colon \Kt(\cY,\Lambda)\to \Kt(\cX,\Lambda).
\end{gather*}

\begin{Corollary}\label{c.Laumonup}
We have ${f^*}\sptilde={f^!}\sptilde$.
\end{Corollary}

\begin{proof}
We let $\equiv$ denote congruence modulo the ideal generated by
$[\Lambda(1)]-[\Lambda]$. Let $y\in K(\cY,\Lambda)$.

If $f$ is smooth of relative dimension $d$, then $f^!y=f^*y(d)$. If $f$ is a
closed immersion, then $y=j_!j^* y+f_*f^* y=j_*j^*y+f_*f^!y$, where $j$ is
the complementary open immersion. It follows then from Theorem
\ref{t.Laumon} that $f_*f^*y\equiv f_*f^!y$, so that $f^*y\equiv f^!y$. The
case of an immersion follows.

In the general case, let $(\cY_\alpha)_{\alpha\in I}$ be a partition of
$\cY$ into locally closed substack such that each $\cY_\alpha$ is the
quotient stack of an affine scheme by a finite group action. For each
$\alpha$, form the 2-Cartesian square
\[\xymatrix{\cX_\alpha\ar[d]_{f_\alpha} \ar[r]^{j'_\alpha} & \cX\ar[d]^f\\
\cY_\alpha\ar[r]^{j_\alpha} & \cY.}
\]
Then $y=\sum_{\alpha\in I}j_{\alpha!}j_\alpha^* y=\sum_{\alpha\in
I}j_{\alpha*}j_\alpha^! y$, so that
\[f^*y=\sum_{\alpha\in I}f^*j_{\alpha!}j_\alpha^* y=\sum_{\alpha\in I}
j'_{\alpha!} f_\alpha^*j_\alpha^* y,\quad f^!y=\sum_{\alpha\in
I}f^!j_{\alpha*}j_\alpha^! y=\sum_{\alpha\in I} j'_{\alpha*}
f_\alpha^!j_\alpha^! y.
\]
Thus we may assume $\cY=[Y/H]$, where $Y$ is an affine scheme endowed with
an action of a finite group $H$. Similarly, we may assume $\cX=[X/G]$, where
$X$ is an affine scheme endowed with an action of a finite group $G$. Up to
changing $X$ and $G$, we may further assume that $f=[g/\gamma]$, for
$(g,\gamma)\colon (X,G)\to (Y,H)$ (see for example \cite[Proposition
5.1]{Zind}, which extends trivially to our case). In this case $g$ can be
decomposed into $G$-equivariant morphisms $X\xrightarrow{i} Z\xrightarrow{p}
Y$ where $i$ is a closed immersion and $p$ is an affine space. Then
$f=[p/\gamma][i/\id]$, where $[i/\id]$ is a closed immersion and
$[p/\gamma]$ is smooth.
\end{proof}

Now assume that $S$ is the spectrum of a field. For a finite-type
separated-diagonal Deligne-Mumford $S$-stack $\cX$, consider the functors
\[\otimes,\otimes^!\colon \Dbc(\cX,\Lambda)\times \Dbc(\cX,\Lambda)\to \Dbc(\cX,\Lambda),\]
where $-\otimes^!-:=D_\cX(D_\cX-\otimes D_\cX-)$. Note that, if we let
$\Delta_{\cX/S}\colon \cX\to \cX\times_S \cX$ denote the diagonal morphism,
then $-\otimes-\simeq \Delta_{\cX/S}^*(-\boxtimes_S -)$, and
$-\otimes^!-\simeq R\Delta_{\cX/S}^!(-\boxtimes_S -)$ by K\"unneth formula
(Proposition \ref{p.KF} extends trivially to the non-separated case). The
functors induce bilinear maps
\begin{gather*}
\otimes,\otimes^! \colon K(\cX,\Lambda)\times K(\cX,\Lambda)\to K(\cX,\Lambda),\\
\otimes\sptilde,{\otimes^!}\sptilde \colon \Kt(\cX,\Lambda)\times \Kt(\cX,\Lambda)\to \Kt(\cX,\Lambda).
\end{gather*}
Applying Corollary \ref{c.Laumonup} to $\Delta_{\cX/S}$, we obtain the
following.

\begin{Corollary}
We have $\otimes\sptilde={\otimes^!}\sptilde$.
\end{Corollary}

\bibliographystyle{abbrv}
\bibliography{sixop}

\end{document}